\documentclass[11pt,reqno]{amsart}

\usepackage{amssymb, amsfonts, amsmath, mathrsfs, bbm, dsfont, euscript, stmaryrd, verbatim, setspace, mathtools, lineno, cite, extarrows}
\usepackage{textcomp}
\usepackage{tikz}
\usepackage{tikz-cd}

\usepackage{amsrefs}

\usepackage[colorlinks, colorlinks, linkcolor=blue, anchorcolor=blue, urlcolor=blue, citecolor=blue]{hyperref}

\allowdisplaybreaks[4]

\usepackage{relsize}

\usepackage{color}

\newcommand{\fK}{{\mathbb K} }
\newcommand{\fR}{{\mathbb R} }
\newcommand{\fC}{{\mathbb C} }

\newcommand{\poly}{{ \mathrm{poly} } }

\newcommand{\nullbracket}{ {[\cdot\, , \cdot]} }

\newcommand{\id}{ {\mathrm{id} } }

\newcommand{\cD}{  {\mathcal{D}}  }

\newcommand{\Hom}{ {\mathrm{Hom}} }

\newcommand{\End}{ {\mathrm{End}} }

\newcommand{\rZ}{ {\mathbb{Z}} }

\newcommand{\cA}{ {\mathcal{A}} }

\newcommand{\cP}{ {\mathcal{	P}} }

\newcommand{\cF}{ {\mathcal{F}} }

\newcommand{\Ker}{ {\mathrm{Ker}} }

\newcommand{\cH}{ {\mathcal{H}} }

\newcommand{\lb}{ {\bigl\{ } }
\newcommand{\rb}{ {\bigr\} } }

\newcommand{\BInfty}{  { B_{\infty} } }

\newcommand{\frakg}{ {\mathfrak{g}} }
\newcommand{\frakl}{ {\mathfrak{l}} }

\newcommand{\Gutt}{ {\mathrm{Gutt}} }
\newcommand{\PBW}{ {\mathrm{PBW}} }

\newcommand{\ADT}{ W_{(\frakg, \frakl)} }

\newcommand{\pbw}{ {\mathrm{pbw}} }

\newcommand{\ad}{ {\mathrm{ad}} }

\newcommand{\Laurent}[1]{ { (\!(#1)\!)  } }

\newcommand{\m}{\mathrm{m}}

\newcommand{\tinywedge}{{_{_\wedge}}}

\newtheorem{definition}{Definition}[section]
\newtheorem{lemma}[definition]{Lemma}
\newtheorem{proposition}[definition]{Proposition}
\newtheorem{corollary}[definition]{Corollary}
\newtheorem{theorem}[definition]{Theorem}
\newtheorem{remark}[definition]{Remark}
\newtheorem{example}[definition]{Example}

\newtheorem{Thm}{Theorem}

\numberwithin{equation}{section}

\begin{document}

\title{Brace $B_{\infty}$ algebras associated with Hopf algebroids}

\author{Jiahao Cheng}
\address{Jiahao Cheng, College of Mathematics and Information Science, Center for Mathematical Sciences, Nanchang Hangkong University, Nanchang, 330063, P.R. China}
\email{jiahaocheng@nchu.edu.cn}

\author{Zhuo Chen}
\address{Zhuo Chen, Department of Mathematics, Tsinghua University, Beijing, 100084, P.R. China}
\email{chenzhuo@mail.tsinghua.edu.cn}

\author{Yu Qiao}
\address{Yu Qiao, School of Mathematics and Statistics, Shaanxi Normal University, Xi'an, 710119, P.R. China}
\email{yqiao@snnu.edu.cn}

\date{}

\begin{abstract}  We apply the operadic modeling of brace $B_{\infty}$ algebras, as developed by Gerstenhaber and Voronov, to the context of Hopf algebroids in the sense of Xu. Specifically, we construct a strict $B_{\infty}$ isomorphism between the type I and type II twisted brace $B_{\infty}$ algebras arising from any twistor of a Hopf algebroid. As an application of this framework, we examine two specific brace $B_{\infty}$ algebras derived from Lie algebra pairs, and reveal previously obscured relationships between them. One of these is the dg Lie algebra governing deformations of algebraic dynamical twists, while the other arises from the quantum groupoid comprised of particular invariant differential operators.
\end{abstract}

\maketitle

\vspace{-0.4cm}
\begin{small}
\hspace*{0.72cm}
\textit{Key words}: brace $B_{\infty}$ algebra, Hopf algebroid, quantum groupoid, dynamical twist

\hspace*{0.72cm}
\textit{AMS subject classification (2020)}:
17B37, 16E45, 53D55, 81R50
\end{small}

\tableofcontents

\section{Introduction}

This paper applies the Gerstenhaber-Voronov operadic modeling of brace $B_{\infty}$ algebras \cite{Gerstenhaber-Voronov-Moduli} to Hopf algebroids \cites{Xu-Quantum, Xu-QuantumCNRS}. Next, we    briefly present our ideas and results in two parts. For more details about brace $B_{\infty}$ algebras and Hopf algebroids, see Section \ref{Section:Pre}.

\subsection{The canonical brace $B_{\infty}$ algebra  associated with a Hopf algebroid}

Let $A$ be an associative algebra.
On the graded space $C(A, A)$ of Hochschild cochains of   $A$, Getzler \cite{Getzler} introduced a series of higher operations called \textbf{braces} (see also Kadeishvili \cite{Kadeishvili-first}):
\begin{eqnarray*}
\mu_{k} :
C(A\, , A) \otimes \bigl(C(A\, , A)[1] \bigr)^{\otimes k}  &\rightarrow&
C(A\, , A) \, , \quad \text{for all integers $k\geqslant 0$},\\
D \otimes D_1[1] \otimes \cdots \otimes D_k[1] &\mapsto&
D\{D_1\, , \cdots\, , D_k\}\, .
\end{eqnarray*}
Gerstenhaber and Voronov \cite{Gerstenhaber-Voronov} discovered a series of identities relating these braces to both the standard cup product and the Hochschild differential on $C(A, A)$. Abstracting these algebraic structures, they defined the concept of a \textbf{homotopy $G$ algebra}, now commonly referred to as a \textbf{Gerstenhaber-Voronov algebra} or a \textbf{brace $B_{\infty}$ algebra}.

Brace $B_{\infty}$ algebras are specific examples of the broader class of $B_{\infty}$ algebras, as initially defined in \cites{Getzler-Jones, Baues}. Furthermore, Gerstenhaber and Voronov \cite{Gerstenhaber-Voronov-Moduli} provided a construction of brace $B_{\infty}$ algebras from any operad, a concept we refer to as the \textbf{operadic modeling of brace $B_{\infty}$ algebras} (see Section \ref{Subsec:GVoperadicmodeling}, particularly Theorem \ref{Operadtobrace}). The notion of brace $B_{\infty}$ algebras proved crucial in resolving the Deligne conjecture, which posited that $C(A\, , A)$ possesses an algebra structure over the little $2$-discs operad (see \cites{Kontsevich-Soibelman-DeligneConjecture, Tamarkin-DeligneConjecture, Voronov, Getzler-Jones, Gerstenhaber-Voronov, Gerstenhaber-Voronov-Moduli, McClure-Smith}).

 Hopf algebras are   significant generalizations of groups. Consequently, a variety of more general Hopf-like structures have been explored in recent literature, including weak Hopf algebras, Hopf monads, $\times_R$-Hopf algebras, and compact quantum groups. In this article, we adopt the definition of Hopf algebroids introduced by Xu \cites{Xu-Quantum, Xu-QuantumCNRS}. Specifically, a \textbf{Hopf algebroid} over an associative $\fK$-algebra $R$ (where $\fK$ is a   field of characteristic zero) consists of the data $\cH=(\cH\, , \alpha\, , \beta\, , \cdot\, , \Delta\, , \epsilon)$, where $\cH$ is an associative $\fK$-algebra. The maps $\alpha,\beta: R \rightarrow \cH$ are the source and target maps, respectively, $\Delta\colon \cH \rightarrow \cH\otimes_{R} \cH$ denotes the coproduct, and $\epsilon\colon \cH \rightarrow R$ is the counit (see Definition \ref{Def:Hopfalgebroid}).\footnote{It is worth noting that the literature contains several inequivalent definitions of a Hopf algebroid; for a discussion of these various notions, see B\"ohm \cite{Bohm}.} The notion of a \textbf{quantum groupoid} proposed by Xu \cites{Xu-Quantum, Xu-QuantumCNRS} can be seen as a specific instance of a Hopf algebroid, unifying the concepts of quantum groups and star products.

Kadeishvili's result demonstrating that the cobar construction of a bialgebra yields a brace $B_{\infty}$ algebra \cite{Kadeishvili} serves as the motivation for this work. While specialists are aware that Kadeishvili's construction of brace $B_{\infty}$ algebras extends to Hopf algebroids, we explicitly state this generalization as a foundational result for the present paper. We call this construction \emph{the canonical brace $B_{\infty}$ algebra associated with a Hopf algebroid:}

\begin{Thm}[See Theorem \ref{MainTheorembrace}]\label{MainThmIntroduction1} Given a Hopf algebroid $\cH=(\cH\, , \alpha\, , \beta\, , \cdot\, , \Delta\, , \epsilon)$, we can construct a brace $B_{\infty}$ algebra, denoted $\BInfty(\cH)$, as follows:
\begin{equation*}
\BInfty(\cH):=
\bigl(\oplus_{n=0}^{\infty}
(\underbrace{\cH\otimes_{R}\cdots\otimes_{R}\cH)}_{n}[-n]\, , \delta_{\cH} \, , \cup\, ,
\{\mu_{k}^{\Delta}\}_{k\geqslant 0}
\bigr)\, .
\end{equation*}	
Here, $\delta_{\cH}$ is the Hochschild differential (see Equation \eqref{deltaHochschild}), $\cup$ denotes the multiplication (see Equation \eqref{CupProduct}), and $\{\mu_k^{\Delta}\}$ are the braces (see Equations \eqref{Brace1} and \eqref{Brace2}).
\end{Thm}

Our focus lies specifically on the twisted structure of the brace algebra $B_{\infty}$ outlined above. Consider a \textbf{twistor} (see Section \ref{subSec:twistors}) $\cF\in \cH\otimes_{R}\cH$ of the Hopf algebroid $\cH$, as previously defined in \cites{Xu-R-matrices, Xu-Quantum}. We then denote by $\cH_{\cF}=(\cH\, ,  \alpha_{\cF}\, , \beta_{\cF}\, , \cdot\, , \Delta_{\cF}\, , \epsilon)$ the resulting twisted Hopf algebroid, which is defined over the twisted associative algebra $R_{\cF}$. The second primary result of this paper can be formulated in terms of these objects, as follows.

\begin{Thm}
  [See Theorem \ref{TwoTypes}]\label{MainThmIntroduction2}
 There is a strict $B_{\infty}$ isomorphism
\begin{equation*}
\cF^{\sharp}: \BInfty{(\cH_{\cF})} \xlongrightarrow{\simeq} \BInfty{(\cH)}^{\cF}\, .
\end{equation*}
Here
\begin{equation*}
\BInfty(\cH_{\cF}):=
\bigl(\oplus_{n=0}^{\infty}
(\underbrace{\cH_{\cF}\otimes_{R_{\cF}}\cdots\otimes_{R_{\cF}}\cH_{\cF})}_{n}[-n]\, , \delta_{\cH_{\cF}} \, , \cup\, ,
\{\mu_{k}^{\Delta_{\cF}}\}_{k\geqslant 0}
\bigr)
\end{equation*}	
is the \textbf{type I twisted} brace $B_{\infty}$ algebra,
which is the canonical brace $B_{\infty}$ algebra associated with the twisted Hopf algebroid $\cH_{\cF}$, and
\begin{equation*}
\BInfty(\cH)^{\cF}:=\bigl(\oplus_{n=0}^{\infty}
(\underbrace{\cH\otimes_{R}\cdots\otimes_{R}\cH}_{n})[-n]\, ,
[\underline{}{\cF}\, , \,\cdot\,]_{\mathrm{G}} \, , \cup_{\cF}\, , \{\mu_{k}^{\Delta}\}_{k\geqslant 0}
\bigr)\, ,
\end{equation*}
is the \textbf{type II twisted} brace $B_{\infty}$ algebra
associated with the Hopf algebroid $\cH$ and the twistor
$\cF$ of $\cH$,
whose braces are the same as those of $\BInfty(\cH)$,
whose differential deforms from
$\delta_{\cH}=[\underline{}{1\otimes_{R} 1}\, , \, \cdot\,]_{\mathrm{G}}$ to
$[\underline{}{\cF}\, , \,\cdot\,]_{\mathrm{G}}$, and
whose multiplication deforms from $\underline{}{x}\cup \underline{}{y}$ to
$\underline{}{x}\cup_{\cF}\underline{}{y}:=(-1)^{|x|}
\underline{}{\cF}\{ \underline{}{x}\, , \underline{}{y}\}$ for all $\underline{}{x}\, , \underline{}{y}\in \BInfty(\cH)$.  	
\end{Thm}

 When the Hopf algebroid $H$ specializes to a bialgebra over $\fK$, a twistor  of $H$ become precisely a Drinfeld's twist \cite{Drinfeld-QuantumGroup}, say $J\in H\otimes_{\fK} H$. Consequently, applying Theorem \ref{MainThmIntroduction2} yields an isomorphism of brace $B_{\infty}$ algebras  $ J^{\sharp}:  \BInfty(H_{J})\rightarrow \BInfty(H)^{J}$. Certainly,  it underlies an isomorphism of differential graded (dg) Lie algebras, a result established by Esposito and de Kleijn in \cite[Proposition $4.9$]{Esposito-deKleijn}.

 \subsection{Two particular brace $B_{\infty}$ algebras arising from Lie algebra pairs and algebraic dynamical twists}\label{Reason}

Let $\frakg$ be a finite-dimensional Lie $\fK$-algebra and $\frakl \subset \frakg$ be a Lie subalgebra. The pair $(\frakg, \frakl)$ is referred to as a \emph{Lie algebra pair}. An element in $U(\frakg)^{\otimes 2}\otimes U(\frakl)\llbracket \hbar \rrbracket$ that is $\frakl$-invariant is called an \textbf{algebraic dynamical twist} if it satisfies both the {algebraic dynamical twist equation (ADTE)} and the {$\hbar$-adic valuation property}, as defined in \cites{Calaque-QuantizationFormal, Enriquez-Etingof1} (see Definition \ref{Def:algebraicDT}). In \cite{Calaque-QuantizationFormal}, a brace $B_{\infty}$ algebra
\begin{equation*}
\ADT:=\bigl(
\oplus_{n=0}^{\infty}
(U(\frakg)^{\otimes n}\otimes U(\frakl)\llbracket \hbar \rrbracket)^{\frakl}[-n]\, ,
\delta_{(\frakg, \frakl)}\, , \cup_{(\frakg, \frakl)}\, , \{\mu_k^{(\frakg, \frakl)}\}_{k\geqslant 0}
\bigr)	
\end{equation*}
is constructed, whose underlying dg Lie algebra governs the deformations of algebraic dynamical twists.

The brace $B_{\infty}$ algebra structure maps of $\ADT$, particularly the brace brackets $\mu_k^{(\frakg, \frakl)}$, are presented in detail in Section \ref{Subsec:AgebraADT}, and are indeed quite complex. This complexity naturally leads to the question of whether there exists a Hopf algebroid, denoted $\cH$, such that $\ADT$ can be expressed as its associated canonical brace $B_{\infty}$ algebra, i.e., $\ADT=\BInfty(\cH)$. However, the answer to this question is \emph{negative}.

In fact, if such an $\mathcal{H}$ existed, then the degree $n$ component of $B_{\infty}(\mathcal{H})$ comprises tensor products of $n$ copies of $\mathcal{H}$ over some associative $\fK$-algebra $R$. Consequently, we would have the following equalities: $$B_{\infty}(\mathcal{H})^0=R=(U(\frakl)\llbracket \hbar \rrbracket)^{\frakl} , $$ $$B_{\infty}(\mathcal{H})^1=\mathcal{H}=(U(\frakg)\otimes U(\frakl)\llbracket \hbar \rrbracket)^{\frakl} , $$ and $$B_{\infty}(\mathcal{H})^2=\mathcal{H}\otimes_{R}\mathcal{H}= (U(\frakg)\otimes U(\frakl)\llbracket \hbar \rrbracket)^{\frakl} \otimes_{(U(\frakl)\llbracket \hbar \rrbracket)^{\frakl}} (U(\frakg)\otimes U(\frakl)\llbracket \hbar \rrbracket)^{\frakl}.$$ Moreover, by the definition of $W_{(\frakg,\frakl)}$, we also have $$B_{\infty}(\mathcal{H})^2=W_{(\frakg,\frakl)}^2= (U(\frakg)\otimes U(\frakg)\otimes U(\frakl)\llbracket \hbar \rrbracket)^{\frakl}.$$ Therefore, the following equality must hold: \begin{equation}\label{Eqt:equalwhichisnottrue} (U(\frakg)\otimes U(\frakl)\llbracket \hbar \rrbracket)^{\frakl} \otimes_{(U(\frakl)\llbracket \hbar \rrbracket)^{\frakl}} (U(\frakg)\otimes U(\frakl)\llbracket \hbar \rrbracket)^{\frakl}=(U(\frakg)\otimes U(\frakg)\otimes U(\frakl)\llbracket \hbar \rrbracket)^{\frakl} .
\end{equation}
However, this equality does not generally hold for arbitrary Lie algebra pairs $(\frakg, \frakl)$. To illustrate the invalidity of relation \eqref{Eqt:equalwhichisnottrue}, consider the case where $\frakg$ is a semisimple Lie algebra and $\frakl \subset \frakg$ is a Cartan subalgebra. In this context, we have
$$(U(\frakg)\otimes U(\frakl)\llbracket \hbar \rrbracket)^{\frakl} \otimes_{(U(\frakl)\llbracket \hbar \rrbracket)^{\frakl}} (U(\frakg)\otimes U(\frakl)\llbracket \hbar \rrbracket)^{\frakl}
=\bigl(U(\frakg) \bigr)^{\frakl} \otimes \bigl( U(\frakg)\bigr)^{\frakl} \otimes U(\frakl)\llbracket \hbar \rrbracket.$$
However, this expression is strictly contained within
$$(U(\frakg)\otimes U(\frakg)\otimes U(\frakl)\llbracket \hbar \rrbracket)^{\frakl}=
\bigl( U(\frakg)\otimes U(\frakg) \bigr)^{\frakl}\otimes U(\frakl)\llbracket \hbar \rrbracket , $$
demonstrating that relation \eqref{Eqt:equalwhichisnottrue} does not hold.

This paper clarifies the precise interrelation between the brace $B_{\infty}$ algebra $\ADT$ and a specific  Hopf algebroid --- we demonstrate that $\ADT$ can be embedded into a larger brace $B_{\infty}$ algebra, denoted as $\BInfty(\mathscr{H}_{(\frakg, \frakl)})\Laurent{\hbar}$. The latter is the $\fK\Laurent{\hbar}$-linear extension of the
canonical brace $B_{\infty}$ algebra associated with a quantum groupoid $$\mathscr{H}_{(\frakg, \frakl)}:=U(\frakg)\otimes \cD\llbracket \hbar \rrbracket$$ over $R:=C^{\infty}(\frakl^{\ast})\llbracket \hbar \rrbracket$, where $\cD$ denotes the algebra of $\fK$-linear smooth differential operators on $\frakl^{\ast}$.

 To this end, we introduce two operads, denoted as $\cP_{(\frakg, \frakl)}$ and $\cP_{\mathscr{H}_{(\frakg, \frakl)}}$. These operads are designed, via the Gerstenhaber-Voronov operadic modeling of brace $B_{\infty}$ algebras   as detailed in Section \ref{Subsec:GVoperadicmodeling}, to yield brace $B_{\infty}$ algebras $W_{(\frakg,\frakl)}$ and $B_{\infty}(\mathscr{H}_{(\frakg, \frakl)})$, respectively.

Now we can state our third result --- the relationship between the brace $B_{\infty}$ algebra $\ADT$ and the quantum groupoid $\mathscr{H}_{(\frakg, \frakl)}$, which can be summarized as follows:\begin{Thm} \label{MainThmIntroduction3}\
\begin{itemize}
\item[(1)]
There is a morphism of operads
$\mathsf{c}:\cP_{(\frakg,\frakl)}\rightarrow \cP_{\mathscr{H}_{(\frakg, \frakl)}}\Laurent{\hbar}$.
which induces an injective and   strict
$B_{\infty}$ morphism
\begin{equation*}
\mathsf{c}: \ADT \rightarrow
\BInfty(\mathscr{H}_{(\frakg, \frakl)})\Laurent{\hbar}
\end{equation*}	
between brace $B_{\infty}$ algebras
via the Gerstenhaber-Voronov operadic modeling,
such that for any
$K=K_1\otimes \cdots \otimes K_n \otimes K_{n+1}\in (U(\frakg)^{\otimes n}\otimes U(\frakl)\llbracket \hbar \rrbracket)^{\frakl}$, we have
\begin{equation*}
\mathsf{c}(K)=
(K_1\otimes \cdots \otimes K_n) \cdot
\Delta^{n-1}( \varphi(K_{n+1})) \cdot
\bigl(
(\Theta_{\Gutt})_{1 \tinywedge 2\cdots \tinywedge  n-1 , n}
\boldsymbol{\cdot}
\cdots
\boldsymbol{\cdot}
(\Theta_{\Gutt})_{1 , 2}
\bigr)
\, .
\end{equation*}
Here $\Theta_{\Gutt}$ is the bi-differential operator corresponding to the \textbf{Gutt star product} on the cotangent bundle
$T^{\ast}L$ (where $L$ is a Lie group integrating $\frakl$) \cite{Gutt}.
For  $\varphi$, a $\fK\llbracket \hbar \rrbracket$-linear morphism of Hopf algebras $U(\frakl)\llbracket \hbar \rrbracket\rightarrow \cD\otimes \fK\Laurent{\hbar}$, see Section \ref{Sec:morphismofoperadsc}.
\item[(2)]
An element $K\in (U(\frakg)^{\otimes 2}\otimes U(\frakl)\llbracket \hbar \rrbracket)^{\frakl}$
satisfying   the $\hbar$-adic valuation property
is an algebraic dynamical twist
if and only if
$\mathsf{c}(K)$ is a twistor of the quantum groupoid
$\mathscr{H}_{(\frakg, \frakl)}$ over $R$.	

\item[(3)]
If $K$ is an algebraic dynamical twist, then there is an injective and strict  $B_{\infty}$ morphism
\begin{equation*}
\mathsf{c}_{K}: \ADT^{K} \rightarrow
B_{\infty}(\mathscr{H}_{(\frakg, \frakl), \mathsf{c}(K)})\Laurent{\hbar}	
\end{equation*}
from the twisted brace $B_{\infty}$ algebra  $\ADT^K$ to
the $\fK\Laurent{\hbar}$-linear extension of the
canonical brace $B_{\infty}$ algebra  of the dynamical quantum groupoid $\mathscr{H}_{(\frakg, \frakl),\mathsf{c}(K)}$, such that the following diagram commutes:
\begin{equation*}
\begin{tikzcd}
\ADT^{K} \arrow[r, "\mathsf{c}_K"] \arrow[d, "\mathsf{c}"] &
B_{\infty}(\mathscr{H}_{(\frakg, \frakl), \mathsf{c}(K)})\Laurent{\hbar}
\arrow[dl, "\bigl(\mathsf{c}(K)\bigr)^{\sharp}"] \\
\bigl(B_{\infty}(\mathscr{H}_{(\frakg, \frakl)})\Laurent{\hbar}\bigr)^{\mathsf{c}(K)}
\, . &
\end{tikzcd}	
\end{equation*}
\end{itemize}
\end{Thm}

The parallel between part $(2)$ of Theorem \ref{MainThmIntroduction3} and our previous work on \textit{smooth dynamical twists} (see \cite{Cheng-Chen-Qiao-Xiang-QDYBEI}) is noteworthy. In particular, in \cite{Cheng-Chen-Qiao-Xiang-QDYBEI}, we demonstrated that an $\frakl$-invariant element $$F=1\otimes 1\otimes 1+ O(\hbar)=F_1\otimes F_2 \otimes F_3\in U(\frakg)^{\otimes 2} \otimes C^{\infty}(\frakl^{\ast})\llbracket \hbar \rrbracket$$ qualifies as a smooth dynamical twist if and only if the element $$(F_1\otimes F_2) \cdot \Delta(F_3\star_{\PBW}) \cdot \Theta_{\Gutt}\in \mathscr{H}_{(\frakg, \frakl)}\otimes_{R}\mathscr{H}_{(\frakg, \frakl)}$$ constitutes a twistor of the quantum groupoid $\mathscr{H}_{(\frakg, \frakl)}$. The subsequent Theorems \ref{MainEmbedding}, \ref{FormalTwistANDTwistor}, and \ref{TwistedEmbedding} will provide further details.

\subsection{Notations  and conventions}\

{For our notations and conventions of fields, rings, modules, coproducts, and tensor products, see below.}
\begin{itemize}
\item
We fix $\fK$ as a field of characteristic zero, and all vector spaces are assumed to be over $\fK$. The space of $\fK$-linear morphisms between $\fK$-vector spaces $V$ and $W$ is denoted by $\Hom_{\fK}(V, W)$.

\item  We denote by $\fK\llbracket \hbar \rrbracket$ the algebra of formal power series in the variable $\hbar$ over $\fK$, and by $\fK\Laurent{\hbar}$ the algebra of formal Laurent series in $\hbar$ over $\fK$, which is the quotient field of $\fK\llbracket \hbar \rrbracket$. A $\fK\llbracket \hbar \rrbracket$-module $M$ is called completed if it satisfies $M = \varprojlim_k M / \hbar^k M$.
\item
 The symbol $\cdot \otimes_R \cdot$ stands for the tensor product over an associative $\fK$-algebra $R$. When $R = \fK$, we simplify the notation to $\cdot \otimes \cdot$. If $R$ is a completed associative $\fK\llbracket \hbar \rrbracket$-algebra, then $\cdot \otimes_R \cdot$ denotes the completed tensor product. For a $\fK$-vector space $V$ and a completed $\fK\llbracket \hbar \rrbracket$-module $M$, $V \otimes M$ denotes the completed tensor product $\varprojlim_k V \otimes_{\fK} (M / \hbar^k M)$, and similarly for $M \otimes V$.

\item  Let $A$ be a coalgebra over $R$ with coproduct $\Delta: A \rightarrow A \otimes_R A$. We employ \emph{Sweedler's notation} to simplify expressions involving coproducts, suppressing summation symbols. For example, $\Delta(x) = x_1 \otimes_R x_2$, $\Delta^2(x) = x_1 \otimes_R x_2 \otimes_R x_3$.
\end{itemize}

\medskip

{For our conventions about gradings and cochain complexes, see below.}
\begin{itemize}
\item
Throughout this paper, all gradings are assumed to be
$\rZ$-gradings. For a $\fK$-vector space $W$ with a
$\rZ$-grading, we write $W=\oplus_{i \in \rZ} W^{i}$, where $W^i$ denotes the component of degree $i$.
Given such a graded vector space $W$,
the $k$-shifted graded space $W[k]$ is defined by
$(W[k])^{i} := W^{i+k}$.
A degree $k$ morphism between graded $\fK$-vector spaces $V_1$ and $V_2$ is a $\fK$-linear map from $V_1$ to $V_2[k]$ that preserves the gradings.

\item
For a homogeneous element $x \in W$, we denote its degree by $|x|$. The copy of $x$ in $W[k]$ is denoted by $x[k]$, and its degree is given by $|x[k]| = |x| - k$.

\item
We denote a cochain complex over $\fK$ by $(W, d_W)$, where $W$ is a $\mathbb{Z}$-graded $\fK$-vector space and $d_W : W^i \to W^{i+1}$ is the differential.
\end{itemize}

\subsection*{Acknowledgements.}
We are grateful to Xiao-Wu Chen, Hua-Lin Huang, Zhangju Liu, Ping Xu, Yu Ye, and Guodong Zhou for their insightful discussions and valuable feedback. This work is supported in part by the National Natural Science Foundation of China (NSFC) under Grant No. 12301050 (Cheng), Grant No. 12071241 (Chen), and Grant No. 12271323 (Qiao), as well as by the Research Fund of Nanchang Hangkong University under Grant No. EA202107232 (Cheng).

\section{Preliminaries}\label{Section:Pre}

\subsection{$B_{\infty}$ and brace $B_{\infty}$ algebras}\label{subSec:BinfandbraceBinf}
We recall the definitions of $B_{\infty}$ algebras \cites{Getzler-Jones} and brace $B_{\infty}$ algebras \cites{Getzler-Jones, Chen-Li-Wang}, the latter also being known as homotopy $G$ algebras or Gerstenhaber-Voronov algebras \cites{Gerstenhaber-Voronov, Gerstenhaber-Voronov-Moduli}.

Consider a graded $\fK$-vector space $W=\oplus_{n\in \rZ} W^n$, and its tensor coalgebra $T^c(W[1]):=\oplus_{k=0}^{\infty} \otimes^k (W[1])$. The tensor coalgebra is equipped with the standard coproduct $T^c(W[1]) \rightarrow T^c(W[1])\otimes T^c(W[1])$:
\begin{equation*}
a_1\otimes \cdots \otimes a_n
\mapsto \sum_{i=1}^{n+1} (a_1\otimes \cdots \otimes a_{i-1})
\bigotimes  (a_{i}\otimes \cdots \otimes a_{n})\, ,
\quad \forall a_1\, , \cdots\, , a_n \in W[1]\, .
\end{equation*}
\begin{definition}[\cite{Getzler-Jones}]
A $B_{\infty}$ algebra structure on $W$
is a differential graded bialgebra structure
on $T^c(W[1])$
such that the coproduct is the standard one,
and $1\in T^c(W[1])$ serves as the algebra unit.
\end{definition}

Given a $B_{\infty}$ algebra $W$, let $D$ and $m$ denote the differential and the product on $T^c(W[1])$, respectively. Since $T^c(W[1])$ is cofree, $D$ and $m$ are determined by their projections onto $W[1]$. Consequently, we can express $D$ and $m$ as $D=\sum_{k\geqslant 1} D_k$ and $m=\sum_{k\geqslant 0, l\geqslant 0} m_{kl}$, where
\begin{equation*}
\begin{split}
D_k : & \otimes^k (W[1]) \rightarrow W[2]\, ,
\quad k\geqslant 1 \, ,  \\
m_{kl}: &
\otimes^k(W[1])\bigotimes \otimes^l(W[1])
\rightarrow 	W[1]\, , \quad k\geqslant 0\, , l \geqslant 0\, ,
\end{split}
\end{equation*}
are $\fK$-linear maps. These maps satisfy specific compatibility relations inherent to the $B_{\infty}$ structure. In particular, the collection of maps $\{D_k\}_{k\geqslant 1}$ defines an $A_{\infty}$ algebra structure on $W$.
Furthermore, it is worth noting that $m_{00}=0$,
$m_{01}=m_{10}$ is the identity map,
and $m_{0l}=0$ for all $l\geqslant 2$, which follows from the fact that $1\in T^c(W[1])$ is the unit.

\begin{definition}[\cites{Getzler-Jones, Chen-Li-Wang}]
A brace $B_{\infty}$ algebra  structure on
  $W$ is
a $B_{\infty}$ algebra structure such that
$D_k =0$ for $k\geqslant 3$ and $m_{kl}=0$ for $k\geqslant 2$ and $l\geqslant 0$.
In other words, the only non-zero structure maps
are $D_1$, $D_2 $, and $m_{1k}$ ($k\geqslant 0$).
\end{definition}
We now elaborate on the definition of a brace $B_{\infty}$ algebra. By shifting degrees, the map $D_1$ becomes a map $d: W \rightarrow W[1]$, the map $D_2$ becomes a product $\cdot : W \otimes W \rightarrow W$, and the map $m_{1k}$ becomes a map $\mu_k: W \otimes (W[1])^{\otimes k} \rightarrow W$. The \textit{brace operations} (\cite[\S $5.4$]{Chen-Li-Wang}) are defined as:
\begin{equation*}
x\lb x_1\, , \cdots\, , x_k \rb:=\mu_{k}(x \otimes x_1[1]\otimes \cdots \otimes x_k[1]) \in W\, , \quad \forall x\, , x_1\, , \cdots\, , x_k\in W\, ,
\end{equation*}
for all integers $k\geqslant 0$. The quadruple $(W\, , d\, , \cdot\, , \{ \mu_k \}_{k\geqslant 0})$ constitutes a brace $B_{\infty}$ algebra if and only if the following conditions are satisfied \cite{Gerstenhaber-Voronov, Willwacher-BracesFormality}:

(1) The map $\mu_{0}: W \rightarrow  W $ is the identity map, i.e.,  $x\lb \rb =x$ for all $x\in W$.

(2) The \textbf{higher pre-Jacobi identities} hold:
\begin{equation}\label{Jacobi}
\begin{split}
& x\lb y_1\, , \cdots\, , y_n \rb
\lb z_1\, , \cdots\, , z_m \rb\\=&
\sum (\pm)
 x \lb z_1\, , \cdots\, , z_{i_1}\, ,
			y_1\lb z_{i_1+1}\, , \cdots\, , z_{i_1+r_1}
			\rb\, ,
\cdots\, , z_{i_n}\, ,
y_{n}\lb z_{i_n+1}\, , \cdots\, , z_{i_n+r_n}\rb\, , \cdots\, , z_m \rb\, , 	
\end{split}
\end{equation}
where $x$, $y_i$, and $z_j$ are homogeneous elements of $W$, the sign $\pm$ is given by
\begin{equation*}
(-1)^{\sum_{p=1}^n (|y_p|+1) \sum_{q=1}^{i_p}(|z_q|+1)}\, ,
\end{equation*}
and the summation goes over all sequences of positive integers $(i_1\, , \cdots\, , i_n\, , r_1\, , \cdots\, , r_n)$ satisfying
\begin{equation*}
i_1+r_1 \leqslant i_2\, , \,\,\,
i_2+r_2 \leqslant i_3\, , \,\,\, \cdots\, , \,\,\,
i_{n-1}+r_{n-1} \leqslant i_n \, , \,\,\,
i_n+r_n \leqslant m\, .
\end{equation*}

(3)
The triple $(W,\, d,\, \cdot)$ forms a
differential graded associative algebra (without unit).
	
(4)
The \textbf{distributivity} relations hold:
\begin{equation*}\label{Distributivity}
\begin{split}
& \qquad\quad
(x_1\cdot x_2) \lb y_1\, , \cdots\, , y_n\rb = \\
& \qquad\qquad\qquad\qquad
\sum_{k=0}^n
		(-1)^{|x_2|\cdot \sum_{p=1}^k (|y_p|+1)} x_1
		\lb y_1\, , \cdots\, , y_k \rb
		\cdot
		x_2\lb y_{k+1}\, , \cdots\, , y_n \rb\, .
		\end{split}
\end{equation*}
	
(5)
The \textbf{higher homotopies} hold:
\begin{equation*}\label{Homotopy}
\begin{split}
& d (x\lb x_1\, , \cdots\, , x_n \rb)
			- (dx) \lb x_1\, , \cdots\, , x_n \rb \\
&\qquad\quad -
			\sum_{i=1}^n
			(-1)^{|x|+1+\sum_{j=1}^{i-1}(|x_j|+1)}
			x\lb x_1\, , \cdots\, , x_{i-1}\, , dx_{i}\, , \cdots\, , x_n\rb \\
= & (-1)^{|x|(|x_1|+1)} x_1\cdot (x
\lb x_2\, , \cdots\, , x_{n} \rb) \\
			&
-\sum_{i=1}^{n-1}(-1)^{|x|+1+
\sum_{j=1}^{i}(|x_j|+1)}
x\lb x_1\, , \cdots\, , x_i\cdot x_{i+1}\, , \cdots \, , x_n \rb \\
			& \, \,
+(-1)^{|x|+1+\sum_{j=1}^{n-1}(|x_j|+1)}
(x\lb x_1\, , \cdots\, , x_{n-1} \rb)\cdot x_n\, ,
		\end{split}	
	\end{equation*}
for all homogeneous elements
$x\, , x_1\, , \cdots\, , x_n \in W$.

The notion of \textit{brace algebras}, which is a weaker structure than brace $B_{\infty}$ algebras, is introduced in \cite{Gerstenhaber-Voronov-Moduli}.

\begin{definition}[\cite{Gerstenhaber-Voronov-Moduli}]
A brace algebra $(W, \{\mu_{k}\}_{k \geqslant 0})$ consists of a graded $\fK$-vector space $W$ and a family of $\fK$-linear maps $\mu_{k}: W \otimes (W[1])^{\otimes k} \rightarrow W $ for each integer $k\geqslant 0$,
such that $\mu_{0}: W \rightarrow W $ is the identity map, and the brace operations defined by
\begin{equation*}
x\lb x_1, \cdots, x_k \rb:=\mu_{k}(x \otimes x_1[1]\otimes \cdots \otimes x_k[1]) \in W, \quad \forall x, x_1, \cdots, x_k\in W,
\end{equation*}
satisfy the higher pre-Jacobi identities \eqref{Jacobi}.
\end{definition}

Given a brace algebra $(W, \{\mu_{k}\}_{k\geqslant 0})$, one obtains the \textbf{Gerstenhaber bracket} on $W$, which is the $\fK$-bilinear operation defined by
\begin{equation*}\label{Eqt:fakegerstenhaberbraceketfrombrace}
[x, y]_{\mathrm{G}}:=x\lb y \rb -(-1)^{(|x|+1)(|y|+1)} y \lb x \rb,
\end{equation*}
for all homogeneous elements $x, y \in W$.

It is possible to derive a brace $B_{\infty}$ algebra from a brace algebra together with a certain element.

\begin{proposition}[\cite{Gerstenhaber-Voronov-Moduli}]\label{BraceToBraceBInfinity}
Let $(W, \{\mu_{k}\}_{k\geqslant 0})$ be a brace algebra, and let $\m\in W^2$ be a degree $2$ element satisfying $\m\lb \m \rb=0$.
Define the operations
\begin{equation*}\label{mCup}
\begin{split}
\cdot: W\times W  & \rightarrow W, \\
x \cdot y : & =(-1)^{|x|}\m \lb x\, ,  y \rb.
\end{split}
\end{equation*}
and
\begin{equation*}\label{mHochschild}
\begin{split}
d: W & \rightarrow W[1],  \\
dx: & = \m \lb x \rb- (-1)^{|x|+1} x \lb \m \rb=[\m, x]_{\mathrm{G}}.
\end{split}
\end{equation*}
Then the quadruple $(W, d, \cdot, \{ \mu_k \}_{k\geqslant 0})$ forms a brace $B_{\infty}$ algebra.
\end{proposition}

A morphism of $B_{\infty}$ algebras is defined via
an equivalent formulation in terms of tensor coalgebras.
\begin{definition}[\cite{Chen-Li-Wang}]
Given two $B_{\infty}$ algebras $W_1$ and $W_2$, a $B_{\infty}$ morphism from $W_1$ to $W_2$ is defined as a dg bialgebra morphism from $T^c(W_1[1])$ to $T^c(W_2[1])$.
\end{definition}
Such a morphism $f: T^c(W_1[1])\rightarrow T^c(W_2[1])$ can be decomposed as $f=\sum_{k\geqslant 1} f_{k}$, where each component $f_k: \otimes^k (W_1[1]) \rightarrow W_2[1]$ is a $\fK$-linear map for $k\geqslant 1$. The defining characteristic of a $B_{\infty}$ morphism is that these components, $f_k$, satisfy compatibility relations with the $B_{\infty}$ structures on $W_1$ and $W_2$  \cite{Chen-Li-Wang}. It is noteworthy that a $B_{\infty}$ algebra morphism also constitutes an $A_{\infty}$ algebra morphism.

\begin{definition}[\cite{Chen-Li-Wang}]\label{Def:strictBmorphism}
A $B_{\infty}$ morphism $f$ is called
a \textbf{strict $B_{\infty}$ morphism}, if $f_{k}=0$ for all $k\geqslant 2$.
Furthermore, $f$ is called a \textbf{strict $B_{\infty}$ isomorphism}
if it is a strict $B_{\infty}$ morphism and its component
$f_1$ is an isomorphism of graded $\fK$-vector spaces.
\end{definition}

\begin{remark}\label{dgLie}
Given a brace $B_{\infty}$ algebra $(W\, , d\, , \cdot\, , \{ \mu_k \}_{k\geqslant 0})$,
the triple $(W[1]\, , -d\, , \nullbracket_{\mathrm{G}})$ constitutes a dg Lie algebra, where
$[x[1]\, , y[1]]_{\mathrm{G}}:=(-1)^{|x|}([x\, , y]_{\mathrm{G}})[1]$ for all homogeneous elements
$x\, , y\in W$.
Consequently, strict $B_{\infty}$ morphisms between brace $B_{\infty}$ algebras induce morphisms of the corresponding dg Lie algebras.
\end{remark}

\subsection{The Gerstenhaber-Voronov operadic modeling of brace $B_{\infty}$ algebras}\label{Subsec:GVoperadicmodeling}

Following Loday and Vallette \cite{Loday-Vallette}, we recall the definition of a non-symmetric operad.

\begin{definition}[{\cite[$\S 5.9.3$]{Loday-Vallette}}]
A non-symmetric operad $\cP$ consists of the following data: a collection of $\fK$-vector spaces $\cP(n)$ for each integer $n\geqslant 0$; a distinguished element $\id_{\cP}\in \cP(1)$, called the \textbf{identity}; and a family of $\fK$-linear maps, known as \textbf{composition maps}, given by
\begin{equation*}
\begin{split}
\gamma: \cP(n)\otimes \cP(k_1)\otimes\cdots\otimes\cP(k_n) & \rightarrow \cP(k_1+\cdots+k_n)\, , \\
x\otimes (x_1\otimes\cdots \otimes x_{n}) & \mapsto \gamma(x\, ; x_1\, , \cdots\, , x_n)\, .
\end{split}  	
\end{equation*}
These data are subject to the following conditions:

- \textbf{Associativity}: For any
$x\in \cP(n)$, $y_1\in \cP(k_1)\, , \cdots\, , y_n\in \cP(k_n)$,
$z_1\in \cP(m_1)$, $\cdots$, $z_{k_1}\in \cP(m_{k_1})$,
$z_{k_{1}+1}\in \cP(m_{k_1+1})$, $\cdots$,
$z_{k_1+k_2}\in \cP(m_{k_1+k_2})$,
$z_{k_1+\cdots+k_{n-1}+1}\in \cP(m_{k_1+\cdots+k_{n-1}+1})$, and
$z_{k_1+\cdots+k_{n-1}+k_n}\in \cP(m_{k_1+\cdots+k_{n-1}+k_n})$, the following associativity property holds:
\begin{equation}\label{CompositionAssociativity}
\begin{split}	
& \gamma\biggl(\gamma(x\, ; y_1\, , \cdots\, , y_n)\, ;
z_{1}\, , \cdots\, ,z_{m_{k_1+\cdots+k_n}}
\biggr) \\
= &	
\gamma\biggl(x\, ; \gamma(y_1\, ; z_1\, , \cdots\, , z_{k_1})\, ,
\gamma\bigl(y_2\, ; z_{k_1 + 1}\, , \cdots\, , z_{k_1+k_2})\, ,
\cdots\, ,
\\
& \qquad\qquad\qquad\qquad\qquad
\gamma\bigl(y_n\, ; z_{k_1+\cdots+k_{n-1}+1}\, , \cdots\, , z_{k_1+\cdots+k_{n-1}+k_n}\bigr)
\biggr)\, .
\end{split}
\end{equation}

- \textbf{Unitality}:
The identity element $\id_{\cP}\in \cP(1)$ satisfies
\begin{equation}\label{Unitality}
\gamma(\id_{\cP}\, ; x)=x\, , \qquad \mbox{and } \quad
\gamma(x\, ; \underbrace{\id_{\cP}\, , \cdots\, , \id_{\cP}}_{n})=x\, , \quad \forall x\in \cP(n).
\end{equation}
\end{definition}

Given a family of $\fK$-linear maps $\gamma: \cP(n)\otimes \cP(k_1)\otimes\cdots\otimes\cP(k_n) \rightarrow \cP(k_1+\cdots+k_n)$ and a distinguished element $\id_{\cP}\in \cP(1)$, one can define for any $1\leqslant i \leqslant n$, $x\in \cP(n)$, and $y\in \cP(m)$, the \textit{partial composition map} as:
\begin{equation}\label{PartialComposition}
x\circ_{i}y:=
\gamma(x; \underbrace{\id_{\cP}\, , \cdots\, , \id_{\cP}\, , y}_{i}
\, , \id_{\cP}\, , \cdots\, , \id_{\cP}) \in \cP(n+m-1)\, .	
\end{equation}
According to \cite[$\S 5.9.4$]{Loday-Vallette}, these maps $\gamma$ satisfy the associativity property \eqref{CompositionAssociativity} if and only if the following relations hold for all $x\in \cP(n)$, $y\in \cP(m)$, and $z\in \cP(l)$:

- \textbf{Sequential composition axiom}:
\begin{equation}\label{SequentialComposition}
(x\circ_i y) \circ_{i-1+j} z = x\circ_{i} (y \circ_{j} z)\, , \quad \text{for $1\leqslant i \leqslant n$, $1\leqslant j \leqslant m$}.
\end{equation}

- \textbf{Parallel composition axiom}:
\begin{equation}\label{ParallelComposition}
(x\circ_{i} y)\circ_{k-1+m} z =  (x\circ_{k}z)\circ_{i}y
\, , \quad \text{for $1\leqslant i< k \leqslant n$}.  	
\end{equation}
Similarly, the distinguished element $\id_{\cP}\in \cP(1)$ satisfies the unitality property \eqref{Unitality} if and only if the following relations, also known as the \textbf{unitality axiom}, hold for all $x\in \cP(n)$ and integers $1\leqslant i \leqslant n$:
 \begin{equation}\label{Unitality2}
\id_{\cP}\circ_{1} x=x\, , \qquad x\circ_{i} \id_{\cP}=x\, . 	
\end{equation}

An equivalent definition of non-symmetric operads can then be formulated as follows:
\begin{definition}[{\cite[$\S5.9.4$]{Loday-Vallette}}]\label{PartialDefinitionOperad}
A non-symmetric operad $\cP$ consists of $\fK$-vector spaces $\cP(n)$ ($n\geqslant 0$), a family of $\fK$-linear maps $\cdot \circ_i \cdot: \cP(n) \otimes \cP(m)\rightarrow \cP(n+m-1)$ referred to as the partial composition maps ($1\leqslant i \leqslant n$), and a distinguished element
$\id_{\cP}\in \cP(1)$. These data must satisfy the sequential composition axiom \eqref{SequentialComposition}, the parallel composition axiom \eqref{ParallelComposition}, and the unitality axiom \eqref{Unitality2}.
\end{definition}
It is noteworthy that the composition maps can be reconstructed from the partial composition maps via the formula:
\begin{equation*}
\gamma(x; y_1\, , \cdots\, , y_n):=
( \cdots ((x \circ_1 y_1)\circ_{k_{1} + 1} y_2 ) \circ_{k_{1} + k_{2} +1 } \cdots )\circ_{k_{1}+\cdots+k_{n-1}+1} y_n\, , 	
\end{equation*}
for all $x\in \cP(n)$, $y_1\in \cP(k_1)$, $\cdots$, $y_n\in \cP(k_n)$.

The following definition appears implicitly in
\cite[\S $5.9$]{Loday-Vallette}.

\begin{definition}[{\cite{Loday-Vallette}}]\label{MorphismOperad}
Let $\cP_1$ and $\cP_2$ be two non-symmetric operads.
A morphism of operads $f: \cP_1 \rightarrow \cP_2$ consists of
$\fK$-linear maps
$f_{n}: \cP_1(n) \rightarrow \cP_2(n)$ for all $n\geqslant 0$
such that $f_1(\id_{\cP_1})=\id_{\cP_2}$  and
  the following diagrams are commutative:
\begin{equation*}
\begin{tikzcd}
\cP_1(n)\otimes \cP_1(k_1)\otimes\cdots\otimes\cP_1(k_n) \arrow[r, "\gamma"]\arrow[d, "f_n\otimes f_{k_1}\otimes \cdots \otimes f_{k_n}"] & \cP_1(k_1+\cdots+k_n)
\arrow[d, "f_{k_1+\cdots+k_n}"]	\\
\cP_2(n)\otimes \cP_2(k_1)\otimes\cdots\otimes\cP_2(k_n) \arrow[r, "\gamma"] & \cP_2(k_1+\cdots+k_n)\, .
\end{tikzcd}
\end{equation*}
We say that $f$ is an isomorphism of operads, if
each $f_n$ is an isomorphism of $\fK$-vector spaces.
\end{definition}

Suppose that a set of $\fK$-linear maps
$\{f_n: \cP_1(n)\rightarrow \cP_{2}(n)\}_{n\geqslant 0}$
is given and $f_1(\id_{\cP_1})=\id_{\cP_2}$ is assumed.
To demonstrate that the set of $\fK$-linear maps $\{f_n\}_{n\geqslant 0}$ defines a morphism of operads, it is sufficient to verify that these maps are compatible with the partial composition maps.
More precisely, we need to show that the following diagrams commute for all $m\geqslant 0\, ,  n\geqslant 0\, , 1\leqslant i \leqslant n$:
\begin{equation}\label{PartialMorphism}
\begin{tikzcd}
\cP_1(n)\otimes \cP_1(m) \arrow[d, "f_n\otimes f_m"] \arrow[r, "\cdot\circ_{i}\cdot"] & \cP_1(n+m-1) \arrow[d, "f_{n+m-1}"]	 \\
\cP_2(n)\otimes \cP_2(m) \arrow[r, "\cdot\circ_{i}\cdot"] & \cP_2(n+m-1)\, .
\end{tikzcd}	
\end{equation}
This commutativity ensures that for any $n, m \geqslant 0$, the image of the partial composition in $\cP_1$ under the map $f_{n+m-1}$ equals the partial composition in
$\cP_2$ of their respective images under the maps
$f_{n}$ and $f_{m}$, thus confirming the operadic morphism property.

We now recall the operadic modeling of brace $B_{\infty}$ algebras, namely, a
construction of brace $B_{\infty}$ algebras from operads, as presented by Gerstenhaber and Voronov \cite{Gerstenhaber-Voronov-Moduli}. Given any  non-symmetric operad $\cP$, consider the graded $\fK$-vector space
\begin{equation*}\label{WP}
W_{\cP}:=\oplus_{n=0}^{\infty} \cP(n)[-n]\, .
\end{equation*}
Here, an element $x \in \cP(n)$ is assigned a degree of $n$ in $W_{\cP}$, denoted as $|x|=n$. Consequently, the degree of $x[1]$ in $W_{\cP}[1]$ is $|x[1]| = n-1$. For each $k\geqslant 1$, we define the brace operation
\begin{equation*}
\mu_{k}^{\cP}: W_{\cP} \bigotimes (W_{\cP}[1])^{\otimes k} \rightarrow W_{\cP}
\end{equation*}
as
\begin{equation}\label{OperadbraceI}
\begin{split}
& \mu_{k}^{\cP}( x\, ; y_1[1]\, , \cdots\, , y_k[1])=
x\{y_1\, , \cdots\, , y_k  \} \\
:= &
\sum (-1)^{\sum_{p=1}^k i_p (|y_p|-1) }
\gamma(x\, ; \id_{\cP}\, , \cdots\, , \id_{\cP}\, , y_1\, , \id_{\cP}\, , \cdots\, , \id_{\cP}\, , y_k\, , \id_{\cP}\, , \cdots\, , \id_{\cP})\, ,
\end{split}
\end{equation}
where $x, y_1, \cdots, y_k \in W_{\cP}$. In Equation \eqref{OperadbraceI}, $\gamma$ represents the composition map of the operad $\cP$, and $\id_{\cP}$ denotes the identity element. The summation goes over all possible insertions of $y_1, \dots, y_k$ into $x$, with the identity element $\id_{\cP}$ filling the remaining inputs.
In the summation, the index $i_p$ represents the count of inputs preceding $y_p$ within each term.
Specifically, each instance of $\id_{\cP}$ preceding $y_p$ contributes a count of $1$, while each $y_q$ preceding $y_p$ contributes a count equal to its degree $|y_q|$.

\begin{definition}[\cite{Gerstenhaber-Voronov-Moduli}]\label{MultiplicationDefinition}
A multiplication on a non-symmetric operad $\cP$ is an element $\m\in \cP(2)$ satisfying the associativity condition:
\begin{equation*}\label{OperadMult}
\gamma(\m; \m\, , \id_{\cP})=\gamma(\m; \id_{\cP}\, , \m)\, .
\end{equation*}
\end{definition}

Based on the work of Gerstenhaber and Voronov \cite{Gerstenhaber-Voronov-Moduli},
it is not hard to conclude the following theorem and proposition, respectively,
which describe the operadic modeling of brace $B_{\infty}$ algebras and their functoriality.

\begin{theorem} \label{Operadtobrace}
(1) The graded $\fK$-vector space $W_{\cP} = \oplus_{n=0}^{\infty}\cP(n)[-n]$, equipped with the brace operations $\{\mu_{k}^{\cP}\}_{k\geqslant 0}$, forms a brace algebra:
\begin{equation*}
(W_{\cP}=\oplus_{n=0}^{\infty}\cP(n)[-n]\, ,
\{\mu_{k}^{\cP}\}_{k\geqslant 0})\, .	
\end{equation*}

(2) An element $\m\in \cP(2)$ defines a multiplication on the operad $\cP$ if and only if $\m \{ \m \}=0$. Under this condition, we define two operations on $W_{\cP}$:
\begin{itemize}
	\item The cup product $\cup: W_{\cP} \times W_{\cP}  \rightarrow W_{\cP}  $ defined by
\begin{equation*}\label{OperadCup}
x \cup y : = (-1)^{|x|}\m\lb x\, , y \rb \, .
\end{equation*}

 \item The Hochschild differential $\delta: W_{\cP}  \rightarrow W_{\cP}[1]$ defined by
	\begin{equation*}\label{OperadHochschild}
		\begin{split}
						\delta(x): = \m\{ x \} - (-1)^{|x|+1} x \{ \m \} =[\m\, , x]_{\mathrm{G}} \, ,
		\end{split}
	\end{equation*}
	where $[\,\cdot\, ,  \cdot\,]_{\mathrm{G}}$ denotes the Gerstenhaber bracket.
\end{itemize}
Then, the quadruple
\begin{equation*}
(W_{\cP}=\oplus_{n=0}^{\infty}\cP(n)[-n]\, , \delta\, , \cup\, , \{\mu_{k}^{\cP}\}_{k\geqslant 0})	
\end{equation*}
constitutes a brace $B_{\infty}$ algebra.
\end{theorem}

\begin{proposition} \label{functoriality}
Let $f: \cP_1 \rightarrow \cP_2$ be a morphism of non-symmetric operads, and let $\m_1\in \cP_1(2)$ and $\m_2 \in \cP_2(2)$ be multiplications on $\cP_1$ and $\cP_2$, respectively, such that $f_2(\m_1)=\m_2 $. Then there exists a strict $B_{\infty}$ morphism (as defined in Definition \ref{Def:strictBmorphism}) between the corresponding brace $B_{\infty}$ algebras:
\begin{equation*}
\begin{split}
(W_{\cP_1}=\oplus_{n=0}^{\infty}\cP_1(n)[-n]\, , \delta_1\, , \cup_1\, , \{\mu_{k}^{\cP_1}\}_{k\geqslant 0})
& \xrightarrow{ \oplus_{n=0}^{\infty} f_n } \\
& \quad (W_{\cP_2}=\oplus_{n=0}^{\infty}\cP_2(n)[-n]\, , \delta_2\, , \cup_2\, , \{\mu_{k}^{\cP_2}\}_{k\geqslant 0})
\, .	
\end{split}
\end{equation*}
\end{proposition}

\begin{example}
Let $A$ be an associative algebra over $\fK$.
Consider the non-symmetric endomorphism operad $\End(A)$, defined by $\End(A)(n)=\Hom_{\fK}(A^{\otimes n}, A)$
(for details, see \cite[$\S 5.9.8$]{Loday-Vallette}).
The associative product on $A$ corresponds to
a multiplication $\m$ on $\End(A)$.
Consequently, by applying Theorem \ref{Operadtobrace} to
$\End(A)$ and $\m$, we obtain the canonical brace $B_{\infty}$ algebra structure on the graded space $C(A, A)$ of Hochschild cochains of $A$, as described in \cite{Gerstenhaber-Voronov}.
\end{example}

For further development of
the Gerstenhaber-Voronov operadic modeling,
the reader may consult, for example,
\cites{Kowalzig-OperadFromBialgebroid, Kowalzig-Operad, Menichi}.

\section{From Hopf algebroids to brace $B_{\infty}$ algebras}\label{HopfOperad}

Starting from a bialgebra (or a Hopf algebra), one can derive a  brace $B_{\infty}$ algebra \cite{Kadeishvili}.
In this part, we generalize this construction to the setting of a Hopf algebroid. We begin by recalling the essential definitions and properties of Hopf algebroids, drawing primarily from
\cites{Xu-Quantum, Xu-QuantumCNRS}.

\subsection{Hopf algebroids}\label{Section1}
Let $R$ be an associative $\fK$-algebra with unit $1$.
\begin{definition}\label{Def:Hopfalgebroid}
A Hopf algebroid over $R$ is a sextuple $(\cH\, , \alpha\, , \beta\, , \cdot\, , \Delta\, , \epsilon)$ where:
\begin{enumerate}
\item $(\cH, \,\cdot\,)$ is an associative $\fK$-algebra with unit $1$.
\item $\alpha\colon R \rightarrow \cH$ is an algebra homomorphism, called the \textbf{source map}, and $\beta\colon R \rightarrow \cH$ is an algebra anti-homomorphism, called the \textbf{target map}. The images of these maps commute, i.e.,
$\alpha(a)\cdot \beta(b)=\beta(b) \cdot \alpha(a)$, for all $a, b\in R$.
\item $\Delta\colon \cH \rightarrow \cH\otimes_{R} \cH$, called the \textbf{coproduct}, is an $(R\,, R)$-bimodule morphism satisfying $\Delta(1)=1\otimes_{R} 1$ and the coassociativity condition:
\begin{equation}\label{CoAssociativity}
(\Delta\otimes_{R} \id)\circ \Delta= (\id\otimes_{R} \Delta)\circ \Delta\colon \cH \rightarrow \cH\otimes_{R} \cH\otimes_{R} \cH\, .	
\end{equation}
\item The coproduct is compatible with the source map, target map, and the product in the following sense:
\begin{equation}\label{Important1}
\Delta(x)\cdot( \beta(a)\otimes 1 -1\otimes \alpha(a) )=0\, , \mbox{ for all } x\in \cH, a\in R\, , 	
\end{equation}
and
\begin{equation}\label{CoproductProduct}
\Delta(x_1\cdot x_2)=\Delta(x_1)\cdot \Delta(x_2)
\, , \mbox{ for all } x_1\, , x_2\in \cH \, .
\end{equation}
\item $\epsilon\colon \cH \rightarrow R$, called the \textbf{counit}, is an $(R\,, R)$-bimodule map, satisfying $\epsilon(1)=1$ and the condition:
\begin{equation*}
(\epsilon\otimes_{R}\id)\circ	\Delta=(\id\otimes_{R}\epsilon)\circ \Delta=\id\, ,
\end{equation*}
where we identify $R\otimes_{R} \cH \simeq \cH\otimes_{R}R \simeq \cH$.
\item  {$\Ker(\epsilon)$ is a left ideal of $\cH$.}	
\end{enumerate}
\end{definition}

For brevity, we often denote $(\cH\, , \alpha\, , \beta\, , \cdot\, , \Delta\, , \epsilon)$ simply by $\cH$.
We list below several facts about a Hopf algebroid:
\begin{itemize}
\item
A natural $(R, R)$-bimodule structure is induced on $\cH$ by
\begin{equation*}
a \cdot x := \alpha(a)\cdot x, \quad x \cdot b:=\beta(b)\cdot x\, ,
\end{equation*}
for all $a, b\in R$ and $x\in \cH$.

\item
For any $n\geqslant 2$, the $n$-fold tensor product of $\cH$ over $R$ can be expressed as
\begin{equation*}
\underbrace{\cH\otimes_{R}\cdots \otimes_{R}\cH}_{n}	= (\otimes^n\cH) / I_n\, ,
\end{equation*}
where $\otimes^n\cH$ represents the $n$-fold tensor product of $\cH$ over $\fK$, forming a natural associative algebra, and $I_n$ is the right ideal generated by elements of the form
\begin{equation*}
\underbrace{1\otimes\cdots \otimes 1}_{i-1}\otimes
	\beta(r)\otimes	 1\otimes\cdots\otimes 1
-\underbrace{1\otimes\cdots \otimes 1}_{i-1}\otimes
	1\otimes \alpha(r)  \otimes\cdots\otimes  1\, ,
\end{equation*}
for all $r\in R$ and $1 \leqslant i \leqslant n-1$.

The tensor product $\underbrace{\cH\otimes_{R}\cdots \otimes_{R}\cH}_{n}$ admits a right action by $\otimes^n\cH$:
\begin{equation}\label{RIGHTACTION}
(h_1\otimes_{R}\cdots \otimes_{R}h_n)
\cdot
(u_1\otimes\cdots\otimes u_n):=
(h_1\cdot u_1)\otimes_{R}\cdots
\otimes_{R} (h_n \cdot u_n)\, , 	
\end{equation}
for all $h_1\, , \cdots\, , h_n\, , u_1\, , \cdots\, , u_n \in \cH$.

Note, however, that a right action of $\underbrace{\cH\otimes_{R}\cdots \otimes_{R}\cH}_{n}$ on itself cannot be defined analogously.

Furthermore, given elements
$h\in \underbrace{\cH\otimes_{R}\cdots \otimes_{R}\cH}_{n}$,
$h'\in \underbrace{\cH\otimes_{R}\cdots \otimes_{R}\cH}_{m}$,
$u\in \otimes^n \cH$,
and $u'\in \otimes^m \cH$,
the right action of $u\otimes_{\fK}u'$ on $h\otimes_{R}h'$ satisfies
\begin{equation}\label{PassThrough}
(h\otimes_{R}h')	\cdot (u\otimes_{\fK}u')= (h\cdot u)\otimes_{R} (h'\cdot u')\, .
\end{equation}

\item  Following \cite{Lu},
we include the condition that $\Ker(\epsilon)$ is a left ideal of $\cH$ in the above definition, which is not required in~\cite{Xu-Quantum}. This condition is crucial for endowing $R$ with a left $\cH$-module structure $\rhd\colon \cH\times R \rightarrow R$ defined by
\begin{equation}\label{Action}
x \rhd a := \epsilon(x\cdot \alpha(a) )= \epsilon(x\cdot \beta(a) )\, , 	
\end{equation}
for all $x\in \cH$ and $a\in R$.
\end{itemize}

\subsection{The operad associated with a Hopf algebroid}\label{OperadHopf}

We aim to construct a non-symmetric operad $\cP_{\cH}$ associated with a Hopf algebroid $\cH=(\cH, \alpha, \beta, \cdot, \Delta, \epsilon)$ over $R$. Our strategy is to explicitly define the data that constitutes $\cP_{\cH}$.

\subsubsection{The vector space $\cP_{\cH}(n)$ and iterated coproduct $\Delta^n$}

For each non-negative integer $n\geqslant 0$, we define $\cP_{\cH}(n)$ to be the $\fK$-vector space given by the $n$-fold tensor product of $\cH$ over $R$:
\begin{equation*}
\cP_{\cH}(n):=	\underbrace{\cH\otimes_{R}\cdots \otimes_{R}\cH}_{n}  \, .
\end{equation*}
Next, we define a series of $\fK$-linear maps
\begin{equation*}
\Delta^{n}: \cP_{\cH}(1)\rightarrow \cP_{\cH}(n+1)
\end{equation*}
for $n \geq 0$, iteratively, by using the coproduct $\Delta$ of the Hopf algebroid $\cH$. Explicitly, we define:
\begin{equation*}
\Delta^0:=\id\, , \qquad \Delta^1:=\Delta\, , \qquad
\Delta^{n}:=(\Delta^{n-1}\otimes_{R}\id)\circ \Delta\, .
\end{equation*}

\subsubsection{The insertion product $\overline{\mathrm{Ins}}^{\Delta}$}

We now introduce the insertion product, denoted $\overline{\mathrm{Ins}}^{\Delta}$, a crucial operation for defining the operad structure. To begin with, consider the following $\fK$-linear map:
\begin{equation*}\label{KeyMap1}
\begin{split}
\mathrm{Ins}^{\Delta}:\quad
\cH \otimes (\otimes^n\cH) & \rightarrow \cP_{\cH}(n)	\, , \\
h\otimes (h_1\otimes\cdots\otimes h_n) & \mapsto
\Delta^{n-1}(h)\cdot (h_1\otimes\cdots\otimes h_n) \, ,
\end{split}
\end{equation*}
where $n\geqslant 1$, and $\Delta^{n-1}(h)\cdot (h_1\otimes \cdots\otimes h_n)$ represents the right action of $h_1\otimes \cdots\otimes h_n\in \otimes^n\cH$ on $\Delta^{n-1}(h)\in \cP_{\cH}(n)$ (refer to Equation \eqref{RIGHTACTION}).

We can visualize this with the following commutative diagram:
\begin{equation}\label{DiagramRDelta}
\begin{tikzcd}
\cH \otimes (\otimes^n\cH) \arrow[d, "p"]
\arrow[r, "\mathrm{Ins}^{\Delta}"] & \cP_{\cH}(n) \\
\cP_{\cH}(1) \otimes \cP_{\cH}(n)\, .
\arrow[ru, dashed, "\overline{\mathrm{Ins}}^{\Delta}"'] &
\end{tikzcd}	
\end{equation}
Here, the vertical arrow $p$ represents the natural quotient map.  The existence of the dashed map $\overline{\mathrm{Ins}}^{\Delta}$ is guaranteed by the following lemma.

\begin{lemma}\label{KeyLemma}
There exists a well-defined $\fK$-linear map
\begin{equation*}
\overline{\mathrm{Ins}}^{\Delta}: ~
\cP_{\cH}(1) \otimes \cP_{\cH}(n)
\rightarrow
\cP_{\cH}(n)
\end{equation*}
such that Diagram
\eqref{DiagramRDelta} commutes.
\end{lemma}

\begin{proof}
When $n=1$, $p$ becomes the identity map, and
both $\overline{\mathrm{Ins}}^{\Delta}$ and
$\mathrm{Ins}^{\Delta}$ become the multiplication map of the Hopf algebroid $\cH$. In this case Diagram \eqref{DiagramRDelta} commutes.

To prove the existence of $\overline{\mathrm{Ins}}^{\Delta}$ that makes Diagram \eqref{DiagramRDelta} commutative for $n\geqslant 2$,
we must demonstrate that $\mathrm{Ins}^{\Delta}( \Ker p )=0$.  In other words, we need to show that the following identity holds for any $n\geqslant 2$:
\begin{equation}\label{Ind}
 \Delta^{n-1}(h)\cdot \bigl(\underbrace{1\otimes\cdots \otimes 1}_{i-1}\otimes
	\beta(r)\otimes	 1\otimes\cdots\otimes 1
- \underbrace{1\otimes\cdots \otimes 1}_{i-1}\otimes
	1\otimes \alpha(r)  \otimes\cdots\otimes 1\bigr)=0
\end{equation}
for all $h\in \cH$, $r\in R$, and $1\leqslant i \leqslant n-1$.
We proceed by induction on $n$, starting from the base case $n=2$.

\textbf{-}
For the base case, $n=2$, Equation \eqref{Ind} simplifies to the compatibility condition expressed in Equation \eqref{Important1}.

\textbf{-}
Now, assume that Equation \eqref{Ind} holds for some fixed integer $n$. We will show that it also holds for $n+1$.
Let $\Delta(h)= h_{1}\otimes_{R} h_{2}$ for some
$h_{1}, h_{2} \in \cH$ (here Sweedler's notation has been adopted).
Then, for $1\leqslant i \leqslant n-1$, we have
\begin{align*}
&\Delta^{n }(h)\cdot\bigl(\underbrace{1\otimes\cdots \otimes 1}_{i-1}\otimes
	\beta(r)\otimes	 1\otimes\cdots\otimes 1
- \underbrace{1\otimes\cdots \otimes 1}_{i-1}\otimes
	1\otimes \alpha(r)  \otimes\cdots\otimes 1\bigr) \\
= &  (\Delta^{n-1}\otimes_{R}\id)(\Delta(h))
\cdot
\bigl(
(
\underbrace{1\otimes\cdots \otimes 1}_{i-1}\otimes
	\beta(r)\otimes	 1\otimes \cdots
	-\underbrace{1\otimes\cdots \otimes 1}_{i-1}\otimes
	 1 \otimes \alpha(r)	\otimes\cdots
)
\otimes 1
\bigr)\\
= &
(\Delta^{n-1}\otimes_{R}\id)(h_{1}\otimes_{R}h_{2})
\cdot
\bigl(
(
\underbrace{1\otimes\cdots \otimes 1}_{i-1}\otimes
	\beta(r)\otimes	 1\otimes\cdots
-\underbrace{1\otimes\cdots \otimes 1}_{i-1}\otimes 1
	\otimes \alpha(r)	\otimes\cdots
)
\otimes 1
\bigr) \\
= &
(\Delta^{n-1}h_{1}\otimes_{R}h_{2})
\cdot
\bigl(
(
\underbrace{1\otimes\cdots \otimes 1}_{i-1}\otimes
	\beta(r)\otimes	 1\otimes\cdots
-\underbrace{1\otimes\cdots \otimes 1}_{i-1}\otimes
	1 \otimes	\alpha(r) \otimes\cdots)
\otimes 1
\bigr)	\\
= &
\bigl(\Delta^{n-1}h_{1}\cdot
(
\underbrace{1\otimes\cdots \otimes 1}_{i-1}\otimes
	\beta(r)\otimes	 1\otimes\cdots
-\underbrace{1\otimes\cdots \otimes 1}_{i-1}\otimes
	 1 \otimes \alpha(r)	\otimes\cdots
)
\bigr)\otimes_{R}h_{2} \quad \text{(by Equation \eqref{PassThrough})}\\
= & 0. \\
\end{align*}
The last step follows from the induction hypothesis.

For the case where $i=n$, we have
\begin{align*}
& (\Delta^{n}h)\cdot
(1\otimes\cdots \otimes\beta(r)\otimes1-1\otimes\cdots \otimes1\otimes\alpha(r)) \\
= & (\id\otimes_{R}\Delta^{n-1})(\Delta h)
\cdot
(1\otimes\cdots \otimes\beta(r)\otimes1-1\otimes\cdots \otimes1\otimes\alpha(r)) \\
= &
\bigl(h_{1}\otimes_{R} \Delta^{n-1}h_{2}\bigr)
\cdot
\bigl(
1\otimes
(
1\otimes\cdots \otimes\beta(r)\otimes1-1\otimes\cdots \otimes1\otimes\alpha(r))
\bigr) \\
= &
h_{1} \otimes_{R}
\bigl(
\Delta^{n-1}h_{2}\cdot
(1\otimes\cdots \otimes\beta(r)\otimes1-1\otimes\cdots \otimes1\otimes\alpha(r))
\bigr) \qquad
\text{(by Equation \eqref{PassThrough})}\\
= & 0,
\end{align*}
The final step here also relies on the induction hypothesis.
 \end{proof}

The above lemma enables us to introduce the following notation:
\begin{equation}\label{DeltaN}
(\Delta^{n-1}h) \cdot u
:=
\begin{cases}
\overline{\mathrm{Ins}}^{\Delta}(h \otimes u)
\in \cP_{\cH}(n)\, ,
& \text{for } n\geqslant 1, \text{ where } u\in \cP_{\cH}(n), h\in \cP_{\cH}(1)=\cH, \\
& \\
h\rhd u \in \cP_{\cH}(0)=R\, ,
& \text{for } n=0, \text{ where } u\in \cP_{\cH}(0)=R, h\in \cP_{\cH}(1)=\cH.
\end{cases}
\end{equation}
(For the definition of $\rhd$, refer to Equation \eqref{Action}).  Note that for $n=1$, we have $(\Delta^{0} h)\cdot u=h\cdot u \in \cH$ for all $u, h\in \cP_{\cH}(1)=\cH$. When $n=2$, Equation \eqref{DeltaN} yields that
the well-defined term $\Delta(h_1)\cdot\Delta(h_2)$
for all $h_1, h_2\in \cH$,
which appears in Equation \eqref{CoproductProduct},
satisfies
$\Delta(h_1)\cdot\Delta(h_2)= \overline{\mathrm{Ins}}^{\Delta}(h_1 \otimes \Delta(h_2) )\in \cP_{\cH}(2)$.

\subsubsection{Operadic structures on $\cP_{\cH}$}
Now we define the operadic structure on the family of vector spaces $\{ \cP_{\cH}(n) \}$.

\textbf{-}
The \textbf{composition map}
\begin{equation*}
\gamma:
\cP_{\cH}(n)\otimes \cP_{\cH}(k_1)\otimes
\cdots\otimes\cP_{\cH}(k_n) \rightarrow
\cP_{\cH}(k_1+\cdots+k_n )
\end{equation*}
for integers $n\geqslant 0, k_1\geqslant 0, \cdots, k_n \geqslant 0$ is determined by the rule:
\begin{equation}\label{Composition}
\gamma(h\, ; u_1\, , \cdots\, , u_{n})
:=
(\Delta^{k_1 -1}h_1) \cdot u_{1}\otimes_{R}
\cdots \otimes_{R} (\Delta^{k_n -1}h_n)\cdot u_n\, ,
\end{equation}
for all $h=h_1\otimes_{R}\cdots\otimes_{R}h_n\in \cP_{\cH}(n)$ with $h_1, \cdots, h_n \in \cH$, and $u_1\in \cP_{\cH}(k_1)$, $\cdots$, $u_n\in \cP_{\cH}(k_n)$. The terms $(\Delta^{k_1 -1}h_1) \cdot u_{1}$, $\cdots$, $(\Delta^{k_n -1}h_n) \cdot u_{n}$ are as defined in Equation \eqref{DeltaN}.
The map $\gamma$ is well-defined due to Equation \eqref{PassThrough}.

\textbf{-}
The \textbf{identity}
$\id_{\cP_{\cH}}$ is simply the element $1\in \cH=\cP_{\cH}(1)$.

\begin{proposition}\label{MainTheoremOperad}
The family of $\fK$-vector spaces $\{\cP_{\cH}(n)\}_{n\geqslant 0}$, along with the composition maps $\gamma$ defined by Equation \eqref{Composition}, and the identity element $\id_{\cP_{\cH}}=1\in \cP_{\cH}(1)$, constitute a non-symmetric operad $\cP_{\cH}$.
\end{proposition}
\begin{proof}
The unitality property \eqref{Unitality} clearly holds for $\id_{\cP_{\cH}}=1$.
To verify that $\gamma$ defined by \eqref{Composition} satisfies the associativity property \eqref{CompositionAssociativity}, we examine the equivalent sequential composition axiom \eqref{SequentialComposition} and the parallel composition axiom \eqref{ParallelComposition}. For all $h=h_1\otimes_{R}\cdots \otimes_{R}h_n\in \cP(n)$, $u=u_1\otimes_{R} \cdots \otimes_{R} u_m\in \cP(m)$, and $v\in \cP(l)$, we have
 \begin{equation*}
\begin{split}
 & (h \circ_{i} u)\circ_{i-1+j} v \\
= &
h_1\otimes_{R} \cdots \otimes_{R}
\biggl(
(\underbrace{\id\otimes_{R} \cdots \otimes_{R}\Delta^{l-1}}_{i-1+j}\otimes_{R} \cdots\otimes_{R} \id)
\bigl((\Delta^{m-1}h_i)  \cdot u  \bigr)
\biggr)\cdot v
\otimes_{R} \cdots \otimes_{R}h_n \\
= &
h_1\otimes_{R} \cdots \otimes_{R}
(\Delta^{(m+l-1)-1}h_i)\cdot
\biggl(
u_1\otimes_{R}\cdots \otimes_{R}
(\Delta^{l-1}u_j)\cdot v
\otimes_{R} \cdots \otimes_{R} u_m
\biggr)
\otimes_{R} \cdots \otimes_{R}h_n \\
& \qquad \text{(by Equations \eqref{CoAssociativity} and \eqref{CoproductProduct})} \\
= & h\circ_{i} (u\circ_{j}v)\, ,
\end{split}	
\end{equation*}
where $1\leqslant i \leqslant n$ and $1\leqslant j \leqslant m$.

Additionally, we have
\begin{equation*}
\begin{split}
& (h\circ_{i}u)\circ_{k-1+m}v  \\
= & h_1\otimes_{R} \cdots \otimes_{R}
(\Delta^{m-1}h_i)\cdot u \otimes_{R} \cdots \otimes_{R}
(\Delta^{l-1}h_k)\cdot v \otimes_{R} \cdots \otimes_{R} h_n
\\
= & (h\circ_{k}v)\circ_{i}u\, ,
\end{split}
\end{equation*}	
where $1\leqslant i< k \leqslant n$. This completes the proof.
\end{proof}

For more on constructing cyclic operads with multiplications and cyclic modules over operads with multiplications derived from Hopf algebroids, see \cites{Kowalzig-OperadFromBialgebroid, Kowalzig-Operad}.

\subsection{The brace $B_{\infty}$ algebra  associated with a Hopf algebroid}\label{BraceHopf}

Having constructed a non-symmetric operad $\cP_{\cH}$ from a Hopf algebroid $\cH=(\cH, \alpha, \beta, \cdot, \Delta, \epsilon)$ over $R$, we now focus on a particular element:
\begin{equation*}
\m:=1\otimes_{R}1\in \cH\otimes_{R}\cH = \cP_{\cH}(2)\, .
\end{equation*}
This element $\m$ defines a multiplication on the operad $\cP_{\cH}$ (see Equation \eqref{OperadMult}).
We are able to state the following result,
which is ought to be known to experts in the field.

\begin{theorem}\label{MainTheorembrace}
To any Hopf algebroid $\cH=(\cH, \alpha, \beta, \cdot, \Delta, \epsilon)$ over $R$, we can associate a brace $B_{\infty}$ algebra
\begin{equation*}
\BInfty(\cH):=
\bigl(\oplus_{n=0}^{\infty}
(\underbrace{\cH\otimes_{R}\cdots\otimes_{R}\cH)}_{n}[-n], \delta_{\cH}, \cup,
\{\mu_{k}^{\Delta}\}_{k\geqslant 0}
\bigr)\, ,	
\end{equation*}
whose structure maps are given as follows:

\begin{itemize}
\item[\textbf{-}]
The \textbf{Hochschild differential} $\delta_{\cH}$:

For all $\underline{}{x_1\otimes_{R} \cdots\otimes_{R} x_n} \in (\underbrace{\cH\otimes_{R}\cdots\otimes_{R}\cH)}_{n}[-n]$, we define
\begin{equation}\label{deltaHochschild}
\begin{split}
\delta_{\cH}\bigl(\underline{}{x_1\otimes_{R} \cdots\otimes_{R} x_n} \bigr)	 & :=
\underline{}{x_1\otimes_{R}\cdots \otimes_{R}x_n\otimes_{R} 1} \\
& \qquad - \sum_{i=1}^n (-1)^{n-i}
\underline{}{x_1\otimes_{R} \cdots \otimes
\Delta(x_i)\otimes_{R} \cdots \otimes_{R} x_n} \\
& \qquad\quad - (-1)^{n}
\underline{}{1\otimes_{R}x_1\otimes_{R}\cdots
\otimes_{R}x_n}
\, ,
\end{split}
\end{equation}

\item[\textbf{-}]
The \textbf{cup product} $\cup$:

For elements
\begin{equation*}
\begin{split}
& x=x_1\otimes_{R} \cdots\otimes_{R} x_n \in (\underbrace{\cH\otimes_{R}\cdots\otimes_{R}\cH)}_{n}[-n]
\\
\text{and}\quad
& y=y_1\otimes_{R} \cdots\otimes_{R} y_m
\in (\underbrace{\cH\otimes_{R}\cdots\otimes_{R}\cH)}_{m}[-m]\, ,
\end{split}
\end{equation*}
we define
\begin{equation}\label{CupProduct}
\begin{split}
& x
\cup
y
:=
(-1)^{|x||y|}x\otimes_{R}y=
(-1)^{nm}
x_1\otimes_{R} \cdots\otimes_{R} x_n\otimes_{R}y_1\otimes_{R} \cdots \otimes_{R}y_m
\, .
\end{split}
\end{equation}

\item[\textbf{-}]
The maps $\mu_{k}^{\Delta}$:

Given elements
\begin{equation*}
\begin{split}
& \underline{}{x}=\underline{}{x_1 \otimes_{R}\cdots\otimes_{R} x_p} \in \underbrace{(\cH\otimes_{R}\cdots \otimes_{R}\cH)}_p[-p]\, , \\
\text{and}\quad
& \underline{}{y_{1}}\in \underbrace{(\cH\otimes_{R}\cdots \otimes_{R}\cH)}_{q_1}[-q_1]\, , \cdots\, ,
\underline{}{y_{k}}\in \underbrace{(\cH\otimes_{R}\cdots \otimes_{R}\cH)}_{q_k}[-q_k]\, ,	
\end{split}
\end{equation*}
we define $\mu_k^\Delta$ based on the relationship between $p$ and $k$:
\begin{itemize}
\item[(1)]
If $p\geqslant k$, we define
\begin{equation}\label{Brace1}
\begin{split}
& \mu_{k}^{\Delta}(\underline{}{x}\otimes \underline{}{y_1}[1]\otimes \cdots \otimes \underline{}{y_k}[1]) := \underline{}{x}\lb\underline{}{y_1}\, , \cdots\, , \underline{}{y_k}\rb
=\sum_{1 \leqslant j_1<\cdots<j_k\leqslant p}\!\!\!\!
(-1)^{\sum_{l=1}^{k}(j_{l} - l + q_{1}+\cdots + q_{l-1})(q_{l}-1)}\\
&
\biggl(\underline{}{
x_1  \otimes_{R}\cdots \otimes_{R}
\bigl((\Delta^{q_{1}-1}x_{j_1}) \cdot y_1\bigr)
\otimes_{R}
\cdots
\otimes_{R}
\bigl((\Delta^{q_{k}-1}x_{j_{k}})\cdot y_{k}\bigr)
\otimes_{R}\cdots \otimes_{R} x_{p}
} \biggr)\, ,
\end{split}
\end{equation}
where the terms
$(\Delta^{q_{1}-1}x_{j_1})	\cdot y_1$,
$\cdots$,
$(\Delta^{q_{k}-1}x_{j_k})	\cdot y_k$
are defined as in Equation \eqref{DeltaN}.

\item[(2)]
If $p<k$, we simply define
\begin{equation}\label{Brace2}
\mu_{k}^{\Delta}(\underline{}{x}\otimes \underline{}{y_1}[1]\otimes \cdots \otimes \underline{}{y_k}[1]): = \underline{}{x}\lb\underline{}{y_1}\, , \cdots\, , \underline{}{y_k}\rb =0\, .	
\end{equation}	
\end{itemize}
\end{itemize}
\end{theorem}

\begin{proof}
These definitions of structure maps are obtained by applying Theorem \ref{Operadtobrace} to the operad $\cP_{\cH}$ and the multiplication $\m=1\otimes 1\in \cP_{\cH}(2)$.
The brace operations
$\mu_k^{\Delta}$ are taken to be $\mu_{k}^{\cP_{\cH}}$.
\end{proof}

Below are particular examples illustrating Theorem \ref{MainTheorembrace}.

\begin{example}
Given that $A$ is a finite-dimensional associative algebra over $\fK$,
it is established in \cite{Lu} that
$\cH = \Hom_{\fK}(A, A)$ possesses a Hopf algebroid structure over $A$, and moreover, for all integers $n \geqslant 0$, there exists an isomorphism
\begin{equation*}
\underbrace{\cH\otimes_{A}\cdots \otimes_{A}\cH}_n\simeq \Hom_{\fK}(A^{\otimes n}, A)\, .	
\end{equation*}
Consequently, by applying Theorem \ref{MainTheorembrace} to the Hopf algebroid $\cH$ over $A$, we deduce that the resulting brace $B_{\infty}$ algebra $B_{\infty}(\cH)$ can be identified with the canonical brace $B_{\infty}$ algebra structure defined on the graded space $C(A, A)$ of Hochschild cochains of $A$, as presented in \cite{Gerstenhaber-Voronov}.
\end{example}

\begin{example}
Let $M$ be a smooth manifold and $A \rightarrow M$ be a Lie $\fK$-algebroid over $M$
(for a precise definition, see \cite{Mackenzie-Book}).
Let $C^{\infty}(M)$ denotes the algebra of $\fK$-valued
smooth functions on $M$.
The universal enveloping algebra $U(A)$ of
the Lie-Rinehart algebra $(C^{\infty}(M), \Gamma(A))$ \cite{Rinehart} is endowed with a canonical
Hopf algebroid structure over
$C^{\infty}(M)$, the source and target maps
are both the natural inclusion $C^{\infty}(M) \hookrightarrow U(A)$
\cites{Xu-Quantum, Xu-QuantumCNRS}.
Applying Theorem \ref{MainTheorembrace} to the Hopf algebroid $U(A)$ over $C^{\infty}(M)$,
we obtain a brace $B_{\infty}$-algebra
\begin{equation*}
B_{\infty}( U(A) )=(D_{\poly}(A)\, , \delta\, , \cup\, ,
\{\mu_{k}^{\Delta}\}_{k\geqslant 0})\, ,
\end{equation*}
where $D_{\poly}(A)$ denotes the graded space
$\oplus_{n=0}^{\infty} (\otimes_{C^{\infty}(M)}^n U(A))[-n]$ \cites{Bandiera-Stienon-Xu, Stienon-Xu-Survey, Calaque-formality}.
This brace $B_{\infty}$-algebra structure plays a crucial role in proving global $G_{\infty}$ formality results for Lie algebroids \cite{Calaque-VandenBergh}.

Now, suppose that $\cA$ is a Lie groupoid integrating $A$. Then $D_{\poly}(A)$ can be identified with the space of left-invariant polydifferential operators on $\cA$.
Under this identification, $D_{\poly}(A)$ together with the canonical Hochschild differential, cup product, and brace operations inherited from the graded space of Hochschild cochains of $\cA$, forms a brace $B_{\infty}$-algebra which coincides with $B_{\infty}(U(A))$.
\footnote{This fact was communicated to us by Ping Xu.}
\end{example}

\begin{remark}
The graded space $\oplus_{n=0}^{\infty}
(\underbrace{\cH\otimes_{R}\cdots\otimes_{R}\cH)}_{n}[-n]$ associated with a Hopf algebroid $\cH$ is used in the cyclic cohomology theory of Hopf algebroids, see \cites{Kowalzig-Posthuma, Kowalzig-Krahmer}.
\end{remark}

\section{Two types of twisted brace $B_{\infty}$ algebras}
Deformation quantization serves as a source of inspiration for the higher structures investigated in this work. Specifically, let $C^{\infty}(M)$ denote the commutative algebra of $\fK$-valued smooth functions on a smooth manifold $M$. Consider its extension $C^{\infty}(M)\llbracket \hbar\rrbracket:=C^{\infty}(M)\otimes \fK\llbracket \hbar\rrbracket$, equipped with a \emph{star product} $\star$ that deforms the standard commutative multiplication on $C^{\infty}(M)$. Within this framework, two relevant brace $B_{\infty}$ algebras emerge.
First, the graded space of Hochschild cochains of the deformed algebra $(C^{\infty}(M)\llbracket \hbar\rrbracket, \star)$ possesses a canonical brace $B_{\infty}$ algebra structure \cite{Gerstenhaber-Voronov}.
Second, consider the cochain complex of polydifferential operators
$(D_{\poly}(M), \delta)$, which is a subcomplex of the Hochschild cochain complex of $C^{\infty}(M)$.
Let $\gamma\in D_{\poly}^2(M)\otimes \hbar\fK\llbracket \hbar\rrbracket$ represent the Maurer-Cartan element associated with the star product $\star$. Then, the deformed differential $\delta + [\gamma, \,\cdot\,]_{\mathrm{G}}$, coupled with Kontsevich's deformed cup product (as detailed in \cite[\S $8.1$]{Kontsevich}), naturally extends to define a brace $B_{\infty}$ structure on
$D_{\poly}(M)\llbracket \hbar \rrbracket:=D_{\poly}(M)\otimes \fK\llbracket \hbar \rrbracket$.

In this section,
we will introduce the \emph{type I twisted} and \emph{type II twisted} brace $B_{\infty}$ algebras.
These algebras are derived from twistors of Hopf algebroids and provide a systematic framework for interpreting the above
two specific brace $B_{\infty}$ algebras that arise within the context of deformation quantization.

\subsection{Twistors}\label{subSec:twistors}
As usual, consider a Hopf algebroid $\cH=(\cH\, , \alpha\, , \beta\, , \cdot\, , \Delta\, , \epsilon)$    over   $R$. For any   element
$\cF=\sum_i (\cF_{1})_i \otimes_{R} (\cF_{2})_i \in \cH\otimes_{R}\cH$,
we adopt Sweedler's abbreviated notation to
write   $\cF=\cF_1\otimes_{R}\cF_2$. From $\cF$, we derive   the following maps:
\begin{align*}
\alpha_{\cF}\colon &R \rightarrow \cH, & \alpha_{\cF}(a)&:=
\alpha(\cF_1\rhd a)\cdot \cF_2 \,\,\, \biggl(=
\sum_i \alpha((\cF_{1})_i\rhd a)\cdot (\cF_{2})_i \biggr)\, , \\
\beta_{\cF}\colon &R \rightarrow \cH, & \beta_{\cF}(a)& :=
\beta(\cF_2 \rhd a) \cdot \cF_1
\,\,\,
\biggl(=\sum_i \beta((\cF_{2})_i\rhd a)\cdot (\cF_{1})_i
\biggr)\, , \\
\star_{\cF}\colon &R \times R \rightarrow R, &
a \star_{\cF} b & :=(\cF_1\rhd a)\cdot (\cF_2 \rhd b)
\,\,\,
\biggl(=\sum_i((\cF_{1})_i\rhd a) \cdot ((\cF_{2})_i\rhd b) \biggr)\, ,
\end{align*}
for all $a, b \in R$.

\begin{theorem}[\cites{Xu-R-matrices, Xu-Quantum}] \label{Pre}
Suppose that $\cF\in \cH\otimes_{R}\cH$ satisfies
\begin{equation}\label{TwistorEquation}
(\Delta\otimes_{R} \id)\cF\cdot (\cF\otimes_{R} 1)= (\id\otimes_{R} \Delta)\cF \cdot (1\otimes_{R} \cF)	 \in \cH\otimes_{R}\cH\otimes_{R}\cH\, ,
\end{equation}
and
\begin{equation}\label{counitEquation}
(\epsilon\otimes_{R}\id)	\cF=(\id\otimes_{R}\epsilon)\cF=1\, .	
\end{equation}
Then
\begin{enumerate}
\item $R_{\cF}:=(R\, , \star_{\cF})$ is an associative $\fK$-algebra  with unit $1$.
\item $\alpha_{\cF}\colon R_{\cF}\rightarrow \cH$ is an algebra homomorphism and $\beta_{\cF}\colon R_{\cF}\rightarrow \cH$ is an algebra anti-homomorphism, satisfying $\alpha_{\cF}(a)\cdot \beta_{\cF}(b) = \beta_{\cF}(b) \cdot \alpha_{\cF}(a)$ and
\begin{equation*}
\cF \cdot (\beta_{\cF}(a)\otimes 1 - 1 \otimes \alpha_{\cF}(a))=0, 	
\end{equation*}
for all $a, b \in R$.
\end{enumerate}

As a consequence, such an element $\cF$ determines a linear map of $\fK$-vector spaces
\begin{align}\label{Eqt:Fsharp}
\cF^{\sharp}\colon M_1\otimes_{R_{\cF}}M_2 & \rightarrow M_1 \otimes_{R} M_2 \notag \\
\cF^{\sharp}(D_1 \otimes_{R_{\cF}} D_2  ) &:=
(\cF_1\cdot D_1)\otimes_{R} (\cF_2\cdot D_2 )
\,\,\,
\biggl(=\sum_i ((\cF_{1})_i\cdot D_1) \otimes_{R} ((\cF_{2})_i\cdot D_2 ) \biggr)\, ,
\end{align}
for any two left $\cH$-modules $M_1$ and $M_2$.
\end{theorem}

\begin{definition}[\cites{Xu-R-matrices, Xu-Quantum}]\label{TwistorDefinition}
An element  $\cF\in \cH\otimes_{R}\cH$ is called a
\emph{twistor} of the Hopf algebroid $\cH$, if it satisfies conditions \eqref{TwistorEquation} and \eqref{counitEquation},  and the associated $\fK$-linear maps $\cF^{\sharp}$ in \eqref{Eqt:Fsharp} are
\textbf{invertible} for any two
left $\cH$-modules $M_1$ and $M_2$.
\end{definition}

A twistor $\cF$ defines a twisted Hopf algebroid
\begin{equation*}
\cH_{\cF}:=(\cH, \alpha_{\cF}, \beta_{\cF}, \cdot, \Delta_{\cF}, \epsilon) 	
\end{equation*}
over $R_{\cF}=(R\, , \star_{\cF})$,
where the associative algebra product $\cdot$ and counit $\epsilon$ are identical to those of the original Hopf algebroid $\cH=(\cH\, , \alpha\, , \beta\, , \cdot\, , \Delta\, , \epsilon)$. The source and target maps are given by $\alpha_{\cF}$ and $\beta_{\cF}$, respectively, while the coproduct $\Delta_{\cF}$ is defined by
\begin{equation}\label{NewCoProduct}
	\Delta_{\cF}(x):=(\cF^{\sharp})^{-1}( \Delta(x)\cdot \cF )\, , 	
\end{equation}
for all $x \in \cH$.

\begin{remark}
A certain case is noteworthy.
When $\cH=\cD\llbracket \hbar \rrbracket:=\cD\otimes\fK\llbracket \hbar\rrbracket$, where $\cD$ is the Hopf algebroid of smooth differential operators on a manifold $M$ and $R=C^{\infty}(M)\llbracket \hbar \rrbracket$, and given a twistor $\cF\in \cH\otimes_{R}\cH=\cD \otimes_{C^{\infty}(M)} \cD \llbracket  \hbar \rrbracket$, the source and target maps $\alpha_{\cF}, \beta_{\cF} \colon R_{\cF}\rightarrow \cH$ within the twisted Hopf algebroid $\cH_{\cF}=(\cH, \alpha_{\cF}, \beta_{\cF}, \cdot, \Delta_{\cF}, \epsilon)$ are related to the product $\star_{\cF}$ via the following expressions:
\begin{align*}\label{NewSourceTarget}
\alpha_{\cF}(f)(g) = f\star_{\cF}g, \quad  \beta_{\cF}(f)(g) = g\star_{\cF}f,
\end{align*}
for all $f, g \in C^{\infty}(M)\llbracket \hbar \rrbracket$.	
\end{remark}

\subsection{Some facts about $\cup_{\cF}$ and $\cF^{\sharp}$}
Let $\cF$ be a twistor for a Hopf algebroid $\cH$ over $R$. The twisted Hopf algebroid $\cH_{\cF}$ is an $(R_{\cF}, R_{\cF})$-bimodule, allowing us to define the graded space
\[
W_{\cH_{\cF}} := \bigoplus_{n=0}^{\infty} (\underbrace{\cH_{\cF} \otimes_{R_{\cF}} \cdots \otimes_{R_{\cF}} \cH_{\cF}}_{n})[-n].
\]
We define the product on $W_{\cH_{\cF}}$ by $x \cup y := (-1)^{|x||y|} x \otimes_{R_{\cF}} y$ for homogeneous elements $x, y \in W_{\cH_{\cF}}$.  Then $(W_{\cH_{\cF}}, \cup)$ is a graded associative algebra that can be identified with the free algebra generated by $\cH_{\cF}[-1]$ over $R_{\cF}$.

Consider also the (non-twisted) graded space
\[
W_{\cH} := \bigoplus_{n=0}^{\infty} (\underbrace{\cH \otimes_{R} \cdots \otimes_{R} \cH}_{n})[-n],
\]
equipped with the twisted product defined by
\[
x \cup_{\cF} y := (-1)^{|x|}\cF\lb x, y \rb,
\]
for all $x, y \in W_{\cH}$. By Equations \eqref{TwistorEquation} and \eqref{counitEquation}, $(W_{\cH}, \cup_{\cF})$ forms a graded associative algebra with unit $1 \in W_{\cH}$.

Let $T(\cH[-1])$ denote the free tensor algebra generated by $\cH[-1]$ over $\fK$. We have two natural maps:
\begin{itemize}
    \item A quotient map
    \[
    p: T(\cH[-1]) \rightarrow W_{\cH_{\cF}}
    \]
    that maps $x_1[-1] \otimes \cdots \otimes x_n[-1]$ to $x_1 \cup \cdots \cup x_n = (-1)^{\frac{n(n-1)}{2}} x_1 \otimes_{R_{\cF}} \cdots \otimes_{R_{\cF}} x_n \in W_{\cH_{\cF}}^n$ for all integers $n \geq 0$ and all $x_1, \ldots, x_n \in \cH$.
    \item A morphism of graded associative algebras
    \[
    q: T(\cH[-1]) \rightarrow W_{\cH}
    \]
    that maps $x_1[-1] \otimes \cdots \otimes x_n[-1]$ to $x_1 \cup_{\cF} \cdots \cup_{\cF} x_n \in W_{\cH}^n$ for all integers $n \geq 0$ and all $x_1, \ldots, x_n \in \cH$.
\end{itemize}

\begin{lemma}\label{KeyM1}
The $\fK$-linear map
\begin{equation*}
\cF^{\sharp}: W_{\cH_{\cF}} \rightarrow W_{\cH}, \qquad \cF^{\sharp}(x_1 \cup \cdots \cup x_n) := x_1 \cup_{\cF} \cdots \cup_{\cF} x_n\, , 	
\end{equation*}
for all integers $n\geqslant 0$ and all $x_1\, , \cdots\, , x_n \in \cH$,
is well-defined and intertwines the products $\cup$ and $\cup_{\cF}$. Moreover, the following diagram of graded associative algebras commutes:
\begin{equation*}
\begin{tikzcd}
T(\cH[-1]) \arrow[d, "p"]
\arrow[r, "q"] & (W_{\cH}, \cup_{\cF}) \\
(W_{\cH_{\cF}}, \cup)
\arrow[ru, dashed, " \cF^{\sharp} "'] &
\end{tikzcd}
\end{equation*}
\end{lemma}
\begin{proof}
It suffices to verify that $q(\Ker p) = 0$, which is guaranteed by statement (2) of Theorem \ref{Pre}.
\end{proof}
Since $x_1\otimes_{R_{\cF}}\cdots \otimes_{R_{\cF}} x_n =
(-1)^{\frac{n(n-1)}{2}}x_1\cup \cdots \cup x_n$
for all $x_1\, , \cdots\, , x_n \in \cH$,
we have
\begin{equation}\label{SignFixed}
\begin{split}
\cF^{\sharp}(x_1\otimes_{R_{\cF}}\cdots \otimes_{R_{\cF}} x_n) & =(-1)^{\frac{n(n-1)}{2}}
x_1\cup_{\cF} \cdots \cup_{\cF} x_n	\\
& \,\, =
\cF\lb \cF \lb  \cdots \cF\lb \cF\lb x_1 \, , x_2\rb\, , x_3 \rb\, , \cdots \rb\, , x_n \rb \\
& \,\,\,\, = ( \cF_{1 \tinywedge \cdots \tinywedge n-1, n}\boldsymbol{\cdot} \cdots  \boldsymbol{\cdot} \cF_{1 \tinywedge 2,3}\boldsymbol{\cdot} \cF_{1,2})
\cdot
(x_1 \otimes_{R} \cdots \otimes_{R} x_n)\, .
\end{split}
\end{equation}
In the above formula, we adopt the notation
$\cF_{1 \tinywedge \cdots \tinywedge k-1, k}:=
(\Delta^{k-2}\otimes_{R} \id)\cF$.
\footnote{We use the subscript sequence
$1\tinywedge 2 \tinywedge \cdots \tinywedge k-2 \tinywedge k-1$ to make the notation transparent.}
In particular, we have
$\cF^{\sharp}(x_1\otimes_{R_{\cF}}x_2)=\cF\lb x_1\, , x_2\rb= (\cF_1 \cdot x_1) \otimes_R (\cF_2 \cdot x_2)$,
which coincides with the expression in \eqref{Eqt:Fsharp}.

\begin{lemma}\label{KeyM2}
For any $y_1\, , \cdots\, , y_n\in W_{\cH_{\cF}}$
with degrees $|y_1|=k_1$, $\cdots$, $|y_n|=k_n$, we have
\begin{equation*}
\begin{split}
& \cF^{\sharp}(y_1\otimes_{R_{\cF}} \cdots \otimes_{R_{\cF}} y_n) \\
= &
(\Delta^{k_1-1}\otimes_{R}\cdots \otimes_{R} \Delta^{k_n -1})( \cF_{1 \tinywedge \cdots \tinywedge n-1, n}\boldsymbol{\cdot} \cdots  \boldsymbol{\cdot} \cF_{1 \tinywedge 2,3}\boldsymbol{\cdot} \cF_{1,2})
\cdot
(\cF^{\sharp}y_1 \otimes_{R} \cdots \otimes_{R} \cF^{\sharp}y_n)\, .
\end{split}
\end{equation*}
\end{lemma}
\begin{proof}
By Lemma \ref{KeyM1} and Equation \eqref{SignFixed},
we have
\begin{equation*}
\begin{split}
& \cF^{\sharp}(y_1\otimes_{R_{\cF}} \cdots \otimes_{R_{\cF}} y_n) \\
= &
\cF\lb  \cdots \cF\lb \cF\lb \cF^{\sharp}y_1 \, , \cF^{\sharp}y_2\rb\, , \cF^{\sharp}y_3 \rb\, , \cdots \rb\, , \cF^{\sharp}y_n \rb	\\
= &
(\Delta^{k_1-1}\otimes_{R}\cdots \otimes_{R} \Delta^{k_n -1})\cF_{1 \tinywedge \cdots \tinywedge n-1, n}
\boldsymbol{\cdot}
\cdots
\boldsymbol{\cdot}
(\Delta^{k_1-1}\otimes_{R}\Delta^{k_2-1})\cF_{1,2}
\cdot
(\cF^{\sharp}y_1 \otimes_{R} \cdots \otimes_{R} \cF^{\sharp}y_n) \\
= &
(\Delta^{k_1-1}\otimes_{R}\cdots \otimes_{R} \Delta^{k_n -1})( \cF_{1 \tinywedge \cdots \tinywedge n-1, n}\boldsymbol{\cdot} \cdots  \boldsymbol{\cdot} \cF_{1 \tinywedge 2,3}\boldsymbol{\cdot} \cF_{1,2})
\cdot
(\cF^{\sharp}y_1 \otimes_{R} \cdots \otimes_{R} \cF^{\sharp}y_n)\, .
\end{split}
\end{equation*}
\end{proof}

\begin{lemma}\label{KeyM3}
For any $h\in \cH_{\cF}$	 and $u\in \underbrace{\cH_{\cF}\otimes_{R_{\cF}}\cdots \otimes_{R_{\cF}}\cH_{\cF}}_k$, we have
\begin{equation*}
\cF^{\sharp}\bigl( (\Delta_{\cF}^{k-1}h) \cdot u \bigr)=
(\Delta^{k-1}h) \cdot (\cF^{\sharp}u)\, .	
\end{equation*}
\end{lemma}
\begin{proof}
We first verify that
\begin{equation}\label{EquationBk}
  \cF^{\sharp}(\Delta_{\cF}^{k-1}h)=(\Delta^{k-1}h)\cdot( \cF_{1 \tinywedge \cdots \tinywedge k-1, k}\boldsymbol{\cdot} \cdots  \boldsymbol{\cdot} \cF_{1\tinywedge 2,3}\boldsymbol{\cdot} \cF_{1,2})
\end{equation}
holds for all $k\geqslant 2$.
It is clear the equation  holds for $k=2$ by the definition of
$\Delta_{\cF}$ (see Equation \eqref{NewCoProduct}).

Assume that \eqref{EquationBk} is verified up to $k=n-1$.
Write $\Delta_{\cF}h=h_1\otimes_{R_{\cF}}h_2$.  We have
\begin{align*}
\cF^{\sharp}(\Delta_{\cF}^{n-1}h) & =
\cF^{\sharp}\biggl(
(\Delta_{\cF}^{n-2}\otimes_{R_{\cF}} \id)(\Delta_{\cF}h)	
\biggr)	\\
& = \cF^{\sharp}( \Delta_{\cF}^{n-2}h_1 \otimes_{R_{\cF}} h_2 ) \\
& = \cF_{1 \tinywedge \cdots \tinywedge n-1, n} \cdot
\bigl(
\cF^{\sharp}( \Delta_{\cF}^{n-2}h_1)\otimes_{R}h_2
\bigr) \qquad \text{(by Lemma \ref{KeyM2})} \\
& = (\Delta^{n-2}\otimes_{R} \id)\cF \cdot
\bigl(
\Delta^{n-2}h_1 \cdot
( \cF_{1 \tinywedge \cdots \tinywedge n-2, n-1}\boldsymbol{\cdot} \cdots  \boldsymbol{\cdot} \cF_{1\tinywedge 2,3}\boldsymbol{\cdot} \cF_{1,2}
)\otimes_{R} h_2
\bigr) \\
& \qquad\qquad\qquad \text{(by the induction assumption)} \\
& = (\Delta^{n-2}\otimes_{R} \id)\bigl( \cF \cdot (h_1\otimes_{R}h_2) \bigr)
\cdot
( \cF_{1 \tinywedge \cdots \tinywedge n-2, n-1}\boldsymbol{\cdot} \cdots  \boldsymbol{\cdot} \cF_{1 \tinywedge 2,3}\boldsymbol{\cdot} \cF_{1,2}
) \\
& = (\Delta^{n-2}\otimes_{R} \id)\bigl( \Delta(h) \cdot \cF  \bigr)
\cdot
( \cF_{1 \tinywedge \cdots \tinywedge n-2, n-1}\boldsymbol{\cdot} \cdots  \boldsymbol{\cdot} \cF_{1 \tinywedge 2,3}\boldsymbol{\cdot} \cF_{1,2})
\\
& \qquad\qquad\qquad \text{(by Equation \eqref{NewCoProduct})} \\
& = (\Delta^{n-1}h) \cdot ( \cF_{1 \tinywedge \cdots \tinywedge n-1, n}\boldsymbol{\cdot} \cdots  \boldsymbol{\cdot} \cF_{1 \tinywedge 2,3}\boldsymbol{\cdot} \cF_{1,2})\, .
\end{align*}
This proves Equation \eqref{EquationBk} for all $k\geqslant 2$. Then for any $u=u_1\otimes_{R_{\cF}}\cdots \otimes_{R_{\cF}} u_k$, where $u_1\, , \cdots\, , u_k\in \cH_{\cF}$, we have
\begin{equation*}
\begin{split}
\cF^{\sharp}\bigl( (\Delta_{\cF}^{k-1}h) \cdot u \bigr) & =
\cF^{\sharp}(\Delta_{\cF}^{k-1}h)
\cdot
(u_1\otimes_{R}\cdots \otimes_{R} u_k) \\
&  \qquad \text{(by Lemma \ref{KeyM2})}\\
& = (\Delta^{k-1}h) \cdot ( \cF_{1 \tinywedge \cdots \tinywedge k-1, k}\boldsymbol{\cdot} \cdots  \boldsymbol{\cdot} \cF_{1 \tinywedge 2,3}\boldsymbol{\cdot} \cF_{1,2})
\cdot
(u_1\otimes_{R}\cdots \otimes_{R} u_k)\,  \\
& \qquad \text{(by Equation \eqref{EquationBk})} \\
& = (\Delta^{k-1}h)\cdot \cF^{\sharp}u\, \\
& \qquad \text{(by Lemma \ref{KeyM2}).}
\end{split}
\end{equation*}

This accomplishes the proof.
\end{proof}

\subsection{Isomorphism of  operads $\cP_{\cH_{\cF}}$ and $\cP_{\cH}$}

By Proposition \ref{MainTheoremOperad},
we obtain two non-symmetric operads
$\cP_{\cH_{\cF}}$ and $\cP_{\cH}$ associated with
two Hopf algebroids $\cH_{\cF}$ and $\cH$, respectively.
By Equations \eqref{Brace1} and \eqref{TwistorEquation},
the twistor $\cF$ is indeed a multiplication on the operad
$\cP_{\cH}$.
By Lemma \ref{KeyM1} and Equation \eqref{SignFixed},
the $\fK$-linear map
\begin{equation*}
\cF^{\sharp}_n: \cP_{\cH_{\cF}}(n) \rightarrow
\cP_{\cH}(n)\, , \qquad 	
\cF^{\sharp}_n(x_1\otimes_{R_{\cF}}\cdots \otimes_{R_{\cF}} x_n):=
\cF\lb \cF \lb  \cdots \cF\lb \cF\lb x_1 \, , x_2\rb\, , x_3 \rb\, , \cdots \rb\, , x_n \rb\, ,
\end{equation*}
is well-defined. An important fact is the following proposition.

\begin{proposition}\label{TwoOperads}
The $\fK$-linear maps
$\{ \cF_{n}^{\sharp} \}_{n\geqslant 0}$
constitutes an isomorphism of operads
\begin{equation*}
\cF^{\sharp}: \cP_{\cH_{\cF}} \xlongrightarrow{\simeq}
\cP_{\cH}\, . 	
\end{equation*}
Moreover, $\cF^{\sharp}$ sends the multiplication $1\otimes_{R_{\cF}}1\in \cP_{\cH_{\cF}}(2)$
to the multiplication $\cF\in \cP_{\cH}(2)$.
\end{proposition}
\begin{proof}
By the definition of a twistor, $\cF_n^{\sharp}$ are isomorphisms of $\fK$-vector spaces.
Clearly, we have
$\cF^{\sharp}(1\otimes_{R_{\cF}}1)=\cF$ and
$\cF^{\sharp}(1)=1$. So we only need to   verify that $\cF^{\sharp}$ intertwines the relevant operadic compositions.
For any $h=h_1\otimes_{R_{\cF}}\cdots \otimes_{R_{\cF}} h_n\in \cP_{\cH_{\cF}}(n)$ and
$u_1\in \cP_{\cH_{\cF}}(k_1)$, $\cdots$,
$u_n\in \cP_{\cH_{\cF}}(k_n)$,  we have
\begin{equation*}
\begin{split}
& \cF^{\sharp}\gamma(h; u_1\, , \cdots\, , u_n) \\
= & \cF^{\sharp}\biggl(
(\Delta_{\cF}^{k_1-1}h_1\cdot u_1)\otimes_{R_{\cF}}\cdots
\otimes_{R_{\cF}} (\Delta_{\cF}^{k_n-1}h_n\cdot u_n)
\biggr) \\
= & (\Delta^{k_1-1}\otimes_{R}\cdots \otimes_{R} \Delta^{k_n -1})( \cF_{1 \tinywedge \cdots \tinywedge n-1, n}\boldsymbol{\cdot} \cdots  \boldsymbol{\cdot} \cF_{1 \tinywedge 2,3}\boldsymbol{\cdot} \cF_{1,2})
\boldsymbol{\cdot}\\
& \qquad\qquad \bigl(\cF^{\sharp}(\Delta_{\cF}^{k_1-1}h_1\cdot u_1) \otimes_{R} \cdots \otimes_{R}
\cF^{\sharp}(\Delta_{\cF}^{k_n-1}h_n\cdot u_n)
\bigr)
\qquad \text{(by Lemma \ref{KeyM2})}\\
= & (\Delta^{k_1-1}\otimes_{R}\cdots \otimes_{R} \Delta^{k_n -1})( \cF_{1 \tinywedge \cdots \tinywedge n-1, n}\boldsymbol{\cdot} \cdots  \boldsymbol{\cdot} \cF_{1 \tinywedge 2,3}\boldsymbol{\cdot} \cF_{1,2})
\boldsymbol{\cdot}\\
& \qquad\qquad \bigl(
(\Delta^{k_1-1}h_1\cdot \cF^{\sharp} u_1) \otimes_{R} \cdots \otimes_{R}
(\Delta^{k_n-1}h_n\cdot \cF^{\sharp} u_n)
\bigr)
\qquad \text{(by Lemma \ref{KeyM3})} \\
= & (\Delta^{k_1-1}\otimes_{R}\cdots \otimes_{R} \Delta^{k_n -1})
\biggl(
( \cF_{1 \tinywedge \cdots \tinywedge n-1, n}\boldsymbol{\cdot} \cdots  \boldsymbol{\cdot} \cF_{1 \tinywedge 2,3}\boldsymbol{\cdot} \cF_{1,2})
\boldsymbol{\cdot}
(h_1\otimes_{R}\cdots \otimes_{R} h_n)
\biggr)
\boldsymbol{\cdot} \\
& \qquad\qquad
(\cF^{\sharp}u_1\otimes_{R}\cdots \otimes_{R} \cF^{\sharp}u_n)	  \\
= &
(\Delta^{k_1-1}\otimes_{R}\cdots \otimes_{R} \Delta^{k_n -1})(\cF^{\sharp}h)\cdot (\cF^{\sharp}u_1\otimes_{R}\cdots \otimes_{R} \cF^{\sharp}u_n)
\qquad \text{(by Lemma \ref{KeyM2})} \\
= & \gamma(\cF^{\sharp}h; \cF^{\sharp}u_1\, , \cdots\, , \cF^{\sharp}u_n)\, .
\end{split}
\end{equation*}

This accomplishes the proof.
\end{proof}

\subsection{The strict $B_{\infty}$ isomorphism}

\begin{definition}
With notations as established previously.
Associated with the operad $\cP_{\cH_{\cF}}$, we call
	 \begin{equation*}
	 	\BInfty(\cH_{\cF}):=
	 	\bigl(\bigoplus_{n=0}^{\infty}
	 	(\underbrace{\cH_{\cF}\otimes_{R_{\cF}}\cdots\otimes_{R_{\cF}}\cH_{\cF})}_{n}[-n]\, , \delta_{\cH_{\cF}} \, , \cup\, ,
	 	\{\mu_{k}^{\Delta_{\cF}}\}_{k\geqslant 0}
	 	\bigr)\, ,
	 \end{equation*}	
the \emph{type I twisted} brace $B_{\infty}$ algebra.
Associated with the operad $\cP_{\cH}$ and
the twistor $\cF$ of $\cH$, we call
\begin{equation*}
	  	\BInfty(\cH)^{\cF}:=\bigl(\bigoplus_{n=0}^{\infty}
	  	(\underbrace{\cH\otimes_{R}\cdots\otimes_{R}\cH}_{n})[-n]\, ,
	  	[\underline{}{\cF}\, , \,\cdot\,]_{\mathrm{G}} \, , \cup_{\cF}\, , \{\mu_{k}^{\Delta}\}_{k\geqslant 0}
	  	\bigr)\, ,
	  \end{equation*}
the \emph{type II twisted} brace $B_{\infty}$ algebra.
\end{definition}

Applying Proposition \ref{functoriality} to the operad isomorphism
$\cF^{\sharp}: \cP_{\cH_{\cF}} \xrightarrow{\simeq}
\cP_{\cH}$ established in Proposition \ref{TwoOperads}, we obtain the following result:

\begin{theorem}\label{TwoTypes}
The $\fK$-linear map
\begin{equation*}
\begin{split}
\cF^{\sharp}: \BInfty{(\cH_{\cF})}
& \xrightarrow{\simeq} \BInfty{(\cH)}^{\cF}\, , \\
x_1\otimes_{R_{\cF}}\cdots \otimes_{R_{\cF}}x_n
& \mapsto
\cF\lb \cF \lb  \cdots \cF\lb \cF\lb x_1 \, , x_2\rb\, , x_3 \rb\, , \cdots \rb\, , x_n \rb\, ,
\end{split}
\end{equation*}
defined for all integers $n\geqslant 0$ and elements $x_1\, , \cdots\, , x_n\in \cH$,
is a strict $B_{\infty}$ isomorphism between the type I and type II twisted brace $B_{\infty}$ algebras.
\end{theorem}

In the specific case where the Hopf algebroid degenerates to a bialgebra $H$ (over $\fK$), its twistors correspond precisely to Drinfeld's twists, denoted by $J\in H\otimes_{\fK} H$. Theorem \ref{TwoTypes} then provides an isomorphism of brace $B_{\infty}$ algebras, given by
$J^{\sharp}: \BInfty(H_{J})\rightarrow \BInfty(H)^{J}$.
Notably, since the map $J^{\sharp}$ induces an isomorphism of dg Lie algebras (refer to Remark \ref{dgLie}), this result recovers the finding of Esposito and de Kleijn presented in \cite[Proposition $4.9$]{Esposito-de Kleijn}.

\subsection{Brace $B_{\infty}$ algebras within deformation quantization}

Hereafter, we fix the field $\fK$ to be either $\fR$ or $\fC$. Let $M$ be a smooth manifold, and denote by $C^{\infty}(M)$ the algebra of smooth $\fK$-valued functions on $M$.

Consider a completed associative $\fK\llbracket \hbar \rrbracket$-algebra $R$ with unit, such that $R/\hbar R = C^{\infty}(M)$.
Let $\cH=(\cH\, , \alpha\, , \beta\, , \cdot_{\hbar}\, , \Delta\, , \epsilon)$ be a topological Hopf algebroid over $R$. This signifies that $\cH$ is a completed $\fK\llbracket \hbar \rrbracket$-module, and the maps $\alpha$, $\beta$, $\cdot_{\hbar}$, $\Delta$, and $\epsilon$ are all continuous with respect to the $\hbar$-adic topology.

\begin{definition}[\cites{Xu-Quantum, Xu-QuantumCNRS}]
A topological Hopf algebroid
$\cH=(\cH\, , \alpha\, , \beta\, , \cdot_{\hbar}\, , \Delta\, , \epsilon)$ is called a quantum groupoid,
if the induced Hopf algebroid $\cH/\hbar \cH$ over $C^{\infty}(M)$ is isomorphic to the universal enveloping algebra $U(A)$ of some Lie $\fK$ algebroid $A\rightarrow M$.
\end{definition}
For instance, the Hopf algebroid structure on $\cD=U(TM)$ over $C^{\infty}(M)$ extends to a quantum groupoid structure on $U(TM)\llbracket \hbar \rrbracket:= U(TM)\otimes \fK\llbracket \hbar \rrbracket$ over $C^{\infty}(M)\llbracket \hbar \rrbracket$.
Let us consider an element
\begin{equation*}
\cF=1\otimes_{R}1+ \sum_{k=1}^{\infty} \hbar^{k} B_{k} \in U(TM)\otimes_{C^{\infty}(M)}U(TM)\llbracket \hbar \rrbracket = U(TM)\llbracket \hbar \rrbracket \otimes_{C^{\infty}(M)\llbracket \hbar \rrbracket} U(TM)\llbracket \hbar \rrbracket\, , 	
\end{equation*}
where $\cF$ is a formal power series in $\hbar$ and the coefficients $B_k$ are bi-differential operators on $M$. Then, $\cF$ is a twistor of $\cH$ if and only if
the product on $R$, defined by
\begin{equation*}
f \star_{\cF} g:=\cF(f\, , g)\, , 	
\end{equation*}
for all $f\, , g\in C^{\infty}(M)\llbracket \hbar \rrbracket$,
is associative with $1$ as its unit \cite{Xu-Quantum, Xu-QuantumCNRS}.
Indeed, $\star_{\cF}$ is a star product on $M$, and conversely, every star product arises from some twistor $\cF$.

Having introduced the two brace $B_{\infty}$ algebras within the context of deformation quantization in the beginning of this section,
we now interpret them within the framework of type I and type II twisted brace $B_{\infty}$ algebras.

The type I twisted brace $B_{\infty}$ algebra $B_{\infty}(U(TM)\llbracket \hbar \rrbracket_{\cF})$ is the canonical brace $B_{\infty}$ algebra associated with the twisted quantum groupoid $U(TM)\llbracket \hbar \rrbracket_{\cF}$ over $(C^{\infty}(M)\llbracket \hbar\rrbracket\, , \star_{\cF})$.
Given the left action of $U(TM)\llbracket \hbar \rrbracket_{\cF}$ on $(C^{\infty}(M)\llbracket \hbar\rrbracket\, , \star_{\cF})$ defined in Equation \eqref{Action}, we can construct an embedding of the brace $B_{\infty}$ algebra $B_{\infty}(U(TM)\llbracket \hbar \rrbracket_{\cF})$ into the canonical brace $B_{\infty}$ algebra of Hochschild cochains of the deformation quantized algebra $(C^{\infty}(M)\llbracket \hbar\rrbracket\, , \star_{\cF})$.

The type II twisted brace $B_{\infty}$ algebra associated with $U(TM)\llbracket \hbar \rrbracket$ and its twistor $\cF$ is given by
\begin{equation*}
B_{\infty}(U(TM)\llbracket \hbar \rrbracket)^{\cF}
=\bigl(
D_{\poly}(M)\llbracket \hbar \rrbracket\, ,
[\underline{}{\cF}\, , \,\cdot\,]_{\mathrm{G}} \, , \cup_{\cF}\, , \{\mu_{k}^{\Delta}\}_{k\geqslant 0}
\bigr)\, , 	
\end{equation*}
where $D_{\poly}(M)=\oplus_{n=0}^{\infty}
(\otimes_{C^{\infty}(M)}^n U(TM))[-n]$ denotes the graded space of polydifferential operators on $M$.
We note that the deformed cup product $\cup_{\cF}$ is defined in Kontsevich's work \cite[\S $8.1$]{Kontsevich} which provides a solution to the deformation quantization problem.
Furthermore, the dg Lie algebra associated with the brace $B_{\infty}$ algebra $B_{\infty}(U(TM)\llbracket \hbar \rrbracket)^{\cF}$ governs deformations of the star product $\star_{\cF}$.

Applying Theorem \ref{TwoTypes} to the quantum groupoid $U(TM)\llbracket \hbar \rrbracket $ and the twistor $\cF$, we obtain the following result.

\begin{corollary}
Under the aforementioned assumptions, for any twistor
$\cF$ of the quantum groupoid $U(TM)\llbracket \hbar \rrbracket$, there exists a strict $B_{\infty}$ isomorphism
\begin{equation*}
\cF^{\sharp}:
\BInfty(U(TM)\llbracket \hbar \rrbracket_{\cF})
\xlongrightarrow{\simeq}
\bigl(
D_{\poly}(M)\llbracket \hbar \rrbracket\, ,
[\underline{}{\cF}\, , \,\cdot\,]_{\mathrm{G}} \, , \cup_{\cF}\, , \{\mu_{k}^{\Delta}\}_{k\geqslant 0}
\bigr)\, .	
\end{equation*}	
Notably, $\cF^{\sharp}$ induces an isomorphism of the underlying dg Lie algebras on both sides.
\end{corollary}

\section{Two operads $\cP_{(\frakg, \frakl)}$ and
$\cP_{\mathscr{H}_{(\frakg, \frakl)}}$}

Given a finite-dimensional Lie algebra $\frakg$ over a field $\fK$ and a Lie subalgebra $\frakl \subset \frakg$, we refer to the pair $(\frakg, \frakl)$ as a \textbf{Lie algebra pair}.
The aim of this section is to investigate
several objects associated with such a Lie algebra pair.
In \cite{Calaque-QuantizationFormal},
a brace $B_{\infty}$ algebra $\ADT$ is assigned to $(\frakg\, , \frakl)$, whose underlying dg Lie algebra governs the deformations of algebraic dynamical twists.
Meanwhile, we consider a
specific quantum groupoid $\mathscr{H}_{(\frakg, \frakl)}$ assigned to $(\frakg\, , \frakl)$.
According to Sections \ref{OperadHopf} and \ref{BraceHopf},
$\mathscr{H}_{(\frakg, \frakl)}$ gives rise to the non-symmetric operad $\cP_{\mathscr{H}_{(\frakg, \frakl)}}$ and the canonical brace $B_{\infty}$ algebra $B_{\infty}(\mathscr{H}_{(\frakg, \frakl)})$.
However, as pointed out in Section \ref{Reason},
$W_{(\frakg, \frakl)}$ cannot, in general, be expressed as the canonical brace $B_{\infty}$ algebra of any Hopf algebroid.
To address this, we construct a special non-symmetric operad
$\cP_{(\frakg, \frakl)}$ and a morphism of operads
$\mathsf{c}: \cP_{(\frakg, \frakl)}\rightarrow \cP_{\mathscr{H}_{(\frakg, \frakl)}}\Laurent{\hbar}$.
These, along with the operad $\cP_{\mathscr{H}_{(\frakg, \frakl)}}\Laurent{\hbar}$, are designed to establish an embedding of the brace $B_{\infty}$ algebra $W_{(\frakg,\frakl)}$ into the brace $B_{\infty}$ algebra
$B_{\infty}(\mathscr{H}_{(\frakg, \frakl)})\Laurent{\hbar}$,
via the Gerstenhaber-Voronov operadic modeling of brace $B_{\infty}$ algebras as detailed in Section \ref{Subsec:GVoperadicmodeling}.
Consequently, we relate the deformation theory of algebraic dynamical twists to the canonical brace $B_{\infty}$ algebra associated with the quantum groupoid $\mathscr{H}_{(\frakg, \frakl)}$.

\subsection{The brace
	$B_{\infty}$ algebra $\ADT$}\label{Subsec:AgebraADT}
 An element $A\in U(\frakg)^{\otimes n}\otimes U(\frakl)\llbracket \hbar \rrbracket$
is said to be $\frakl$-invariant, if
$\Delta^{n}(l)\cdot A=A \cdot \Delta^{n}(l)$ for all
$l\in U(\frakl)$.
The subspace of $\frakl$-invariant elements
in $U(\frakg)^{\otimes n}\otimes U(\frakl)\llbracket \hbar \rrbracket$
is denoted by
$(U(\frakg)^{\otimes n}\otimes U(\frakl)\llbracket \hbar \rrbracket)^{\frakl}$.

Consider the graded space
\begin{equation*}
\ADT:=
\oplus_{n=0}^{\infty}
\bigl( U(\frakg)^{\otimes n}\otimes
U(\frakl)\llbracket \hbar \rrbracket
\bigr)^{\frakl}[-n]\, .
\end{equation*}
A homogeneous element $A\in \ADT$ has degree $n$ if it is an honest element of $\bigl( U(\frakg)^{\otimes n}\otimes
U(\frakl)\llbracket \hbar \rrbracket
\bigr)^{\frakl}$.
 We also use $A[1]$ to denote the copy of $A$ in
$\ADT[1]=\oplus_{n}\bigl( U(\frakg)^{\otimes n}\otimes
U(\frakl)\llbracket \hbar \rrbracket
\bigr)^{\frakl}[1-n]$.

A collection of brace operations on $\ADT$ is introduced
in \cite{Calaque-QuantizationFormal}
for every integer $k\geqslant 0$:
\begin{equation*}
\mu_{k}^{(\frakg, \frakl)}: \ADT \bigotimes \otimes^k(\ADT[1])\rightarrow \ADT \, .	
\end{equation*}
Let us now recall the relevant construction.
Given an element $A\in U(\frakg)^{\otimes n}\otimes U(\frakl)\llbracket \hbar \rrbracket$,
we use the following notation\footnote{It was written as
$A^{a, a+1\, , \cdots\, ,  j_1 \cdots j_{1}+q_{1}-1 \, , \cdots\, ,  j_{k}\cdots j_{k}+q_{k}-1\, , \cdots\, , b}$ in \cite{Calaque-QuantizationFormal}.
In order to make the notation more transparent,
here we use subscript sequences like
$j_1 \tinywedge j_1 +1 \tinywedge \cdots \tinywedge j_1+q_1-2 \tinywedge j_1+q_1-1$ instead .}:
 \begin{equation}\label{ImportantConvention}
 	\begin{split}
 		& A_{a, a+1\, , \cdots\, ,  j_1 \tinywedge j_1+1  \tinywedge  \cdots  \tinywedge  j_{1}+q_{1}-1 \, , \cdots\, ,  j_{k} \tinywedge j_k+1 \tinywedge \cdots \tinywedge   j_{k}+q_{k}-1\, , \cdots\, , b} \\ &:=
        \underbrace{\id\otimes \cdots \otimes \id \otimes
 		\Delta^{q_{1}-1} \otimes \id \otimes \cdots \otimes \id
 		\otimes
 		\Delta^{q_{k}-1}
 		\otimes
 		\id \otimes \cdots \otimes \id}_{b}
 		\biggr(
 		\underbrace{1\otimes \cdots \otimes 1 }_{a-1} \otimes A
 		\biggr)\, \\
 		&\quad (\in U(\frakg)^{\otimes b}\otimes U(\frakl)\llbracket \hbar \rrbracket).
 	\end{split}	
 \end{equation}
Here $a\geqslant 1$, $q_1\geqslant 0$, $\cdots$,
$q_k\geqslant 0$, and $j_1\, , \cdots\, ,  j_k$ are integers subject to
$a\leqslant j_1$, $j_1+q_1\leqslant j_2$, $\cdots$,
$j_{k-1}+q_{k-1}\leqslant j_k$, $j_k+q_{k}-1 \leqslant b$ with  $b:=a+n+\sum_{s=1}^k (q_{s} -1)$; the operation $\Delta^{q_s-1}$ is placed at the $i_s$-th position;  when $q_s=0$, $\Delta^{-1}$ degenerates to
the counit map $\epsilon: U(\frakg)\rightarrow \fK$.

The said brace operations are given as follows. For all
$A\in \ADT^p$, $B_1\in \ADT^{q_1}$, $\cdots$, $B_k\in \ADT^{q_k}$,   $p\geqslant k$,  we set
\begin{equation*}\label{ADTBraceI}
\begin{split}
& A\lb B_1\, , \cdots\, , B_{k} \rb
=
\mu_{k}^{(\frakg, \frakl)}(A\, , B_1[1]\, , \cdots\, , B_k[1]) := \!\!\!\!\!
\sum_{ \substack{ 1\leqslant j_{1},  j_{k}+q_{k}-1\leqslant n \\ j_{s}+q_{s}\leqslant j_{s+1}}}
(-1)^{\sum_{s=1}^{k} (q_{s}-1)(j_{s} - 1)}
\\
& \qquad\qquad\quad
A_{1\, , \cdots\, ,  j_1 \tinywedge \cdots \tinywedge  j_{1}+q_{1}-1 \, , \cdots\, ,  j_{k} \tinywedge \cdots \tinywedge  j_{k}+q_{k}-1 \, , \cdots\, , n+1}
\boldsymbol{\cdot}
\Pi_{s=1}^{k}
(B_{s})_{j_{s}, \cdots, j_{s}+q_{s}-1,
j_{s}+q_{s} \tinywedge\cdots \tinywedge n+1 }\, ,
\end{split}
\end{equation*}
where $n=p+\sum_{s=1}^k (q_{s}-1)$. If it happens that $p<k$, one simply requires
\begin{equation*}\label{ADTBraceII}
A\lb B_1\, , \cdots\, , B_{k} \rb=
\mu_{k}^{(\frakg, \frakl)}(A\, , B_1[1]\, , \cdots\, , B_k[1]):=0.
\end{equation*}
\begin{proposition}[\cite{Calaque-QuantizationFormal}]
The pair $(\ADT\, , \{\mu_k^{(\frakg, \frakl)}\}_{k\geqslant 0})$
constitutes a brace algebra. 	
\end{proposition}

Particularly, we have the element $\m=1\otimes 1\otimes 1 \in \ADT^2$ which satisfies
$\m\lb \m \rb=0$.
The associated differential
\begin{equation*}
\delta_{(\frakg, \frakl)}: \ADT \rightarrow \ADT[1]	
\end{equation*}
 is explicitly formulated:
\begin{equation*}\label{ADTHochschild}
\begin{split}
\delta_{(\frakg, \frakl)}(A):=[\m\, , A]_{\mathrm{G}} & =
\m\lb A \rb - (-1)^{|A|+1} A\lb \m \rb \\
& \quad =
(-1)^{n-1}A_{2,3,\cdots, n+2} +\sum_{i=1}^{n+1} (-1)^{n-1+i} A_{1, \cdots,  i \tinywedge i+1 , \cdots, n+2}\, ,
\end{split}	
\end{equation*}
for any $A\in \ADT^n$. Similarly, the multiplication
\begin{equation*}
\cup_{(\frakg, \frakl)}: \ADT \times \ADT \rightarrow \ADT	
\end{equation*}
associated with $\m$ is given by
\begin{equation*}
A \cup_{(\frakg, \frakl)} B:=(-1)^{|A|}\m \lb A\, , B \rb=
(-1)^{nm}
A_{1, \cdots, n,  n+1 \tinywedge \cdots \tinywedge  n+m+1 }\boldsymbol{\cdot}
B_{n+1, \cdots, n+m+1}\, ,
\end{equation*}
for any $A\in \ADT^n$, $B\in \ADT^m$.

By Proposition \ref{BraceToBraceBInfinity}, we obtain a brace
$B_{\infty}$ algebra
\begin{equation*}
\bigl(
\oplus_{n=0}^{\infty}
(U(\frakg)^{\otimes n}\otimes U(\frakl)\llbracket \hbar \rrbracket)^{\frakl}[-n]\, ,
\delta_{(\frakg, \frakl)}\, , \cup_{(\frakg, \frakl)}\, , \{\mu_k^{(\frakg, \frakl)}\}_{k\geqslant 0}
\bigr)
\end{equation*}	
which is still denoted by $\ADT$.

\subsection{The operad $\cP_{(\frakg, \frakl)}$}\label{CalaqueADT}

We introduce a non-symmetric operad $\cP_{(\frakg, \frakl)}$ to reconstruct
the brace $B_{\infty}$ algebra  $\ADT$
via the Gerstenhaber-Voronov operadic modeling.  We follow the format of Definition \ref{PartialDefinitionOperad} to define $\cP_{(\frakg, \frakl)}$.  First, for each integer $n\geqslant 0$, define
\begin{equation*}
\cP_{(\frakg,\frakl)}(n):=(U(\frakg)^{\otimes n}\otimes U(\frakl)\llbracket \hbar \rrbracket)^{\frakl}, 	
\end{equation*}
which consists of $\frakl$-invariant elements in  $  U(\frakg)^{\otimes n}\otimes U(\frakl)\llbracket \hbar \rrbracket$, i.e., elements $A$ satisfying  $\Delta^n (\frakl) \cdot A=A \cdot \Delta^n(\frakl)$ for all $\frakl\in \frakl$.
Second, define the partial composition maps
\begin{equation}\label{PartialCompositionADT}
\begin{split}
\cdot \circ_{i} \cdot :
\cP_{(\frakg,\frakl)}(n) \otimes \cP_{(\frakg,\frakl)}(m)	
& \rightarrow \cP_{(\frakg,\frakl)}(n+m-1)\, , \\
A \circ_i B : & =
A_{1, \cdots, i \tinywedge \cdots \tinywedge i+m-1, i+m, \cdots, n+m}
\cdot B_{i,\cdots, i+m-1, i+m \tinywedge \cdots \tinywedge n+m}
\, , 	
\end{split}
\end{equation}
for any integers $n\geqslant 0$, $m\geqslant 0$,
$1\leqslant i \leqslant n$, and
$A\in \cP(n)$, $B\in \cP(m)$.
In the above formula we have adopted Convention \eqref{ImportantConvention}.
Third, we set the identity element to be
$1\otimes 1 \in \cP_{(\frakg,\frakl)}(1)=(U(\frakg)\otimes U(\frakl)\llbracket \hbar\rrbracket)^{\frakl}$.
These data of $\cP_{(\frakg,\frakl)}$ together form
a non-symmetric operad. In fact,
For any $B\in \cP_{(\frakg,\frakl)}(m)$, $C\in \cP_{(\frakg,\frakl)}(l)$, and
 positive integers $i\, , j\, , r$
that satisfy $i+m \leqslant j$, $j+l\leqslant r$,
since $B$ and $C$ are $\frakl$-invariant,
 $B_{i,\cdots, i+m-1, i+m \tinywedge \cdots \tinywedge r}$
and $C_{j, \cdots, j+l-1, j+l \tinywedge \cdots \tinywedge r}$ commute.
Based on this fact,
it is direct to verify that
the sequential composition axiom
\eqref{SequentialComposition},
the parallel composition axiom \eqref{ParallelComposition},
and the unitality property \eqref{Unitality2} are all satisfied.

The following fact follows directly from Theorem \ref{Operadtobrace}
and the construction of the brace $B_{\infty}$ algebra
$\ADT$.

\begin{proposition}
 The element $\m=1\otimes 1\otimes 1\in (U(\frakg)^{\otimes 2}\otimes U(\frakl)\llbracket \hbar \rrbracket)^{\frakl}$ is a multiplication on the operad $\cP_{(\frakg,\frakl)}$.
The brace $B_{\infty}$ algebra  associated with
$\cP_{(\frakg,\frakl)}$ and $\m$ via the Gerstenhaber-Voronov operadic modeling coincides with the brace $B_{\infty}$ algebra  $\ADT$.
 \end{proposition}

\subsection{The quantum groupoid $\mathscr{H}_{(\frakg, \frakl)}$ and operad $\cP_{\mathscr{H}_{(\frakg, \frakl)}}$}

Fix a vector basis $\{\l_{i}\}$ of $\frakl$, and let $\{ \lambda_{i} \}$ denote the corresponding coordinate functions on $\frakl^{\ast}$. The standard PBW map is then defined as
\begin{equation*}
\begin{split}
\pbw: S(\frakl) & \rightarrow U(\frakl)\, ,  \\
\lambda_{i_1}\cdots \lambda_{i_k}
& \mapsto
\frac{1}{k!}\sum_{\sigma\in \mathrm{S}_k}
  l_{i_{\sigma(1)}} \cdots   l_{i_{\sigma(k)}} \, ,
\end{split}		
\end{equation*}
where $\mathrm{S}_{k}$ denotes the permutation group of the set $\{1\, , 2\, , \cdots\, , k\}$.
We further define the map
\begin{equation}\label{hbarPBWmap}
\begin{split}
\pbw_{\hbar}: \hat{S}(\frakl)\llbracket \hbar\rrbracket & \hookrightarrow U(\frakl)\llbracket \hbar \rrbracket \, , \\
\lambda_{i_1}\cdots \lambda_{i_k}
& \mapsto
\frac{\hbar^k}{k!}\sum_{\sigma\in \mathrm{S}_k}
  l_{i_{\sigma(1)}} \cdots   l_{i_{\sigma(k)}} \, .
\end{split}	
\end{equation}
Equivalently, for any formal power series $f=f(\lambda)\in \hat{S}(\frakl)$, we have $\pbw_{\hbar}f=\pbw( f(\hbar \lambda) )$.
This map $\pbw_{\hbar}$ is an injective morphism of
$\fK\llbracket \hbar \rrbracket$-modules.
Consequently, we can define the PBW star product $\star_{\PBW}$ on $\hat{S}(\frakl)\llbracket \hbar \rrbracket$ \cite{Xu-Nonabelian}, which is associative and satisfies
\begin{equation}\label{PBWStarProduct}
 \pbw_{\hbar}( f\star_{\PBW} g )=(\pbw_{\hbar}f) \cdot (\pbw_{\hbar}g)\, ,
\end{equation}
for all $f\, , g\in \hat{S}(\frakl)\llbracket \hbar\rrbracket$.

Let $G$ be a Lie group integrating $\frakg$, and $L\subset G$ be a Lie subgroup integrating $\frakl$.
Let $\cD$ denote the algebra of smooth differential operators on $\frakl^{\ast}$,
and $R:=C^{\infty}(\frakl^{\ast})\llbracket \hbar \rrbracket$
equipped with the usual multiplication.
Identifying $U(\frakg)$ (resp. $U(\frakl)$) with the space of left-invariant differential operators on $G$ (resp. $L$), we define
\begin{equation*}
\mathscr{H}_{(\frakg, \frakl)}:=U(\frakg)\otimes
\cD\llbracket \hbar \rrbracket.
\end{equation*}
This space can be identified with the $\fK\llbracket \hbar \rrbracket$-linear and left $G$-invariant differential operators on $G \times \frakl^{\ast}$. Furthermore, since $U(\frakg)$ and $\cD$ possess natural Hopf algebra structures, $\mathscr{H}_{(\frakg, \frakl)}$ naturally becomes a quantum groupoid over $R$.

The PBW star product $\star_{\PBW}$ can be extended from $S(\frakl)\llbracket \hbar \rrbracket$ to $C^{\infty}(\frakl^{\ast})\llbracket \hbar \rrbracket$. Viewing $\cdot \star_{\PBW} \cdot$ as a bi-differential operator, we associate it with a twistor $\Theta_{\PBW}$ of $\mathscr{H}_{(\frakg, \frakl)}$ \cite{Cheng-Chen-Qiao-Xiang-QDYBEI}:
\begin{equation*}
\Theta_{\PBW}\in \cD\otimes_{R} \cD \llbracket \hbar \rrbracket
\subset
\mathscr{H}_{(\frakg, \frakl)}\otimes_{R}\mathscr{H}_{(\frakg, \frakl)}\, ,
\end{equation*}
such that $f\star_{\PBW}g=\Theta_{\PBW}(f\, , g)$ for any $f\, , g\in C^{\infty}(\frakl^{\ast})\llbracket \hbar \rrbracket$.

The Gutt star product on $C^{\infty}(G)\otimes C^{\infty}(\frakl^{\ast})\llbracket \hbar \rrbracket$
is a Weyl ordering star product that
extends the PBW star product \cite{Gutt}.
It arises from a deformation quantization of the symplectic manifold $T^{\ast}L$.
In \cite{Xu-Nonabelian}, Xu obtained a normal ordering version of Gutt's original star product, which we denote by $\star_{\Gutt}$.
More specifically, $\star_{\Gutt}$ is given by the formula:
\begin{equation}\label{GuttStarProduct}
f\star_{\Gutt}g:=\sum_{k=0}^{\infty}\frac{\hbar^k}{ k! } \sum_{i_1, \cdots, i_k}
(\frac{\partial}{\partial \lambda_{i_1}}\cdots \frac{\partial}{\partial\lambda_{i_k}}f)\star_{\PBW}
(\overset{\rightarrow}{l_{i_1}}\cdots \overset{\rightarrow}{l_{i_k}}g)\, , 	 	
\end{equation}
for all $f\, , g\in C^{\infty}(G)\otimes C^{\infty}(\frakl^{\ast})\llbracket \hbar \rrbracket$.
Similarly, regarding $\cdot\star_{\Gutt}\cdot$ as a $\fK\llbracket \hbar\rrbracket$-linear and left $G$-invariant bi-differential operator, we associate it with a twistor $\Theta_{\Gutt}$ of $\mathscr{H}_{(\frakg, \frakl)}$  \cite{Cheng-Chen-Qiao-Xiang-QDYBEI}:
\begin{equation*}
\Theta_{\Gutt}\in \mathscr{H}_{(\frakg, \frakl)}\otimes_{R}\mathscr{H}_{(\frakg, \frakl)}\, ,
\end{equation*}
satisfying $f\star_{\Gutt}g=\Theta_{\Gutt}(f\, , g)$ for any $f\, , g\in C^{\infty}(G)\otimes C^{\infty}(\frakl^{\ast}) \llbracket \hbar \rrbracket$.

For any $f\in \hat{S}(\frakl)\llbracket \hbar \rrbracket$, the operation $f \star_{\PBW}:= f \star_{\PBW} \cdot$ of left PBW multiplication by $f$ can be viewed as
a $\fK\llbracket \hbar\rrbracket$-linear differential operator on $R$, i.e., $f \star_{\PBW}\in \cD\llbracket \hbar \rrbracket$. This satisfies
\begin{equation}\label{PBWProperty}
(f\star_{\PBW}g)\star_{\PBW}=(f\star_{\PBW})\cdot (g\star_{\PBW})\, ,	
\end{equation}
for any $f\, , g\in \hat{S}(\frakl)\llbracket \hbar \rrbracket$. Analogously, we consider $f\star_{\Gutt}:=f\star_{\Gutt}\cdot$ as a $\fK\llbracket \hbar\rrbracket$-linear and left $G$-invariant differential operator on $G\times \frakl^{\ast}$, implying $f\star_{\Gutt}\in \mathscr{H}_{(\frakg, \frakl)}$.

Finally, given that $\mathscr{H}_{(\frakg, \frakl)}$ is a quantum groupoid, we can consider its associated non-symmetric operad $\cP_{\mathscr{H}_{(\frakg, \frakl)}}$ (as detailed in Section \ref{HopfOperad}).

\subsection{The morphism of operads
$\mathsf{c}:\cP_{(\frakg, \frakl)}\rightarrow
\cP_{\mathscr{H}_{(\frakg, \frakl)}}\Laurent{\hbar}$}
 \label{Sec:morphismofoperadsc}
By $\fK\Laurent{\hbar}$-linearly
extending the operad structure of
$\cP_{\mathscr{H}_{(\frakg, \frakl)}}$
to the family of vector spaces
$\{ \cP_{\mathscr{H}_{(\frakg, \frakl)}}(n)\Laurent{\hbar}:=\cP_{\mathscr{H}_{(\frakg, \frakl)}}(n)\otimes_{\fK\llbracket \hbar \rrbracket}\fK\Laurent{\hbar}\}_{n\geqslant 0}$,
we obtain a new operad denoted by
$\cP_{\mathscr{H}_{(\frakg, \frakl)}}\Laurent{\hbar}$.
Consequently, the brace $B_{\infty}$ algebra associated with $\cP_{\mathscr{H}_{(\frakg, \frakl)}}\Laurent{\hbar}$, denoted by $B_{\infty}(\mathscr{H}_{(\frakg, \frakl)})\Laurent{\hbar}$, has underlying graded spaces and algebraic structures that are $\fK\Laurent{\hbar}$-linear extensions of those of
$B_{\infty}(\mathscr{H}_{(\frakg, \frakl)})$.

Consider the injective $\fK\llbracket \hbar\rrbracket$-linear morphism of Hopf algebras:
\begin{equation}\label{varphiMAP}
\begin{split}
\varphi: U(\frakl)\llbracket \hbar \rrbracket
& \rightarrow \cD\otimes \fK\Laurent{\hbar}\, , 	 \\
\varphi(l_i) & =\frac{1}{\hbar}\lambda_{i}\star_{\PBW}\, .
\end{split}
\end{equation}

For each integer $n\geqslant 0$, we define a $\fK$-linear map
\begin{equation}\label{CnMap}
\begin{split}
& \mathsf{c}_n: \cP_{(\frakg,\frakl)}(n) \rightarrow
\cP_{\mathscr{H}_{(\frakg, \frakl)}}(n)\Laurent{\hbar}\, ,
\\
\mathsf{c}_n(K) :=
(K_1\otimes \cdots \otimes K_n)
\cdot
& \Delta^{n-1}( \varphi(K_{n+1}))
\cdot \\
& \bigl(
(\Theta_{\Gutt})_{1 \tinywedge 2 \tinywedge \cdots \tinywedge n-1, n}
\boldsymbol{\cdot}
(\Theta_{\Gutt})_{1 \tinywedge 2 \tinywedge \cdots \tinywedge n-2, n-1}
\boldsymbol{\cdot}
\cdots
\boldsymbol{\cdot}
(\Theta_{\Gutt})_{1,2}
\bigr)\, ,
\end{split}
\end{equation}
for any $K=K_1\otimes \cdots \otimes K_n \otimes K_{n+1}\in \cP_{(\frakg,\frakl)}(n)$. Note that when $n=0$, $\Delta^{-1}$ reduces to the counit map $\epsilon: \cD=U(T\frakl^{\ast})\rightarrow C^{\infty}(\frakl^{\ast})$. In particular, the map
\begin{equation*}
\mathsf{c}_0: \cP_{(\frakg,\frakl)}(0)=
(U(\frakl)\llbracket \hbar \rrbracket)^{\frakl}
\rightarrow \cP_{\mathscr{H}_{(\frakg, \frakl)}}(0)\Laurent{\hbar}=C^{\infty}(\frakl^{\ast})\otimes\fK\Laurent{\hbar}	
\end{equation*}
is given by $\mathsf{c}_0(K)=\epsilon(\varphi(K))$ for any $K\in \cP_{(\frakg,\frakl)}(0)$.

Our main result establishes a relationship between the operads $\cP_{(\frakg, \frakl)}$ and $\cP_{\mathscr{H}_{(\frakg, \frakl)}}\Laurent{\hbar}$:
\begin{theorem}\label{MainEmbedding}
The family of $\fK$-linear maps $\{\mathsf{c}_n\}_{n\geqslant 0}$ constitutes a morphism of operads
\begin{equation*}
\cP_{(\frakg,\frakl)}\rightarrow \cP_{\mathscr{H}_{(\frakg, \frakl)}}\Laurent{\hbar}\, .	
\end{equation*}	
This morphism of operads
induces an injective strict $B_{\infty}$ morphism between the brace $B_{\infty}$ algebras:
\begin{equation*}
\mathsf{c}=\oplus_{n=0}^{\infty}\mathsf{c}_n
:\ADT \rightarrow B_{\infty}(\mathscr{H}_{(\frakg, \frakl)})\Laurent{\hbar}\, .	
\end{equation*}
\end{theorem}

Before proceeding with the proof, we first present several key facts of the Gutt twistor
$\Theta_{\Gutt}$ in the subsequent subsection.

\subsection{Facts about the Gutt twistor}

Recall that we have fixed a vector basis
$\{\l_{i}\}$ of $\frakl$ and the corresponded
coordinates $\{ \lambda_{i} \}$ of $\frakl^{\ast}$.

\begin{lemma}\label{LemmaNo1}\

\begin{itemize}
\item[(1)]
For all $f \in \hat{S}(\frakl)\llbracket \hbar \rrbracket$,
we have
\begin{equation}\label{LemmaNo1Equation1}
(\lambda_i\star_{\PBW})\cdot (f\star_{\PBW})- (f\star_{\PBW})\cdot (\lambda_i\star_{\PBW})
= \hbar \bigl( ( \ad^{\ast}_{l_i}f)\star_{\PBW} \bigr) \quad \text{in $\cD\llbracket \hbar \rrbracket$}\, ,
\end{equation}
for all $l_i$ and the corresponded $\lambda_{i}$.
Here $\ad^{\ast}_{l_i}$ denotes the action of
$l_i$ on $\hat{S}(\frakl)\llbracket \hbar \rrbracket$
induced by the coadjoint action of $\frakl$ on
$\frakl^{\ast}$.	

\item[(2)]
The PBW star product is $\frakl$-invariant.	
\end{itemize}
\end{lemma}
\begin{proof}
(1) It follows by the construction of the PBW star product
\eqref{PBWStarProduct}.	

(2)
For each element $l_i$ of the vector basis of $\frakl$ and any $f\, , g\in \hat{S}(\frakl)\llbracket \hbar \rrbracket$,
we have
\begin{align*}
& \bigl(\ad_{l_i}^{\ast}(f\star_{\PBW}g) \bigr)\star_{\PBW} \\
= & \frac{1}{\hbar}[\lambda_{i}\star_{\PBW}\, , (f\star_{\PBW}g)\star_{\PBW}]
\qquad \text{(by Equation \eqref{LemmaNo1Equation1})}	 \\
= &  \frac{1}{\hbar}[\lambda_{i}\star_{\PBW}\, , (f\star_{\PBW})\cdot (g\star_{\PBW})]
\qquad \text{(by Equation \eqref{PBWProperty})} \\
= & \frac{1}{\hbar}[\lambda_{i}\star_{\PBW}\, , f\star_{\PBW}]\cdot (g\star_{\PBW}) +
(f\star_{\PBW})\cdot \frac{1}{\hbar}[\lambda_{i}\star_{\PBW}\, , g\star_{\PBW}] \\
= & \bigl((\ad_{l_i}^{\ast}f)\star_{\PBW} \bigr)\cdot (g\star_{\PBW})+
(f\star_{\PBW})\cdot \bigl((\ad_{l_i}^{\ast}g)\star_{\PBW}\bigr) \qquad \text{(by Equation \eqref{LemmaNo1Equation1})} \\
= & \bigl( (\ad_{l_i}^{\ast}f)\star_{\PBW} g + f \star_{\PBW} (\ad_{l_i}^{\ast}g) \bigr)\star_{\PBW}\, .
\qquad \text{(by Equation \eqref{PBWProperty})}
\end{align*}

Therefore
$\ad_{l_i}^{\ast}(f\star_{\PBW}g)=(\ad_{l_i}^{\ast}f)\star_{\PBW}g+ f \star_{\PBW} (\ad_{l_i}^{\ast}g)$, i.e.,
the PBW star product is $\frakl$-invariant.
\end{proof}

\begin{lemma}\label{LemmaPBW}
For any $f\in \hat{S}(\frakl)\llbracket \hbar\rrbracket$,
we have
\begin{equation}\label{varphihbarPBW}
\varphi(\pbw_{\hbar} f )=f\star_{\PBW}\, .	
\end{equation}
\end{lemma}
\begin{proof}
We only need to prove the statement for
all monomials of the form $\lambda_{i_1}\cdots \lambda_{i_k}$.
By the construction of $\pbw_{\hbar}$ \eqref{hbarPBWmap},
we have
\begin{equation*}\label{MonoidalPBW}
(\lambda_{i_1}\cdots \lambda_{i_k})\star_{\PBW} = \frac{1}{k!}\sum_{\sigma\in \mathrm{S}_{k}}
  (\lambda_{i_{\sigma(1)}}\star_{\PBW})\cdot (\lambda_{i_{\sigma(2)}}\star_{\PBW}) \cdot  \,\, \cdots \,\,
\cdot(\lambda_{i_{\sigma(k)}}\star_{\PBW}),
\end{equation*}
Here $\cdot$ denotes the product in $\cD\llbracket \hbar\rrbracket$.
Then by the definition of $\varphi$ \eqref{varphiMAP}, we have
\begin{equation*}
\begin{split}
\varphi(\pbw_{\hbar}(\lambda_{i_1}\cdots \lambda_{i_k}))
& =
\frac{1}{k!}\sum_{\sigma\in \mathrm{S}_{k}}
\hbar^k\varphi(l_{i_{\sigma_{1}}} \cdots l_{i_{\sigma_{k}}})  \\
& =	\frac{1}{k!}\sum_{\sigma\in \mathrm{S}_{k}}
(\lambda_{i_{\sigma(1)}}\star_{\PBW})\cdot (\lambda_{i_{\sigma(2)}}\star_{\PBW}) \cdot  \,\, \cdots \,\,
\cdot(\lambda_{i_{\sigma(k)}}\star_{\PBW}) \\
& = (\lambda_{i_1}\cdots \lambda_{i_k})\star_{\PBW}\, .
\end{split}
\end{equation*}

Thus the statement is proved.
\end{proof}

\begin{lemma}\label{LemmaNo2}
The morphism of Hopf algebras
$\varphi: U(\frakl)\llbracket \hbar \rrbracket \rightarrow \cD\otimes\fK\Laurent{\hbar}$
satisfies
\begin{equation*}
\varphi ( [l\, , A] )=\mathrm{ad}_{l}^{\ast}\bigl( \varphi(A) \bigr)
\end{equation*}
for all $A\in U(\frakl)\llbracket \hbar \rrbracket$
and $l\in \frakl$.
Here $\ad^{\ast}_{l_i}$ denotes the action of
$l_i$ on $\cD\otimes \fK\Laurent{\hbar}$
induced by the coadjoint action of $\frakl$ on
$\frakl^{\ast}$.	
\end{lemma}
\begin{proof}
The map $\pbw_{\hbar}$ \eqref{hbarPBWmap} satisfies
$\mathrm{Im}(\pbw_{\hbar})\otimes_{\fK\llbracket \hbar \rrbracket}\fK\Laurent{\hbar}= U(\frakl)\otimes_{\fK\llbracket \hbar \rrbracket} \fK\Laurent{\hbar}$, i.e.,
any $A\in U(\frakl)\llbracket \hbar \rrbracket$ can be expressed as a $\fK\Laurent{\hbar}$-linear combination
of elements in $\mathrm{Im}(\pbw_{\hbar})$.
We can assume that $A=\pbw_{\hbar}f=\pbw(f(\hbar \lambda))$
for some $f\in \hat{S}(\frakl)\llbracket \hbar \rrbracket$.
By Equation \eqref{varphihbarPBW},
$\varphi(A)=f(\lambda)\star_{\PBW}$.
Then we have
\begin{align*}
\varphi ( [l\, , A] )=[\varphi(l)\, , \varphi(A)]= &
\frac{1}{\hbar}[\lambda_i \star_{\PBW}\, , \varphi(A)] \\
= & \frac{1}{\hbar}[\lambda_{i}\star_{\PBW}\, , f(\lambda)\star_{\PBW}]	 \\
= & (\ad_{l_i}^{\ast} f(\lambda))\star_{\PBW}
\qquad \text{(by part $(1)$ of Lemma \ref{LemmaNo1})} \\
= & \ad_{l_i}^{\ast} \bigl(f(\lambda)\star_{\PBW} \bigr) \qquad\,\, \text{(by part $(2)$ of Lemma \ref{LemmaNo1})} \\
= & \ad_{l_i}^{\ast} \bigl( \varphi(A) \bigr)\, .
\end{align*}
\end{proof}

\begin{lemma}\label{LemmaNo3}
Given any $A=A_1\otimes \cdots A_n \otimes A_{n+1}\in U(\frakg)^{\otimes n}\otimes U(\frakl)\llbracket \hbar \rrbracket$,
then $A$ is $\frakl$-invariant, i.e.,
$\Delta^n(l)\cdot A = A \cdot \Delta^n(l)$ for all $l\in \frakl$, if and only if
\begin{equation*}
\begin{split}
& \Delta^{n-1}(f\star_{\Gutt})\cdot	
(A_1\otimes \cdots \otimes A_n) \cdot
\Delta^{n-1}(\varphi(A_{n+1})) = \\
& \qquad\qquad\qquad\,\,\, (A_1\otimes \cdots \otimes A_n) \cdot
\Delta^{n-1}(\varphi(A_{n+1}))\cdot
\Delta^{n-1}(f\star_{\Gutt})
\end{split}
\end{equation*}
in $\underbrace{\mathscr{H}_{(\frakg, \frakl)}\otimes_{R} \cdots \otimes_{R} \mathscr{H}_{(\frakg, \frakl)}}_{n}\otimes_{\fK\llbracket \hbar \rrbracket}\fK\Laurent{\hbar}$, for any $f\in S(\frakl)\llbracket \hbar \rrbracket$.
\end{lemma}
\begin{proof}
The differential operation
$f\star_{\Gutt}\in \mathscr{H}_{(\frakg, \frakl)}$
can be written as
\begin{equation*}\label{Expansion}
f\star_{\Gutt} = \sum_{k=0}^{\infty} \frac{1}{k!} \sum_{i_1, \cdots, i_k}\frac{\partial^k f}{\partial \lambda_{i_1} \cdots \partial \lambda_{i_k}} \bigg |_{\lambda_{i_1}=\cdots=\lambda_{i_k}=0} (\lambda_{i_1}\star_{\PBW} + \hbar l_{i_1} )\cdots (\lambda_{i_k}\star_{\PBW} + \hbar l_{i_k})\, .
\end{equation*}	
The verification of the above expression
can be founded in the proof of
\cite[Lemma $3.4$]{Cheng-Chen-Qiao-Xiang-QDYBEI}.
Therefore, we only need to verify
that $A$ is $\frakl$-invariant if and only if
\begin{equation*}
\begin{split}
& \Delta^{n-1}(\lambda_{i}\star_{\PBW}+\hbar l_i)\cdot	
(A_1\otimes \cdots \otimes A_n) \cdot
\Delta^{n-1}(\varphi(A_{n+1}))= \\
&\qquad\qquad\qquad\qquad\quad\,\,\,(A_1\otimes \cdots \otimes A_n) \cdot
\Delta^{n-1}(\varphi(A_{n+1}))\cdot
\Delta^{n-1}(\lambda_{i}\star_{\PBW}+\hbar l_i)\, ,
\end{split}
\end{equation*}
for all $l_i$ and the corresponded $\lambda_{i}$.
Since $\lambda_{i}\star_{\PBW}$ belongs to $\cD\llbracket \hbar \rrbracket \subset \mathscr{H}_{(\frakg,\frakl)}$,
it commutes with elements in $U(\frakg) \subset \mathscr{H}_{(\frakg,\frakl)}$.
Since $l_i$ belongs to $U(\frakl)\subset U(\frakg)$, it commutes with elements in
$\cD\otimes \fK\Laurent{\hbar}$.
Therefore we have
\begin{equation*}
\begin{split}
& \Delta^{n-1}(\lambda_{i}\star_{\PBW}+\hbar l_i)\cdot	
(A_1\otimes \cdots \otimes A_n) \cdot
\Delta^{n-1}(\varphi(A_{n+1})) \, - \\
& \qquad\qquad(A_1\otimes \cdots \otimes A_n) \cdot
\Delta^{n-1}(\varphi(A_{n+1}))\cdot
\Delta^{n-1}(\lambda_{i}\star_{\PBW}+\hbar l_i) \\
= & \hbar \biggl( [\Delta^{n-1}(l_i)\, , A_1\otimes \cdots \otimes A_n]\cdot \Delta^{n-1}(\varphi(A_{n+1})) +
(A_1\otimes \cdots \otimes A_n) \cdot \Delta^{n-1}(\varphi([l_i\, , A_{n+1}])) \biggr)\, .\\
\end{split}
\end{equation*}
Since $\varphi$ is injective,
the above expression vanishes if and only if
\begin{equation*}
\Delta^{n}(l_i)\cdot A - A \cdot \Delta^{n}(l_i)=
[\Delta^{n-1}(l_i)\, , A_1\otimes \cdots \otimes A_n]\otimes A_{n+1}+
A_1\otimes \cdots \otimes A_n \otimes [l_i\, , A_{n+1}]
=0\, , 	
\end{equation*}	
for all $l_i$.

Thus we complete the proof.	
\end{proof}

\begin{lemma}\label{LemmaNo4}
For any integers $i\geqslant 2$ and $m\geqslant 0$,
any $B=B_1\otimes \cdots \otimes B_m\otimes B_{m+1}\in U(\frakg)^{\otimes m}\otimes U(\frakl)\llbracket \hbar \rrbracket$, we have the
following equality holds in $\underbrace{\mathscr{H}_{(\frakg, \frakl)}\otimes_{R} \cdots \otimes_{R} \mathscr{H}_{(\frakg, \frakl)}}_{m} \otimes_{\fK\llbracket \hbar \rrbracket}\fK\Laurent{\hbar}$:
\begin{equation*}
\begin{split}
& (\Theta_{\Gutt})_{1\tinywedge\cdots \tinywedge i-1, i \tinywedge\cdots \tinywedge i+m-1}\cdot
\biggl( \underbrace{1\otimes \cdots \otimes 1}_{i-1} \otimes \bigl( (B_1\otimes \cdots \otimes B_m) \cdot
\Delta^{m-1}(\varphi(B_{m+1})) \bigr) \biggr) = \\	
& \qquad\qquad\qquad
( \underbrace{1\otimes \cdots \otimes 1}_{i-1}\otimes B_1\otimes \cdots \otimes B_m)\cdot \Delta^{i+m-2}(\varphi(B_{m+1}))\cdot
(\Theta_{\Gutt})_{1 \tinywedge\cdots \tinywedge i-1, i \tinywedge \cdots \tinywedge i+m-1}\, .
\end{split}
\end{equation*}	
\end{lemma}
\begin{proof}
We treat both sides of the above formula as
$\fK\llbracket \hbar\rrbracket$-linear polydifferential operators from
$\bigotimes_{\fK\llbracket \hbar \rrbracket}^{i+m-1}
C^{\infty}(G)\otimes C^{\infty}(\frakl^{\ast})\llbracket \hbar \rrbracket$
to $C^{\infty}(G)\otimes C^{\infty}(\frakl^{\ast})\llbracket \hbar \rrbracket$.
In order to verify that
the polydifferential operators on both sides are the same,
we only need to check that
their actions on
$\bigotimes_{\fK\llbracket \hbar \rrbracket}^{i+m-1}
C^{\infty}(G)\otimes S(\frakl)\llbracket \hbar \rrbracket$ coincide.
Given any $f_1\, , \cdots\, , f_{i+m-1}\in C^{\infty}(G)\otimes S(\frakl)\llbracket \hbar \rrbracket$, we have
\begin{align*}
& \biggl( \!
(\Theta_{\Gutt})_{1\tinywedge \cdots \tinywedge i-1, i \tinywedge \cdots \tinywedge i+m-1}
\! \cdot \!
\bigl( \underbrace{1\otimes \cdots \otimes 1}_{i-1} \otimes \bigl( (B_1\otimes \cdots \otimes B_m)
\!\cdot \!
\Delta^{m-1}(\varphi(B_{m+1})) \bigr) \bigr)
\! \biggr)
(f_1 , \cdots , f_{i+m-1})
 \\
& =
\Delta^{m-1}\bigl( (f_1\cdots f_{i-1})\star_{\Gutt} \bigr)
\biggl(
(B_1\otimes \cdots \otimes B_m)\cdot \Delta^{m-1}(\varphi(B_{m+1}))(f_{i}\, ,\cdots\, , f_{i+m-1})
\biggr) \\
& =
\biggl(
(B_1\otimes \cdots \otimes B_m)\cdot \Delta^{m-1}(\varphi(B_{m+1}))\cdot
\Delta^{m-1}\bigl( (f_1\cdots f_{i-1})\star_{\Gutt}
\bigr)
\biggr)(f_{i}\, , \cdots\, , f_{i+m-1})
\\
&  \qquad\qquad \text{(by Lemma\eqref{LemmaNo3})}\\
& =\!\!
\biggl(\!\!
(\underbrace{1\otimes \cdots \otimes 1}_{i-1}\otimes B_1\otimes \cdots \otimes B_m)
\!\cdot\!
\Delta^{i+m-2}(\varphi(B_{m+1}))
\!\cdot\!
(\Theta_{\Gutt})_{1 \tinywedge\cdots \tinywedge i-1, i \tinywedge \cdots \tinywedge i+m-1}
\!\! \biggr)
(f_1\, , \cdots\, , f_{i+m-1}).
\end{align*}

Thus the statement is proved.	
\end{proof}

\begin{lemma}\label{LemmaNo5}
For any integer $k\geqslant 2$ and any $B=B_1\otimes \cdots \otimes B_{k-1}\otimes B_{k}\in U(\frakg)^{\otimes (k-1)}\otimes U(\frakl)\llbracket \hbar \rrbracket$, the following equality holds in $\underbrace{\mathscr{H}_{(\frakg, \frakl)}\otimes_{R} \cdots \otimes_{R} \mathscr{H}_{(\frakg, \frakl)}}_{k-1} \otimes_{\fK\llbracket \hbar\rrbracket}\fK\Laurent{\hbar}$:
\begin{equation}\label{Hard}
\begin{split}
& (\Theta_{\Gutt})_{1\tinywedge \cdots \tinywedge k-1, k}\cdot
\bigl(
((B_1\otimes \cdots \otimes B_{k-1})\cdot \Delta^{k-2}(\varphi(B_{k}))) \otimes 1 \bigr) =  \\
& \qquad\qquad(B_1\otimes \cdots \otimes B_{k-1}\otimes
(B_{k})_{1})
\cdot \Delta^{k-1}( \varphi((B_{k})_{2}) ) \cdot
(\Theta_{\Gutt})_{1 \tinywedge \cdots \tinywedge k-1, k}.
\end{split}
\end{equation}	
Here we use Sweedler's notation $\Delta(B_{k})=(B_{k})_1 \otimes (B_{k})_2$ to denote the coproduct of $B_{k}$.
\end{lemma}
\begin{proof}
Suppose that
$\Theta_{\PBW}=(\Theta_{\PBW})_1\otimes_{R}(\Theta_{\PBW})_2\in \cD\llbracket \hbar \rrbracket\otimes_{R} \cD \llbracket \hbar \rrbracket$.
Then the Gutt star product has the form
\begin{align*}
f\star_{\mathrm{Gutt}}g & =  \Theta_{\mathrm{Gutt}}(f\, , g) \\
& = \sum_k \sum_{i_1, \cdots, i_k} \frac{\hbar^{k}}{k!}
\biggl(
(\Theta_{\PBW})_1\frac{\partial^k}{\partial \lambda_{i_1}\cdots \partial\lambda_{i_k}}
\biggr)\otimes_{R}\biggl(
(\Theta_{\PBW})_2 \overset{\rightarrow}{l_{i_1}}\cdots
\overset{\rightarrow}{l_{i_k}} \biggr) (f\, , g)\, ,
\end{align*}
for all $f\, , g\in C^{\infty}(G)\otimes C^{\infty}(\frakl^{\ast})\llbracket \hbar\rrbracket$.
Thus the Gutt twistor has the form
\begin{align*}
\Theta_{\Gutt}
=(\Theta_{\Gutt})_1\otimes_{R} (\Theta_{\Gutt})_2\, ,
\,\,\,
\text{where $(\Theta_{\Gutt})_1\in \cD$ and
$(\Theta_{\Gutt})_2\in \mathscr{H}_{(\frakg, \frakl)}$}\, .
\end{align*}
Moreover, we have
$(\Theta_{\Gutt})_{1\tinywedge \cdots \tinywedge k-1, k}= \Delta^{k-2}(\Theta_{\Gutt})_1
\otimes_{R}(\Theta_{\Gutt})_2$.

Since $B_1\otimes \cdots \otimes B_{k-1}\in U(\frakg)^{\otimes (k-1)}$ and
$\Delta^{k-2}((\Theta_{\Gutt})_1)\in
\underbrace{\cD\llbracket \hbar \rrbracket \otimes_{R} \cdots \otimes_{R} \cD\llbracket \hbar \rrbracket}_{k-1}$ commutes in $\underbrace{\mathscr{H}_{(\frakg, \frakl)}\otimes_{R} \cdots \otimes_{R} \mathscr{H}_{(\frakg, \frakl)}}_{k-1}$,
Equation \eqref{Hard} is equivalent to
\begin{equation*}
(\Theta_{\Gutt})_{1\tinywedge \cdots \tinywedge k-1,k}\cdot
(\Delta^{k-2}(\varphi(B_k))\otimes 1)=
(\underbrace{1\otimes \cdots \otimes 1}_{k-1}
\otimes (B_{k})_1)\cdot
\Delta^{k-1}(\varphi((B_{k})_2)) \cdot
(\Theta_{\Gutt})_{1\tinywedge \cdots \tinywedge k-1,k}\, .
\end{equation*}

Observe that the above formula holds for any
$B_{k}\in U(\frakl)\llbracket \hbar \rrbracket$
if and only if it holds for $B_{k}=l_i$ where
$l_i$ goes through the set of the given vector basis of
$\frakl$, and note that $\Delta(l_i)=l_i\otimes 1 + 1\otimes l_i$, $\varphi(l_i)=\frac{1}{\hbar}\lambda_{i}\star_{\PBW}$,
we find that Equation \eqref{Hard} is equivalent to
\begin{equation}\label{Hard2}
\begin{split}
(\Theta_{\Gutt})_{1\tinywedge \cdots \tinywedge k-1,k}\cdot
(\Delta^{k-2}(\lambda_{i}\star_{\PBW})\otimes 1) & =
(\underbrace{1\otimes \cdots \otimes 1}_{k-1}
\otimes \hbar l_i)\cdot
(\Theta_{\Gutt})_{1\tinywedge \cdots \tinywedge k-1,k} \,\,  + \\
& \qquad\qquad\qquad
\Delta^{k-1}(\lambda_{i}\star_{\PBW}) \cdot
(\Theta_{\Gutt})_{1\tinywedge \cdots \tinywedge k-1,k}\, .
\end{split}
\end{equation}

We further notice that if the formula
\begin{equation}\label{Hard3}
\Theta_{\Gutt}\cdot	(\lambda_{i}\star_{\PBW}\otimes 1)=
(1\otimes \hbar l_i)\cdot \Theta_{\Gutt}+
\Delta(\lambda_{i}\star_{\PBW})\cdot\Theta_{\Gutt}
\end{equation}
holds, then applying $\Delta^{k-2}\otimes \id$ to the both sides, we can verify Equation \eqref{Hard2}.

Thus, we have seen that Equation \eqref{Hard} is equivalent to Equation \eqref{Hard3}.
Now we verify Equation \eqref{Hard3}.
Regard both sides of Equation \eqref{Hard3} as bi-differential operators which acts on
$C^{\infty}(\frakl^{\ast})\llbracket \hbar\rrbracket \otimes_{\fK\llbracket \hbar\rrbracket}
C^{\infty}(G \times \frakl^{\ast})\llbracket \hbar\rrbracket$,
we only need to check their actions on
$S(\frakl)\llbracket \hbar \rrbracket \otimes_{\fK\llbracket \hbar\rrbracket}
C^{\infty}(G \times \frakl^{\ast})\llbracket \hbar \rrbracket$
coincide.

Given any $f\in S(\frakl)\llbracket \hbar \rrbracket$
and $g\in C^{\infty}(G \times \frakl^{\ast})\llbracket \hbar \rrbracket$, we have
\begin{align*}
\bigl(\Theta_{\Gutt}\cdot (\lambda_{i}\star_{\PBW}\otimes 1) \bigr)(f\, , g) & =
(\lambda_{i}\star_{\PBW}f)\star_{\Gutt}g \\
& =	(\lambda_{i}\star_{\Gutt}f)\star_{\Gutt}g
\qquad \text{(by Equation \eqref{GuttStarProduct})} \\
& = \lambda_{i}\star_{\Gutt}(f\star_{\Gutt}g) \\
& = \lambda_{i}\star_{\PBW}(f\star_{\Gutt}g)+
\sum_{j}\frac{\partial \lambda_{i}}{\partial \lambda_{j}}
\hbar \overset{\rightarrow}{l_{j}}(f\star_{\Gutt}g) \quad \text{(by Equation \eqref{GuttStarProduct})}\\
& = \bigl(
(1\otimes \hbar l_i)\cdot \Theta_{\Gutt}
+
\Delta(\lambda_{i}\star_{\PBW})\cdot \Theta_{\Gutt}
\bigr)(f\, , g)\, .
\end{align*}

Therefore the statement is proved.
\end{proof}

\subsection{Proof of Theorem \ref{MainEmbedding}}
According to Diagram \eqref{PartialMorphism},
it is sufficient to verify that the following diagrams
\begin{equation*}
\begin{tikzcd}[row sep=4 em, column sep=2.5em]
\cP_{(\frakg,\frakl)}(n)\otimes \cP_{(\frakg,\frakl)}(m) \arrow[d, "\mathsf{c}_n\otimes \mathsf{c}_m"] \arrow[r, "\cdot\circ_{i}\cdot"] & \cP_{(\frakg,\frakl)}(n+m-1) \arrow[d, "\mathsf{c}_{n+m-1}"]	 \\
\cP_{\mathscr{H}_{(\frakg, \frakl)}}(n)\Laurent{\hbar}\otimes \cP_{\mathscr{H}_{(\frakg, \frakl)}}(m)\Laurent{\hbar} \arrow[r, "\cdot\circ_{i}\cdot"] & \cP_{\mathscr{H}_{(\frakg, \frakl)}}(n+m-1)\Laurent{\hbar}
\end{tikzcd}	
\end{equation*}
commute for any integers $n\geqslant 0$, $m\geqslant 0$,
and $1\leqslant i \leqslant n$.	

Given any $A=A_1\otimes \cdots \otimes A_n \otimes A_{n+1}\in \cP_{(\frakg,\frakl)}(n)$,
$B=B_1\otimes \cdots \otimes B_m \otimes B_{m+1}\in \cP_{(\frakg,\frakl)}(m)$,
where $A_{n+1}\, , B_{m+1}\in U(\frakl)\llbracket \hbar \rrbracket$.
By Equations \eqref{PartialComposition}, \eqref{Composition}, \eqref{CnMap}, we have
\begin{equation*}
\begin{split}	
& \mathsf{c}_n(A) \circ_{i} \mathsf{c}_m(B) \\	
= & (A_1\otimes \cdots \otimes \Delta^{m-1} A_i \otimes
\cdots \otimes A_n)\cdot
\Delta^{n+m-2}(\varphi(A_{n+1}))
\cdot \\
&
\biggl(
(\Theta_{\Gutt})_{1\tinywedge \cdots \tinywedge n+m-2, n+m-1}
\boldsymbol{\cdot}
\cdots
\boldsymbol{\cdot}
(\Theta_{\Gutt})_{1\tinywedge \cdots \tinywedge i-1+m, i+m}
\boldsymbol{\cdot}
(\Theta_{\Gutt})_{1 \tinywedge \cdots \tinywedge i-1, i \tinywedge \cdots \tinywedge i+m-1}
\boldsymbol{\cdot} \\
& \qquad\qquad\qquad\qquad\qquad\qquad\qquad\qquad\qquad\qquad\qquad
(\Theta_{\Gutt})_{1 \tinywedge \cdots \tinywedge i-2, i-1}
\boldsymbol{\cdot}
\cdots
\boldsymbol{\cdot}
(\Theta_{\Gutt})_{1,2}
\biggr) \cdot
\\
&
\small{
\biggl( \underbrace{1\otimes \cdots \otimes 1}_{i-1}
\otimes
\biggl( (B_1\otimes \cdots \otimes B_m)\cdot \Delta^{m-1}(\varphi(B_{m+1}))
\cdot
\bigl(
(\Theta_{\Gutt})_{1 \tinywedge \cdots \tinywedge m-1, m}
\boldsymbol{\cdot}
\cdots
\boldsymbol{\cdot}
(\Theta_{\Gutt})_{1,2}
\bigr )
\biggr )
\otimes \cdots \otimes 1
\biggr)
} \\
= & (A_1\otimes \cdots \otimes \Delta^{m-1} A_i \otimes
\cdots \otimes A_n)\cdot
\Delta^{n+m-2}(\varphi(A_{n+1}))
\cdot \\
& \biggl(
(\Theta_{\Gutt})_{1\tinywedge \cdots \tinywedge n+m-2, n+m-1}
\boldsymbol{\cdot}
\cdots
\boldsymbol{\cdot}
(\Theta_{\Gutt})_{1\tinywedge \cdots \tinywedge i-1+m, i+m}
\boldsymbol{\cdot}
(\Theta_{\Gutt})_{1 \tinywedge \cdots \tinywedge i-1, i \tinywedge \cdots \tinywedge i+m-1}
\boldsymbol{\cdot} \\
& \qquad\qquad\qquad\qquad\qquad\qquad\qquad\qquad\qquad\qquad\qquad
(\Theta_{\Gutt})_{1 \tinywedge \cdots \tinywedge i-2, i-1}
\boldsymbol{\cdot}
\cdots
\boldsymbol{\cdot}
(\Theta_{\Gutt})_{1,2}
\biggr) \cdot
\\
&
\qquad
\biggl( \underbrace{1\otimes \cdots \otimes 1}_{i-1} \otimes
\biggl( (B_1\otimes \cdots \otimes B_m)\cdot \Delta^{m-1}(\varphi(B_{m+1})) \biggr)
\otimes \cdots \otimes 1
\biggr) \cdot \\
&
\qquad\qquad
\biggl(
\underbrace{1\otimes \cdots \otimes 1}_{i-1} \otimes
\biggl(
(\Theta_{\Gutt})_{1 \tinywedge \cdots \tinywedge m-1, m}
\boldsymbol{\cdot}
\cdots
\boldsymbol{\cdot}
(\Theta_{\Gutt})_{1,2}
\biggr)
\otimes \cdots \otimes 1
\biggr)\, . \\
\end{split}
\end{equation*}

By Lemma \ref{LemmaNo4},
we can exchange
the positions of
$(\underbrace{1\otimes \cdots \otimes 1}_{i-1} \otimes
B_1\otimes \cdots \otimes B_m) \cdot \Delta^{m-1}(\varphi(B_{m+1}))$
and
$(\Theta_{\Gutt})_{1\tinywedge \cdots \tinywedge i-1, i\tinywedge \cdots \tinywedge i+m-1}
\boldsymbol{\cdot}
(\Theta_{\Gutt})_{1 \tinywedge \cdots \tinywedge i-2, i-1}
\boldsymbol{\cdot}
\cdots
\boldsymbol{\cdot}
(\Theta_{\Gutt})_{1,2}$ in the above formula, and obtain
\begin{equation*}
\begin{split}
& \mathsf{c}_{n}(A)\circ_i \mathsf{c}_{m}(B) \\
= &
(A_1\otimes \cdots \otimes \Delta^{m-1} A_i \otimes
\cdots \otimes A_n)\cdot
\Delta^{n+m-2}(\varphi(A_{n+1}))
\cdot \\
& \quad \bigl(
(\Theta_{\Gutt})_{1 \tinywedge \cdots \tinywedge n+m-2, n+m-1}
\boldsymbol{\cdot}
\cdots
\boldsymbol{\cdot}
(\Theta_{\Gutt})_{1 \tinywedge \cdots \tinywedge i+m-1, i+m}
\bigr) \cdot \\
&
\quad\,\,
\biggl(\biggl(
(\underbrace{1\otimes \cdots \otimes 1}_{i-1} \otimes
B_1\otimes \cdots \otimes B_m) \cdot \Delta^{i+m-2}(\varphi(B_{m+1})) \biggr)
\otimes \cdots \otimes 1
\biggr) \cdot \\
&
\quad\,\,\,\,\,
\bigl(
(\Theta_{\Gutt})_{1\tinywedge \cdots \tinywedge i-1, i\tinywedge \cdots \tinywedge i+m-1}
\boldsymbol{\cdot}
(\Theta_{\Gutt})_{1 \tinywedge \cdots \tinywedge i-2, i-1}
\boldsymbol{\cdot}
\cdots
\boldsymbol{\cdot}
(\Theta_{\Gutt})_{1,2}
\bigr)
\cdot
\\
&
\quad\,\,\,\,\,\,\,\,\,
\biggl(
\underbrace{1\otimes \cdots \otimes 1}_{i-1} \otimes
\biggl(
(\Theta_{\Gutt})_{1\tinywedge \cdots \tinywedge m-1, m}
\boldsymbol{\cdot}
\cdots
\boldsymbol{\cdot}
(\Theta_{\Gutt})_{1,2}
\biggr)
\otimes \cdots \otimes 1
\biggr)\, . \\
\end{split}
\end{equation*}

By using Lemma \ref{LemmaNo5} \,$n-i$\,
times to move
$(\underbrace{1\otimes \cdots \otimes 1}_{i-1} \otimes
B_1\otimes \cdots \otimes B_m) \cdot \Delta^{i+m-2}(\varphi(B_{m+1}))$
to pass through
$(\Theta_{\Gutt})_{1 \tinywedge \cdots \tinywedge n+m-2, n+m-1}
\boldsymbol{\cdot}
\cdots
\boldsymbol{\cdot}
(\Theta_{\Gutt})_{1 \tinywedge \cdots \tinywedge i+m-1, i+m}$
in the above formula, we obtain
\begin{align}
& \mathsf{c}_{n}(A)\circ_{i} \mathsf{c}_{m}(B) \notag \\
= &
\biggl(
A_1 \otimes \cdots \otimes
\bigl((\Delta^{m-1}A_i)\cdot (B_1\otimes \cdots \otimes B_m) \bigr)
\otimes
A_{i+1}\cdot (B_{m+1})_1 \otimes \cdots \otimes A_{n}\cdot (B_{m+1})_{n-i}
\biggr) \cdot \notag\\
&
\Delta^{n+m-2}(\varphi(A_{n+1}\cdot(B_{m+1})_{n-i+1} ))
\cdot  \notag \\
&
\biggl(
(\Theta_{\Gutt})_{1\tinywedge \cdots \tinywedge n+m-2, n+m-1}
\boldsymbol{\cdot}
\cdots
\boldsymbol{\cdot}
(\Theta_{\Gutt})_{1\tinywedge \cdots \tinywedge i-1+m, i+m}
\boldsymbol{\cdot}
(\Theta_{\Gutt})_{1 \tinywedge \cdots \tinywedge i-1, i \tinywedge \cdots \tinywedge i+m-1}
\boldsymbol{\cdot} \notag \\
& \qquad\qquad\qquad\qquad\qquad\qquad\qquad\qquad\qquad\qquad\qquad
(\Theta_{\Gutt})_{1 \tinywedge \cdots \tinywedge i-2, i-1}
\boldsymbol{\cdot}
\cdots
\boldsymbol{\cdot}
(\Theta_{\Gutt})_{1,2}
\biggr) \cdot \notag
\\
&
\biggl(
\underbrace{1\otimes \cdots \otimes 1}_{i-1} \otimes
\biggl(
(\Theta_{\Gutt})_{1\tinywedge \cdots \tinywedge m-1, m}
\boldsymbol{\cdot}
\cdots
\boldsymbol{\cdot}
(\Theta_{\Gutt})_{1,2}
\biggr)
\otimes \cdots \otimes 1
\biggr)\, . \label{EquationFinal}
\end{align}
Here the terms $(B_{m+1})_{1}$, $\cdots$, $(B_{m+1})_{n-i}$, $(B_{m+1})_{n-i+1}$ come from
the coproduct
$\Delta^{n-i}(B_{m+1})=(B_{m+1})_{1}\otimes \cdots \otimes (B_{m+1})_{n-i} \otimes (B_{m+1})_{n-i+1}$.

In the above formula, the polydifferential operator
\begin{equation*}
\begin{split}	
& \biggl(
(\Theta_{\Gutt})_{1\tinywedge \cdots \tinywedge n+m-2, n+m-1}
\boldsymbol{\cdot}
\cdots
\boldsymbol{\cdot}
(\Theta_{\Gutt})_{1\tinywedge \cdots \tinywedge i-1+m, i+m}
\boldsymbol{\cdot}
(\Theta_{\Gutt})_{1 \tinywedge \cdots \tinywedge i-1, i \tinywedge \cdots \tinywedge i+m-1}
\boldsymbol{\cdot} \\
& \qquad\qquad\qquad\qquad\qquad\qquad\qquad\qquad\qquad\qquad\qquad
(\Theta_{\Gutt})_{1 \tinywedge \cdots \tinywedge i-2, i-1}
\boldsymbol{\cdot}
\cdots
\boldsymbol{\cdot}
(\Theta_{\Gutt})_{1,2}
\biggr) \cdot
\\
&
\biggl(
\underbrace{1\otimes \cdots \otimes 1}_{i-1} \otimes
\biggl(
(\Theta_{\Gutt})_{1 \tinywedge \cdots \tinywedge m-1, m}
\boldsymbol{\cdot}
\cdots
\boldsymbol{\cdot}
(\Theta_{\Gutt})_{1,2}
\biggr)
\otimes \cdots \otimes 1
\biggr)	\\
= &
(\underbrace{\id \otimes \cdots \id}_{i-1}
\otimes \Delta^{m-1} \otimes \cdots \otimes \id)
\bigl(
(\Theta_{\Gutt})_{1 \tinywedge \cdots \tinywedge n-1, n}
\boldsymbol{\cdot}
\cdots
\boldsymbol{\cdot}
(\Theta_{\Gutt})_{1,2}
\bigr)\cdot \\
&
\biggl(
\underbrace{1\otimes \cdots \otimes 1}_{i-1} \otimes
\biggl(
(\Theta_{\Gutt})_{1 \tinywedge \cdots \tinywedge m-1, m}
\boldsymbol{\cdot}
\cdots
\boldsymbol{\cdot}
(\Theta_{\Gutt})_{1,2}
\biggr)
\otimes \cdots \otimes 1
\biggr)
\end{split}
\end{equation*}
corresponds to the operation
\begin{equation*}
(
\cdots(y_1\star_{\Gutt}y_2)	\star_{\Gutt} \cdots )\star_{\Gutt} y_{i-1})\star_{\Gutt}
\biggl(
\cdots (y_i\star_{\Gutt}y_{i+1}) \star_{\Gutt}
\cdots )\star_{\Gutt} y_{i+m-1}
\biggr)\star_{\Gutt}  \cdots
)\star_{\Gutt}y_{n+m-1}\, ,
\end{equation*}
for any $y_1\, , \cdots\, , y_{n+m-1}\in C^{\infty}(G)\otimes C^{\infty}(\frakl^{\ast})\llbracket \hbar \rrbracket$.
Since $\star_{\Gutt}$ is an associative product, an alternative form of the above operation is
\begin{equation*}
(\cdots(y_1\star_{\Gutt}y_2)	\star_{\Gutt} \cdots )\star_{\Gutt}y_{i-1})\star_{\Gutt} y_{i}) \star_{\Gutt} y_{i+1}) \star_{\Gutt} \cdots )\star_{\Gutt} y_{n+m-1}\, ,
\end{equation*}
for any $y_1\, , \cdots\, , y_{n+m-1}\in C^{\infty}(G)\otimes C^{\infty}(\frakl^{\ast})\llbracket \hbar \rrbracket$,
which corresponds to the polydifferential operator
\begin{equation*}
(\Theta_{\Gutt})_{1 \tinywedge \cdots \tinywedge n+m-2, n+m-1}
\boldsymbol{\cdot}
\cdots
\boldsymbol{\cdot}
(\Theta_{\Gutt})_{1,2}\, .	
\end{equation*}

Inserting
\begin{equation*}
\begin{split}
& \biggl(
(\Theta_{\Gutt})_{1\tinywedge \cdots \tinywedge n+m-2, n+m-1}
\boldsymbol{\cdot}
\cdots
\boldsymbol{\cdot}
(\Theta_{\Gutt})_{1\tinywedge \cdots \tinywedge i-1+m, i+m}
\boldsymbol{\cdot}
(\Theta_{\Gutt})_{1 \tinywedge \cdots \tinywedge i-1, i \tinywedge \cdots \tinywedge i+m-1}
\boldsymbol{\cdot} \\
& \qquad\qquad\qquad\qquad\qquad\qquad\qquad\qquad\qquad\qquad\qquad
(\Theta_{\Gutt})_{1 \tinywedge \cdots \tinywedge i-2, i-1}
\boldsymbol{\cdot}
\cdots
\boldsymbol{\cdot}
(\Theta_{\Gutt})_{1,2}
\biggr) \cdot
\\
&
\biggl(
\underbrace{1\otimes \cdots \otimes 1}_{i-1} \otimes
\biggl(
(\Theta_{\Gutt})_{1 \tinywedge \cdots \tinywedge m-1, m}
\boldsymbol{\cdot}
\cdots
\boldsymbol{\cdot}
(\Theta_{\Gutt})_{1,2}
\biggr)
\otimes \cdots \otimes 1
\biggr)	\\= & (\Theta_{\Gutt})_{1 \tinywedge \cdots \tinywedge n+m-2, n+m-1}
\boldsymbol{\cdot}
\cdots
\boldsymbol{\cdot}
(\Theta_{\Gutt})_{1,2}
\end{split}
\end{equation*}
to Equation \eqref{EquationFinal},
we finally verify that
\begin{align*}
& \mathsf{c}_{n}(A)\circ_{i} \mathsf{c}_{m}(B) \\
= &
\biggl(
A_1 \otimes \cdots \otimes
\bigl((\Delta^{m-1}A_i)\cdot (B_1\otimes \cdots \otimes B_m) \bigr)
\otimes
A_{i+1}\cdot (B_{m+1})_1 \otimes \cdots \otimes A_{n}\cdot (B_{m+1})_{n-i}
\biggr) \cdot \\
& \,\,\,\,\,\,\,\,
\Delta^{n+m-2}(\varphi(A_{n+1}\cdot(B_{m+1})_{n-i+1} ))
\cdot
\bigl(
(\Theta_{\Gutt})_{1 \tinywedge \cdots \tinywedge  n+m-2, n+m-1}
\boldsymbol{\cdot}
\cdots
\boldsymbol{\cdot}
(\Theta_{\Gutt})_{1,2}
\bigr) \\
= & \mathsf{c}_{n+m-1}(A\circ_{i} B)\, .	
\qquad\quad
\text{(by Equations \eqref{PartialCompositionADT} and \eqref{CnMap})}
\end{align*}

It is clear that the maps $\mathsf{c}_{n}$ are injective, therefore
$\mathsf{c}=\oplus_{n=0}^{\infty}\mathsf{c}_{n}$
is an injective $B_{\infty}$-strict morphism.
Hence the proof of Theorem \ref{MainEmbedding} is finished.

\subsection{Formal dynamical twists and algebraic dynamical twists}

The Lie algebra $\frakl$ acts on the spaces
$U(\frakg)^{\otimes 2}\otimes \hat{S}(\frakl)\llbracket \hbar\rrbracket$ by adjoint action.
The space of $\frakl$-invariant element in
$U(\frakg)^{\otimes 2}\otimes \hat{S}(\frakl)\llbracket \hbar\rrbracket$ is denoted by
$(U(\frakg)^{\otimes 2}\otimes \hat{S}(\frakl)\llbracket \hbar\rrbracket)^{\frakl}$.

\begin{definition}[\cites{Calaque-QuantizationFormal, Enriquez-Etingof1}]
An element
$F\in (U(\frakg)^{\otimes 2}\otimes \hat{S}(\frakl)\llbracket \hbar\rrbracket)^{\frakl}$
is called a \emph{formal dynamical twist}
if it satisfies the \emph{dynamical twist equation (DTE)}
\begin{equation*}
F_{1\tinywedge 2,3}(\lambda)\star_{\PBW} F_{1,2}(\lambda+\hbar h_3)=
F_{1,2 \tinywedge 3}(\lambda)\star_{\PBW} F_{2,3}(\lambda)\, ,	
\end{equation*}
and the condition $F=1\otimes 1\otimes 1 + O(\hbar)$.
Here
\begin{equation*}
F_{1,2}(\lambda+\hbar h_3):=\sum_{k=0}^{\infty}\frac{\hbar^k}{ k! } \sum_{i_1, \cdots, i_k}
(\frac{\partial}{\partial \lambda_{i_1}}\cdots \frac{\partial}{\partial\lambda_{i_k}}F)(\lambda)\otimes (h_{i_1}\cdots h_{i_k})	 \in  U(\frakg)^{\otimes 2} \otimes U(\frakl)
\otimes \hat{S}(\frakl)\llbracket \hbar \rrbracket\, .
\end{equation*}
\end{definition}

\begin{definition}[\cites{Calaque-QuantizationFormal, Enriquez-Etingof1}]\label{Def:algebraicDT}
An element
$K\in (
U(\frakg)^{\otimes 2} \otimes U(\frakl) \llbracket \hbar \rrbracket)^{\frakl}$ is called
an \emph{algebraic dynamical twist} if
satisfies the \emph{algebraic dynamical twist equation (ADTE)}
\begin{equation}\label{ADTE}
K_{1 \tinywedge 2,3,4}\cdot	K_{1,2,3 \tinywedge 4} = K_{1,2 \tinywedge 3,4}\cdot K_{2,3,4}\, ,
\end{equation}
and it has the \textbf{$\hbar$-adic valuation property} \cite{Calaque-QuantizationFormal},
namely
$K=1+\sum_{n\geqslant 1} \hbar^n K_n$ satisfies
$K_n\in U(\frakg)^{\otimes 2}\otimes (U(\frakl))_{\leqslant n}$.
Here $(U(\frakl))_{\leqslant n}:=
\Ker(\id - \eta\circ\epsilon)^{n+1}\circ \Delta^n $,
where
$\eta: \fK\rightarrow U(\frakl)$
and $\epsilon: U(\frakl)\rightarrow \fK$ denote
the unit and counit maps of $U(\frakl)$.
\end{definition}

Observe that the map
\begin{equation}\label{Correspondence}
\begin{split}
& (U(\frakg)^{\otimes 2}\otimes \hat{S}(\frakl)\llbracket \hbar\rrbracket)^{\frakl}
\rightarrow 	
(U(\frakg)^{\otimes 2} \otimes U(\frakl) \bigl \llbracket \hbar \rrbracket)^{\frakl} \\
& \qquad\qquad\qquad\quad\,\,\,
F \mapsto K:=(\id\otimes\id\otimes \pbw_{\hbar})F=(\id\otimes\id\otimes\pbw)F(\hbar \lambda)
\end{split}
\end{equation}
is a bijection
between the set of $F$ which satisfies
$F=1\otimes 1\otimes 1 + O(\hbar)$
and the set of $K$ which has the $\hbar$-adic valuation property.
Furthermore, this correspondence can be strengthened:
\begin{proposition}[\cite{Enriquez-Etingof1}]\label{FormaltoAlgebraic}
The map \eqref{Correspondence}
is a bijection
between
the set of formal dynamical twists and the set of
algebraic dynamical twists.	
\end{proposition}

Given any $K\in (U(\frakg)^{\otimes 2}\otimes U(\frakl)\llbracket \hbar\rrbracket)^{\frakl}$, write it as
$K=K_1\otimes K_2 \otimes K_3$
where
$K_1\, , K_2\in U(\frakg)$ and $K_3\in U(\frakl)\llbracket \hbar \rrbracket$.
By Equation \eqref{CnMap}, in this case we have
\begin{equation*}\label{CK}
\mathsf{c}(K)=(K_1\otimes K_2)\cdot \Delta(\varphi(K_3))
\cdot \Theta_{\Gutt}\in
(\mathscr{H}_{(\frakg, \frakl)}\otimes_{R} \mathscr{H}_{(\frakg, \frakl)})\otimes_{\fK\llbracket \hbar \rrbracket}\fK\Laurent{\hbar}\, .
\end{equation*}

Given any $F\in (U(\frakg)^{\otimes 2}\otimes \hat{S}(\frakl)\llbracket \hbar\rrbracket)^{\frakl}$, write it as
$F=F_1\otimes F_2 \otimes F_3$
where $F_1\, , F_2\in U(\frakg)$ and $F_3\in \hat{S}(\frakl)\llbracket \hbar \rrbracket$.
Define
\begin{equation}\label{CurlyF}
\mathscr{F}:=
(F_1\otimes F_2)\cdot \Delta(F_3 \star_{\PBW}) \cdot \Theta_{\Gutt}\in \mathscr{H}_{(\frakg, \frakl)}\otimes_{R} \mathscr{H}_{(\frakg, \frakl)}\, .
\end{equation}

\begin{lemma}
For any $F\in (U(\frakg)^{\otimes 2}\otimes \hat{S}(\frakl)\llbracket \hbar\rrbracket)^{\frakl}$, we have
\begin{equation}\label{FandK}
\mathsf{c}\bigl(\id\otimes\id\otimes\pbw_{\hbar}(F)\bigr)
=\mathscr{F}	\, .
\end{equation}	
\end{lemma}
\begin{proof}
Let $F=F_1\otimes F_2\otimes F_3$,
where $F_1\, , F_2\in U(\frakg)$ and $F_3\in \hat{S}(\frakl)\llbracket \hbar \rrbracket$.
Then we have
\begin{equation*}
\begin{split}
\mathsf{c}\bigl(\id\otimes \id \otimes \pbw_{\hbar}(F)\bigr)
& =(F_1\otimes F_2)\cdot \Delta( \varphi (\pbw_{\hbar}F_3))
\cdot \Theta_{\Gutt}	 \\
& =(F_1\otimes F_2)\cdot \Delta( F_3\star_{\PBW} )
\cdot \Theta_{\Gutt} \quad\quad \text{(by Equation \eqref{varphihbarPBW})} \\
& = \mathscr{F} \qquad \text{(by Equation \eqref{CurlyF})}\, .
\end{split}
\end{equation*}
\end{proof}

\begin{theorem}\label{FormalTwistANDTwistor}
Suppose that
$K\in (U(\frakg)^{\otimes 2}\otimes U(\frakl)\llbracket \hbar\rrbracket)^{\frakl}$
is an element that has
the $\hbar$-adic valuation property.
Then $K$
is an algebraic dynamical twist if and only if
$\mathsf{c}(K)=(K_1\otimes K_2)\cdot \Delta(\varphi(K_3))
\cdot \Theta_{\Gutt}$
is a twistor of the quantum groupoid
$\mathscr{H}_{(\frakg, \frakl)}$ over $R$.	
\end{theorem}
\begin{proof}
Since $K$ has the $\hbar$-adic valuation property,
there is an $F\in (U(\frakg)^{\otimes 2}\otimes \hat{S}(\frakl)\llbracket \hbar\rrbracket)^{\frakl}$ which satisfies $F=1\otimes 1\otimes 1+ O(\hbar)$, such that
$K=(\id\otimes \id \otimes \pbw_{\hbar})F$.
By Equation \eqref{FandK},
we have $\mathsf{c}(K)=\mathscr{F}\in \mathscr{H}_{(\frakg, \frakl)}\otimes_{R} \mathscr{H}_{(\frakg, \frakl)}$, where $\mathscr{F}$ is defined as in Equation \eqref{CurlyF}.

In our previous work
\cite[Theorem $3.1$]{Cheng-Chen-Qiao-Xiang-QDYBEI},
we proved that any $J=J_1\otimes J_2\otimes J_3\in (U(\frakg)\otimes U(\frakg)\otimes C^{\infty}(\frakl^{\ast})\llbracket \hbar \rrbracket)^{\frakl}$
which satisfies
$J=1\otimes 1\otimes 1 +O(\hbar)$ is a smooth dynamical twist if and only if
$(J_1\otimes J_2)$$\cdot$$\Delta(J_3 \star_{\PBW})$$\cdot$$\Theta_{\Gutt}$$\in$$\mathscr{H}_{(\frakg, \frakl)}\otimes_{R} \mathscr{H}_{(\frakg, \frakl)}$
is a twistor of $\mathscr{H}_{(\frakg, \frakl)}$.
Transferring the proof of \cite[Theorem $3.1$]{Cheng-Chen-Qiao-Xiang-QDYBEI} word by word from
the setting of smooth dynamical twists to the setting of formal dynamical twists,
we find that $F$ is a formal dynamical twist if and only if
$\mathscr{F}=\mathsf{c}(K)$ is a twistor of
$\mathscr{H}_{(\frakg, \frakl)}$.
By Proposition \ref{FormaltoAlgebraic},
$F$ is a formal dynamical twist if and only if $K$ is an algebraic dynamical twist.
Thus $K$ is an  algebraic dynamical twist if and only if
$\mathsf{c}(K)$ is a twistor of $\mathscr{H}_{(\frakg, \frakl)}$.	
\end{proof}

The \textbf{dynamical quantum groupoid} of an algebraic dynamical twist $K$ is
the quantum groupoid $\mathscr{H}_{(\frakg, \frakl),\mathsf{c}(K)}$ obtained from twisting $\mathscr{H}_{(\frakg, \frakl)}$ by $\mathsf{c}(K)$.

\subsection{The twisted strict $B_{\infty}$ morphism
$\mathsf{c}_{K}:\ADT^{K}\rightarrow B_{\infty}(\mathscr{H}_{(\frakg, \frakl), \mathsf{c}(K)})\Laurent{\hbar}$}

Given an algebraic dynamical twist
$K\in (U(\frakg)^{\otimes 2}\otimes U(\frakl)\llbracket \hbar\rrbracket)^{\frakl}$.
Since the ADTE \eqref{ADTE} is equivalent to
the fact that $K$ is a multiplication on the operad
$\cP_{(\frakg,\frakl)}$,
by part $(2)$ of Theorem \ref{Operadtobrace},
one obtains a twisted brace $B_{\infty}$ algebra :
\begin{equation*}
\ADT^K:=
\bigl(
\oplus_{n=0}^{\infty}
(U(\frakg)^{\otimes n}\otimes U(\frakl)\llbracket \hbar \rrbracket)^{\frakl}[-n]\, ,
\delta_{(\frakg, \frakl),K}\, , \cup_{(\frakg, \frakl), K}\, , \{\mu_k^{(\frakg, \frakl)}\}_{k\geqslant 0}
\bigr)\, ,
\end{equation*}	
where
\begin{equation*}
\delta_{(\frakg, \frakl),K}(A):=[K\, , A]_{\mathrm{G}} =
K\lb A \rb - (-1)^{|A|+1} A\lb K \rb\, , 	
\end{equation*}
and
\begin{equation*}
A\cup_{(\frakg, \frakl), K}B:=(-1)^{|A|}K\lb A\, , B \rb\, ,
\end{equation*}
for any homogeneous elements $A\, , B\in \ADT$.
The brace operations on $\ADT^{K}$ and $\ADT$ are the same.

The morphism of operad $\mathsf{c}: \cP_{(\frakg,\frakl)}\rightarrow \cP_{\mathscr{H}_{(\frakg, \frakl)}}\Laurent{\hbar}$
sends $K$ to the twistor $\mathsf{c}(K)$ of the quantum groupoid $\mathscr{H}_{(\frakg, \frakl)}$
(Theorem \ref{FormalTwistANDTwistor}).
The twistor $\mathsf{c}(K)$ can be regarded as a multiplication of the operad $\cP_{\mathscr{H}_{(\frakg, \frakl)}}\Laurent{\hbar}$.
By functoriality of the Gerstenhaber-Voronov operadic modeling of brace $B_{\infty}$ algebras (Proposition \ref{functoriality}),
the map $\mathsf{c}$ in Theorem \ref{MainEmbedding} is also a strict $B_{\infty}$-morphism
between twisted brace $B_{\infty}$ algebras:
\begin{equation*}
\mathsf{c}:
\ADT^{K} \rightarrow
\bigl(B_{\infty}(\mathscr{H}_{(\frakg, \frakl)})\Laurent{\hbar} \bigr)^{\mathsf{c}(K)}\, .	
\end{equation*}
Here $\bigl(B_{\infty}(\mathscr{H}_{(\frakg, \frakl)})\Laurent{\hbar} \bigr)^{\mathsf{c}(K)}$ is the brace $B_{\infty}$ algebra  whose underling graded spaces and algebraic structures are $\fK\Laurent{\hbar}$-linear extensions of those of the type II twisted brace $B_{\infty}$ algebra  $B_{\infty}(\mathscr{H}_{(\frakg, \frakl)})^{\mathsf{c}(K)}$.

Finally, we provide an alternative approach to the
twisted brace $B_{\infty}$ algebra $\ADT^{K}$:

\begin{theorem}\label{TwistedEmbedding}
If $K$ is an algebraic dynamical twist, then there is an injective strict $B_{\infty}$ morphism
\begin{equation*}
\mathsf{c}_{K}: \ADT^{K} \rightarrow
B_{\infty}(\mathscr{H}_{(\frakg, \frakl),\mathsf{c}(K)})\Laurent{\hbar}	
\end{equation*}
from the twisted brace $B_{\infty}$ algebra  $\ADT^K$ to
the $\fK\Laurent{\hbar}$-linear extension of the
canonical brace $B_{\infty}$ algebra  of the dynamical quantum groupoid $\mathscr{H}_{(\frakg, \frakl),\mathsf{c}(K)}$, such that the following diagram commutes:
\begin{equation}\label{DiagramLast}
\begin{tikzcd}
\ADT^{K} \arrow[r, "\mathsf{c}_K"] \arrow[d, "\mathsf{c}"] &
B_{\infty}(\mathscr{H}_{(\frakg, \frakl),\mathsf{c}(K)})\Laurent{\hbar}
\arrow[dl, "\bigl(\mathsf{c}(K)\bigr)^{\sharp}"] \\
\bigl(B_{\infty}(\mathscr{H}_{(\frakg, \frakl)})\Laurent{\hbar}\bigr)^{\mathsf{c}(K)}
\, . &
\end{tikzcd}	
\end{equation}	
\end{theorem}
\begin{proof}
By Theorem \ref{TwoTypes}, we have a strict $B_{\infty}$ isomorphism of brace $B_{\infty}$ algebras
\begin{equation*}
\bigl( \mathsf{c}(K) \bigr)^{\sharp}: 	
B_{\infty}(\mathscr{H}_{(\frakg, \frakl),\mathsf{c}(K)})\Laurent{\hbar}
\xlongrightarrow{\simeq}
\bigl(B_{\infty}(\mathscr{H}_{(\frakg, \frakl)})\Laurent{\hbar}\bigr)^{\mathsf{c}(K)}\, .
\end{equation*}
Take
$\mathsf{c}_{K}:=\biggl( \bigl( \mathsf{c}(K) \bigr)^{\sharp} \biggr)^{-1} \circ \mathsf{c}$,
then it is also an injective strict $B_{\infty}$ morphism and makes Diagram \eqref{DiagramLast} commutes.
\end{proof}

\medskip
{\bf Data Availability} \, There is no data associated with this paper.

\medskip
{\bf Declarations}

\medskip
{\bf Conflict of interest} \, The authors declare no Conflict of interest in this paper.

\end{document}